\documentclass[reqno,a4paper,11pt]{amsart}

\usepackage{amssymb,amsmath,amsthm,amsfonts,color} 

\usepackage{etoolbox}
\usepackage[toc,page]{appendix}

\usepackage{mathrsfs,dsfont, comment,mathscinet}

\usepackage{enumerate,esint,accents}
\usepackage[left=3cm,right=3cm,top=3cm,bottom=3cm]{geometry}

\usepackage[symbol]{footmisc}

\usepackage{mathtools}
\usepackage{multirow}
\usepackage{graphicx}
\usepackage{enumitem}
\usepackage{setspace}

\usepackage[colorlinks=true, pdfstartview=FitV, linkcolor=blue, citecolor=blue, urlcolor=blue]{hyperref}

\allowdisplaybreaks
\numberwithin{equation}{section}

\theoremstyle{plain}
\newtheorem{theorem}{Theorem}[section]
\newtheorem{proposition}[theorem]{Proposition}
\newtheorem{lemma}[theorem]{Lemma}
\newtheorem{corollary}[theorem]{Corollary}

\theoremstyle{definition}
\newtheorem{definition}[theorem]{Definition}
\newtheorem{remark}[theorem]{Remark}

\usepackage{amssymb,amsmath}

\newcommand{\C}{\mathbb{C}}

\newcommand{\R}{\mathbb{R}}
\newcommand{\N}{\mathbb{N}}
\renewcommand{\S}{\mathbb{S}}

\renewcommand{\hat}{\widehat}

\newcommand{\black}{\color{black}}
\newcommand{\blue}{\color{black}}

\newcommand{\ba}{\color{black}}
\newcommand{\bu}{\color{black}}
\newcommand{\red}{\color{black}}

\newcommand{\fA}{\mathcal{A}}
\newcommand{\fm}{\mathtt{v}}

\newcommand{\cX}{\mathcal{X}}
\newcommand{\cH}{\mathcal{H}}

\newcommand{\cS}{\mathcal{S}}
\newcommand{\cF}{\mathcal{F}}
\newcommand{\cG}{\mathcal{G}}
\newcommand{\cW}{\mathcal{W}}

\newcommand{\cV}{\mathcal{V}}

\newcommand{\p}{\partial}
\newcommand{\tr}{{\rm tr}}
\newcommand{\wk}{\rightharpoonup}
\newcommand{\cl}[1]{\overline{#1}}

\newcommand{\strictlyincluded}{\subset\subset}
\newcommand{\res}{\mathop{\hbox{\vrule height 9pt width 0.5pt depth 0pt
\vrule height 0.5pt width 5pt depth 0pt}}\nolimits}

\renewcommand{\tilde}{\widetilde}
\newcommand{\supp}{\mathrm{supp}\,}
\newcommand{\dist}{\mathrm{dist}}
\newcommand{\sdist}{\mathrm{sdist}}

\newcommand{\loc}{\mathrm{loc}}
\newcommand{\Lip}{\mathrm{Lip}}
\newcommand{\Int}[1]{\mathrm{Int}(#1)}

\newcommand{\aplim}{\mathop{\mathrm{ap\,\, lim}}}

\newcommand{\openset}{O}

\newcommand{\str}[1]{e(#1)}

\newcommand{\mtwo}{\mathbb{M}^{2\times 2}_{\rm sym}}
\newcommand{\admissible}{\mathcal{C}}
\newcommand{\substrate}{S}
\newcommand{\Ins}[1]{{\rm Int}{(#1\cup \substrate\cup \Sigma)}}
\newcommand{\indexset}{N}

\title[Existence of minimizers for the SDRI model in 2d]{Existence of minimizers for the SDRI model in 2d: wetting and dewetting regime with mismatch strain}

\author[Sh. Kholmatov]{Shokhrukh Yu.\ Kholmatov} 
\address[Shokhrukh Yu. Kholmatov]{University of Vienna\\ Faculty of Mathematics\\  
Oskar-Morgenstern Platz 1\\1090 Vienna, Austria}
\email[Sh. Kholmatov]{shokhrukh.kholmatov@univie.ac.at}

\author[Paolo Piovano]{Paolo Piovano*} 
\thanks{*Corresponding author.}
\address[Paolo Piovano]{Dipartimento di Matematica, Politecnico di Milano, P.zza Leonardo da Vinci 32, 20133 Milano, Italy\footnote{MUR Excellence Department 2023-2027} \& WPI c/o Research Platform MMM ``Mathematics-Magnetism-Materials'', Fak.\ Mathematik Univ.\ Wien, A1090 Vienna}
\email[P. Piovano]{paolo.piovano@polimi.it}

\subjclass[2010]{49J45, 35R35, 74G65}

\keywords{Minimal configurations, elastic energy, surface energy, existence, regularity, density estimates,  SDRI, interface instabilities, thin films, crystal
cavities, fractures}

\date{\today}
\begin{document}

\begin{abstract}  
The model introduced in  \cite{HP:2019} in the framework of the theory on Stress-Driven Rearrangement Instabilities (SDRI) \cite{AT:1972,Gr:1993} for the morphology of crystalline  materials under stress  is considered. As in \cite{HP:2019} and  in agreement with the models in \cite{Li-etal,S2}, a mismatch strain, rather than a Dirichlet condition as in \cite{CF:2020_arma}, is \bu included into the analysis to represent \ba the lattice mismatch between the crystal and possible adjacent (supporting) materials. The existence of solutions is established in dimension two in the absence of graph-like assumptions and of the restriction to a finite number $m$ of connected components for the free boundary of the region occupied by the crystalline material, thus extending previous results for epitaxially strained thin films and material cavities \cite{BCh:2002,FFLM:2007,FFLM:2011,HP:2019}.  Due to the lack of compactness and lower semicontinuity  for the sequences of $m$-minimizers, i.e., minimizers among configurations with at most $m$ connected boundary components, a minimizing candidate is directly constructed, and then shown to be a minimizer by means of uniform density estimates and the convergence of $m$-minimizers' energies to the energy infimum as $m\to\infty$. Finally,  regularity properties for the morphology  satisfied by every minimizer are  established.
\end{abstract}

\maketitle

\vspace{1cm}

\section{Introduction}

In this paper we establish existence and regularity properties for the solutions of the  variational  model  for Stress-Driven Rearrangement Instabilities (SDRI)  \cite{AT:1972,D:2001,Gr:1993} that was introduced in \cite{HP:2019}. Under the name of \bu SDRI \ba are included all those 
material morphologies, such as boundary irregularities, cracks, filaments, and surface patterns, which a crystalline material may exhibit in the presence of external forces, such as in particular the  chemical bonding forces with adjacent materials. In order to release the induced stresses, atoms rearrange from the material optimal crystalline order and instabilities may develop.

The main  advancement provided by  the results in this manuscript with respect to \cite{HP:2019} is the absence of the unphysical restriction on the number of connected components for the boundary of the region occupied by the crystalline material, by also avoiding graph-like assumptions for such boundaries assumed for the specific settings of epitaxially strained thin films in \cite{BCh:2002,CF:2020_arma,FFLM:2007} and material voids in \cite{FFLM:2011}.  In particular, with respect to \cite{CF:2020_arma} we include  into the analysis  the \emph{dewetting regime}, i.e., the presence of other fixed materials with possibly different boundary surface tensions, even if by only treating the two dimensional case, and we establish regularity results for the crystalline morphologies and instabilities satisfied by every  minimizer. Furthermore, our strategy stems from the approach used in \cite{DMMS:1992} for the \emph{Mumford-Shah} functional, and hence differs from  the method introduced in \cite{CF:2020_arma}, which instead is  based on  allowing displacements  
to attain \emph{a limit value $\infty$} on sets with positive measure (and on technically assigning a zero cost to the elastic-energy contribution related to those sets)\bu. \ba

The SDRI model of \cite{HP:2019} is a variational model introduced in the framework of the SDRI theory initiated in the seminal papers of \cite{AT:1972} and \cite{Gr:1993}, and  on the basis of the subsequent analytical descriptions provided in  the  context of epitaxially strained thin films \cite{BCh:2002,DP:2017, DP:2018_1,FFLM:2007},  crystal cavities \cite{BChS:2007,FFLM:2011}, capillarity droplets \cite{GB-W:2004,DPhM:2015}, fractures  
\cite{BFM:2008,CCI:2019.jmpa,CFI:2016,CFI:2018poincare,FGL:2019}, and boundary debonding and delamination \cite{Babadjian:2016,Baldelli:2014}. All such settings are included and can be treated simultaneously   in the SDRI model \cite{HP:2019} (see Section \ref{examples}). In agreement with  \cite{AT:1972,Gr:1993} since SDRI morphologies  relate  to the boundary of crystalline materials and depend on the bulk rearrangements,  the energy $\cF$ characterizing the SDRI model displays both an {\it elastic bulk energy}  and a {\it surface energy} denoted by $\mathcal{W}$ and $\cS$, respectively.  More precisely, the energy $\cF$ is defined as
 \begin{equation}\label{energy_intro}
 \cF(A,u):=\cS(A,u) +  \mathcal{W}(A,u)
\end{equation}
for any admissible {\it configurational pair}  $(A,u)$ consisting of a set  $A$ that represents the region occupied by the crystalline material in a fixed {\it container} $\Omega\subset\R^d$  for $d\in\N$, i.e., 
 $$A\in \fA:=\{A\subset\overline{\Omega} :\,\,\text{$A$ is $\mathcal{L}^2$-measurable and $\p A$ is $\cH^1$-rectifiable}\},$$  
and  of a displacement function  $u$  of the bulk materials (with respect to the optimal crystal arrangement)  given by 
$$
u\in GSBD^2(\Ins{A};\R^d)\cap H_\loc^1(\Int{A}\cup \substrate;\R^d), 
$$ 
where   $\substrate\subset\R^d\setminus \Omega$  is  the region occupied by  a fixed material, which we denote {\it substrate}  in analogy with the thin-film setting and we consider possibly different from the material in the container, and  
$$\Sigma:=\p \substrate\cap \p \Omega$$
 represents   the  {\it contact surface} between the container $\Omega$ and  the substrate $\substrate$.  In the following we refer to $\mathcal{C}$ as the {\it configurational space} and  to  each configuration $(A,u)\in\mathcal{C}$ as a  {\it  free crystal} with $A$ and $u$ as the {\it free-crystal region} and the {\it free-crystal displacement}, respectively (see Figure \ref{free_crystal}).
\begin{figure}[htp] 
\begin{center}
\includegraphics[scale=0.55]{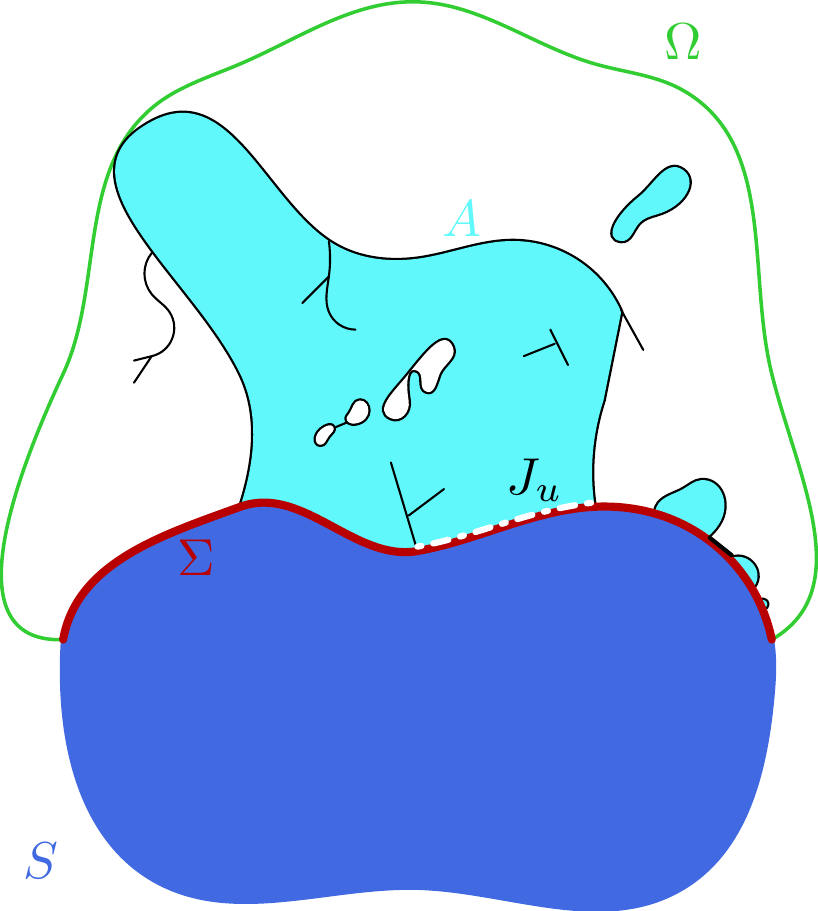} 
\caption{An admissible free-crystal 
region $A$ is displayed in light blue 
in the container $\Omega$, while the substrate $S$ is represented
 in dark blue. The boundary of $A$
(with the cracks) is depicted in black, the container boundary in 
 green, the contact surface
 $\Sigma$ in red (thicker line) while the free-crystal delamination region $J_u$ with a white dashed line. }
\label{free_crystal}
\end{center}
\end{figure}

The bulk elastic energy $\mathcal{W}$ in \eqref{energy_intro} is defined in  \cite{HP:2019} by
$$
\mathcal{W}(A,u)=\int_{A\cup S} W(z,\str{u}-M_0)\,d z,
$$
where the  {\it elastic  density} $W$ is given by 
\begin{equation}\label{intro:energy_density}
W(z,M):=\mathbb{C}(z) M:M
\end{equation}
for any $z\in \Omega\cup S$ and any  $(d\times d)$-symmetric matrix $M\in\mathbb{M}^{d\times d}_{\rm sym}$, and  for a positive-definite   elasticity tensor $\mathbb{C}$, and attains its minimum value zero for  every  $z$ at a fixed strain  $M_0\in M\in\mathbb{M}^{d\times d}_{\rm sym}$ in the following referred to as {\it mismatch strain}. The inclusion  in \eqref{intro:energy_density} of   a mismatch strain $M_0$ defined  by
\begin{equation}\label{intro:mismatch}
 M_0:= 
\begin{cases}
\str{u_0} & \text{in $\Omega ,$}\\
0 &  \text{in $\substrate,$}
\end{cases}
\end{equation}
for a fixed $ u_0\in H^1(\R^d;\R^d)$,  together with the fact that both $M_0$ and $\mathbb{C}$ are let free of jumping across $\Sigma$, allows to model the presence of two different materials in the substrate and in the free crystals, and in particular to take into account the {\it lattice mismatch}  between their optimal crystalline lattices that is crucial, e.g., in the setting of heteroepitaxy \cite{DP:2017,DP:2018_1}.

The surface 
energy $\cS$ in \eqref{energy_intro} is defined as 
  $$
\cS(A,u)= \int_{\partial A} \psi(z,u,\nu)\,d\mathcal{H}^{d-1},
  $$
where the {\it surface tension} $\psi$ is given by
 \begin{equation}\label{surface_tension}
 \psi(z,u,\nu):=\begin{cases}
 \varphi(z,\nu_A(z)) & z\in\Omega\cap\partial^*A,\\
2 \varphi(z,\nu_A(z)) & z\in \Omega\cap (A^{(1)}\cup A^{(0)})\cap\partial A,\\
\varphi(z,\nu_S(z)) + \beta(z) & z\in\Sigma\cap A^{(0)}\cap \partial A,\\
\beta(z) & z\in\Sigma\cap \partial^*A\setminus J_u,\\
 \varphi(z,\nu_S(z)) &  z\in J_u,
\end{cases}
\end{equation}
with   $\varphi\in C(\cl{\Omega}\times\R^d;[0,+\infty))$ being a Finsler norm  such that
$c_1|\xi| \le \varphi(x,\xi) \le c_2|\xi|$ 
for some $c_1,c_2>0$ and representing the
\emph{anisotropy} of the free-crystal material, $\beta$ denoting the \emph{relative adhesion
coefficient}  on $\Sigma$ such that, as for capillarity problems \cite{GB-W:2004,DPhM:2015},
\begin{equation*} 
|\beta(z)|\leq\varphi(z,\nu_S(z))
\end{equation*}
for \bu every \ba $z\in\Sigma$,  $\nu$ coinciding with the exterior 
normal on the reduced boundary $\partial^*A$, 
and $A^{(\delta)}$ denoting the 
set of points of $A$ with density $\delta\in[0,1]$.   

The anisotropic form of $\psi$ in \eqref{surface_tension} distinguishes  various  portions of  the free-crystal topological boundary $\partial A$: the {\it free boundary}  $\partial^*A\cap\Omega$, the  family of {\it internal cracks} $A^{(1)}\cap\Omega\cap\partial A$,    the   family of {\it external filaments}  $A^{(0)}\cap\Omega\cap\partial A$, the {\it delaminated region} $J_u$
\bu, i.e., \ba the portion on the contact surface $\Sigma$ where there is no bonding between the free crystal and the substrate (even if they are adjacent),  the {\it adhesion area} where the free-crystal displacement is continuous  through $\Sigma$, i.e., $\Sigma\cap \partial^*A\setminus J_u$,  and the {\it wetting layer} represented by the filaments on $\Sigma$, i.e., $\Sigma\cap A^{(0)}$. 
In particular, $\psi$ weights the different portions of  $\partial A$ in relation to the active chemical bondings present at each portion, i.e., $\varphi$ when there is no extra chemical bonding, such as at the free profile and  at  the delaminated region, and $\beta$ at the adhesion contact area with the substrate, while  both the cracks and at external filaments are counted $2\varphi$ and the wetting layer sees the contribution of both $\psi$ and $\beta$. 

We consider the case $d=2$ as in \cite{HP:2019},  with the fixed sets $\Omega $ 
and $\substrate$   being  bounded Lipschitz open connected  sets such that   $\Sigma$ is a Lipschitz 1-manifold. For $d\geq3$  results are available for 
the isotropic Griffith model with $L^p$-fidelity term (of the type \eqref{eq:fidelity})  in \cite{CCI:2019.jmpa}  and  with Dirichlet conditions for the displacements at the boundary in \cite{ChC:2019_arxiv}. Moreover, a similar energy as the SDRI energy introduced in  \cite{HP:2019} was subsequently found  in \cite{CF:2020_arma} as a relaxation formula separately for thin films and material voids, for the different setting 
with a Dirichlet condition imposed at $\p\Omega$, and in the wetting regime, i.e., the case where free crystals are expected to cover the substrate. Unfortunately the strategy employed in  \cite{CF:2020_arma} is not implementable in our setting, where rather than prescribing a Dirichlet condition as in  \cite{CF:2020_arma}, the mismatch strain  \eqref{intro:mismatch} (which depends on the substrate region $S$)  is considered in the elastic energy in analogy with the models in \cite{S2} and \cite[Section 4.2.2]{Li-etal} (see also the mathematical treatments \cite{DP:2017,DP:2018_1,FFLM:2007,KP:2021})\bu. \ba 

In fact, the existence results in  \cite{CF:2020_arma}  are achieved by working (in the proofs) with displacements in a larger space than the classical framework of small displacements of linearized elasticity, namely  the space $GSBD_{\infty}^p$ for $p>1$ that includes displacements attaining a limit value $\infty$ 
 in a set of finite perimeter (on which their strain $e(u)$ is defined to be zero \cite[Page 1055]{CF:2020_arma}). Such a method works well with a Dirichlet condition that keeps the displacements anchored, while in our setting it would be always convenient for the displacements in $GSBD_{\infty}^p$ of the minimizing sequences to escape to infinity, as this would result with the definition of the energy in  \cite{CF:2020_arma} in the minimum (zero) value of the elastic energy for the limiting free-crystal region. 
  A treatment for $d\geq3$ of the model under consideration in this paper  with mismatch strain (and without Dirichlet conditions) is under preparation \cite{HP:2022}  by implementing the ideas in this manuscript together with the ones in  \cite{HP:2019}, but without the need of Golab's Theorem (and without employing  the space $GSBD_{\infty}^p$ for the displacements).

Therefore, we must proceed differently here and we rely on the results of \cite{HP:2019} for $d=2$. \bu We begin by observing that, as \ba shown \bu in \cite{HP:2019}, \ba the specific weights of \eqref{surface_tension} are crucial to obtain the lower semicontinuity of the energy $\cF$ \bu  under the constraint on a fixed number $m\in\mathbb{N}$ of boundary connected components for the free-crystal regions\bu, which represented \ba an extension of the more restrictive graph condition assumed in \cite{FFLM:2007}  for the particular setting of epitaxially strained thin films and the starshapedness condition in \cite{FFLM:2011} for material cavities. \bu  More precisely, by considering \ba the subfamily $\admissible_m$ of configurations with \bu free crystals presenting at most \ba $m\in\mathbb{N}$ boundary connected components, namely
$$
\admissible_m:=\Big\{(A,u)\in \admissible:\,\,\text{$\p A$ has  at most $m$ connected components} \Big\}, 
$$ 
\bu in  \cite[Theorem 2.8]{HP:2019} it is shown that \ba
$$
\liminf\limits_{k\to\infty} \cF(A_k,u_k) \ge \cF(A,u)
$$
for every sequence $\{(A_k,u_k)\}\subset\admissible_m$ \bu  converging in a properly chosen topology  $\tau_{\admissible}$ to a configuration \ba $(A,u)\in \admissible_m$.  \bu In particular, the convergence with respect to  $\tau_{\admissible}$  prescribes that  \ba  $\cH^1(\p A_k)$ are equibounded, $\sdist(\cdot,\p A_k)\to \sdist(\cdot,\p A)$ 
locally uniformly in $\R^2$ with {\it sdist} representing the {\it signed distance} function (recall definition at \eqref{signed_distance}), and $u_n\to u$  a.e.\ in $\Int{A}\cup\substrate$. 
\bu We notice that \ba the restriction to the subfamily $\admissible_m$ \bu was needed in \cite{HP:2019}  to establish not only the lower semicontinuity, but also the compactness with respect to $\tau_{\admissible}$, which  indeed \ba fails in $\admissible$  (see Remark \ref{rem:counterexample})\bu, so that by means of the {\it direct method} of the calculus of variations,  the existence of minimizers $(A_m,u_m) \in\admissible_m$ of $\cF$ among all configurations in $\admissible_m$ followed  in  \cite[Theorem 2.6]{HP:2019}. \ba
%

The  aim of the investigation  contained in this paper is to recover   the full generality  avoiding any extra hypothesis on the admissible free-crystal  regions.  
This is achieved by retrieving compactness with respect to the free-crystal regions at least for any sequence of $m$-minimizers $(A_m,u_m)\in\admissible_m$, and  by combining the strategies  of  \cite{DMMS:1992}  and \cite{HP:2019}. 
More precisely, the use in \cite{HP:2019} of the Golab-type Theorem \cite{Gi:2002} is avoided  for  the compactness  of the free-crystal regions by adapting  to our setting  the classical {\it density-estimate} arguments first introduced for surface energies and the Mumford-Shah functional (see, e.g., \cite{AFP:2000,DCL:1989,Ma:2012}), and then extended to the Griffith functional  \cite{ChC:2019_arxiv,CFI:2018poincare}\bu, which in turns allow us also to establish some regularity results. \ba 
\bu Moreover, in our setting there is the extra  difficulty with respect to \cite{DMMS:1992} that the \ba compactness and lower semicontinuity along  sequences of $m$-minimizers (with respect to the topology used to find such $m$-minimizers through the direct method) are both missing. \bu We overcome this issue, by \ba directly constructing a minimizing candidate, proving that it belongs to the class 
$$
\tilde \fA:=\Big\{ A\subset\bu\overline{\Omega}\ba:\, \text{$A$ is $\mathcal{L}^2$-measurable and $\cH^1(\p A)<+\infty$}\Big\},
$$
 and establishing a ``lower-semicontinuity inequality'' (see \eqref{lsc_type} below)  along the selected sequence of $m$-minimizers $(A_m,u_m)$ (see Subsection \ref{subsec:organization}  for more details). 

 Since \bu $\fA\subset \tilde  \fA,$ \ba for proving such lower-semicontinuity property we introduce an auxiliary energy $\tilde \cF$  defined in the larger family $\tilde \admissible$ of configurations $(A,u)$  for which  $A\in\tilde  \fA$, i.e.,
  $$
\tilde \cF(A,u) : = \tilde \cS(A,u) + \cW(A,u),
$$
with auxiliary surface energy $\tilde \cS$  defined as
$$
\tilde \cS(A,u)=\int_{\p A} \tilde\psi(z,u,\nu)d\cH^{d-1},
$$
 where the surface tension $\tilde\psi$ is given by
\begin{equation*} 
 \tilde\psi(z,u,\nu):=
 \begin{cases}
 \varphi(z,\nu_A(z)) & z\in\Omega\cap\partial^*A,\\
2 \varphi(z,\nu_A(z)) & z\in S_u^A,\\
\beta(z) & z\in\Sigma\cap \partial^*A\setminus J_u,\\
 \varphi(z,\nu_S(z)) &  z\in J_u 
\end{cases}
\end{equation*}
for $S_u^A$ denoting the jump set of $u$ along  the $\mathcal{H}^1$-rectifiable portion of the cracks (see \eqref{defini_Su} for the precise definition).

The \bu results \ba of this paper are twofold:   The  existence results contained in Theorem \ref{teo:global_existence} and  the regularity properties of   Theorem \ref{teo:regularity_of_minimizers}. More precisely, in Theorem \ref{teo:global_existence} we prove the existence of a minimum configuration of $\cF$  and $\tilde \cF$ among all configurations in $\admissible$ and $\tilde \admissible,$ respectively, with free-crystal region satisfying a volume constraint, i.e., we solve the minimum problems
\begin{equation}\label{constrainedproblem}
\inf\limits_{(A,u)\in \admissible,\,\,|A| = \fm} \cF(A,u)   
\end{equation}
and 
\begin{equation}\label{constrainedproblemtilde}
\inf\limits_{(A,u)\in \tilde \admissible,\,\,|A| = \fm} \tilde \cF(A,u)   
\end{equation}
for a fixed volume parameter $\fm\in(0,|\Omega|)$  or, if $S=\emptyset$, $\fm=|\Omega|$. Furthermore,  the minimum problems \eqref{constrainedproblem} and \eqref{constrainedproblemtilde} are proven to be equivalent to the \emph{unconstraint minimum problems} consisting in minimizing {\it volume-penalized versions} $\cF^\lambda$ and $\tilde\cF^\lambda$ of the functionals $\cF$ and $\tilde\cF$, for a  {\it penalization constant} $\lambda>0$ provided that $\lambda\geq\lambda_1$ for some uniform constant $\lambda_1>0$.  

In Theorem \ref{teo:regularity_of_minimizers}  regularity properties shared by all solutions  of \eqref{constrainedproblem} and \eqref{constrainedproblemtilde}  are found.  Notice that we cannot directly apply the arguments of \cite{FFLM:2007,FFLM:2011}  based on the \emph{external sphere condition}  considered in \cite{ChL:2003} because of the absence of graph and star-shapedness assumptions on the admissible free-crystal regions. 
As a byproduct of Theorem \ref{teo:global_existence}   and Proposition \ref{prop:min_extend}  
given a configuration $(A,u)$ minimizing \eqref{constrainedproblem} resp.\ \eqref{constrainedproblemtilde}, we can construct  a configuration $(A',u)\in\admissible$ which  minimizes both minimum problems \eqref{constrainedproblem} and \eqref{constrainedproblemtilde} such that $A'$ is an open set with cracks coinciding in $\Omega$ with the jump set of the corresponding  minimizing free-crystal  displacement $u$, and  boundary   $\p A'$  satisfying uniform upper and lower density estimates. Furthermore, we also observe that, any connected component $E$
of $A'$ that \bu does \ba not intersect $\Sigma\setminus J_{u}$ (up to $\cH^1$-negligible sets), must have a sufficiently large area, i.e.,
$$
|E|\ge (c_1 \sqrt{4\pi} /\lambda_1)^2,
$$
and must satisfy $u=u_0$ in $E$ up to  adding a   rigid displacement.

\subsection{Paper organization and detail of the proofs}\label{subsec:organization} 

The paper is organized in 5 sections. In  Section \ref{sec:setting_results} we introduce the mathematical setting, recall the SDRI model from \cite{HP:2019}, and carefully state the main results of the paper. 

In Section \ref{sec:decay_estimates} we prove the upper and lower density estimates for the local decay of  the energy $\cF$ on any sequence of $m$-minimizers $(A_m,u_m) \in\admissible_m$  (see Theorem \ref{teo:density_estimates}) by considering a local version of $\cF^\lambda$ (see \eqref{local_cF}), adapting arguments of \cite{AFP:2000,ChC:2019_arxiv,CFI:2018poincare}
to our setting with   displacements  paired with free-crystal regions, and paying extra care to the fact that $\mathbb{C}$ is possibly not constant (but in $L^\infty(\Omega\cup\substrate)\cap C^{0}(\Omega)$).

In Section \ref{sec:compact_lsc_property}  we prove compactness and lower-semicontinuity properties for a sequence of $m$-minimizers. We  begin by establishing in Proposition \ref{prop:compact_A_m} the compactness  for a sequence of $m$-minimizers $\{(A_m,u_m)\}$ with free-crystal regions $A_m$ not containing isolated points of such free-crystal regions  to a limiting set of finite perimeter $A\subset\Omega$ by means of both  the Blaschke-type selection principle \cite[Proposition 3.1]{HP:2019} and the density estimates established in Section \ref{sec:decay_estimates}. Then, in Proposition \ref{prop:convergence_tangent_line}, we further extend the (already  generalized) Golab-type Theorem \cite[Theorem 4.2]{Gi:2002} to a priori  not-connected  $\mathcal{H}^1$-measurable  (not necessarily $\cH^1$-rectifiable)   sets satisfying uniform density estimates (see \cite{DMMS:1992} for the isotropic case).  The compactness of the displacements in $\{(A_m,u_m)\}$ is then proved in  Propositions \ref{prop:existence_of_u} by carefully constructing the limiting displacement $u$ in view of the property that for every connected component $E_i$ of $A$ the set in which displacements $u_m$ diverge is either the whole component $E_i$ or $\emptyset$, which follows from \cite[Theorem 3.7]{HP:2019}. 
Finally, in Proposition \ref{prop:lsc_surface}  we establish  the lower-semicontinuity property 
\begin{equation}\label{lsc_type}
\liminf\limits_{h\to\infty} \cF(A_{m_h},u_{m_h}) \ge \tilde \cF(A,u),
\end{equation}
 by treating separately the elastic and the surface energy.  For the latter we employ a \emph{blow-up method} differently performed for each portion of the $\p A$ where  $\tilde \psi$  is supported.  In particular extra care is needed for  the  jump set $J_u$ and jump set along cracks $S_u^A$  (since there is no bound on the number of connected components), where we need to extend some ideas from  \cite[Proposition 4.1]{HP:2019}.

In Section \ref{sec:proof_thm1} we prove  the main results of the manuscript, i.e.,  the existence and regularity results that are contained in Theorems \ref{teo:global_existence} and  \ref{teo:regularity_of_minimizers}, respectively.  In order to prove Theorem \ref{teo:global_existence} we first establish in Proposition \ref{prop:min_extend} the \bu equalities \ba
\begin{equation}\label{equal_minimums}
\inf\limits_{\bu(B,v)\ba\in \tilde \admissible,\,\,|\bu B\ba|=\fm} \tilde\cF\bu(B,v)\ba
=\inf\limits_{\bu(B,v)\ba\in \admissible,\,\,|\bu B\ba|=\fm} \cF\bu(B,v)\ba
= \lim\limits_{m\to\infty} \inf\limits_{\bu(B,v)\ba\in \admissible_m,
\,\,|\bu B\ba|=\fm} \cF\bu(B,v)\ba.   
\end{equation}
(recall that the second equality follows from \cite[Theorem 2.6]{HP:2019})  by using similar arguments previously  used in \cite[Theorem 2.6]{HP:2019}. 
In particular, \eqref{lsc_type} and \eqref{equal_minimums} imply that the configuration $(A,u)\in\tilde\admissible$ is a minimizer of $\tilde \cF$ in $\tilde\admissible.$ In Theorem \ref{teo:density_estimates_tilde} we establish the uniform density estimates for the  jump set $S_u^A$ of $u$ along cracks for a minimizer $(A,u)$ of $\tilde\cF.$ In particular, $S_u^A$ is  then essentially closed, and using this fact in Proposition \ref{prop:constructed_min_F} we construct a configuration  $(A',u)\in\admissible,$ which minimizes boths $\tilde \cF$ and $\cF,$ starting from a minimizer $(A,u)$ of $\tilde\cF$ in $\tilde\admissible.$ Moreover, $  (A',u)$ solves both \eqref{constrainedproblem} and \eqref{constrainedproblemtilde} and  satisfies the  properties stated in Theorem \ref{teo:regularity_of_minimizers}.   Theorem \ref{teo:regularity_of_minimizers} is then  a direct consequence of Proposition \ref{prop:constructed_min_F}, comparison arguments, the isoperimetric inequality in $\mathbb{R}^2$, and the equivalence of the constrained minimum problems and the unconstrained penalized minimum problem related  to the energies $\cF^\lambda$ and $\tilde\cF^\lambda$. 
 
We conclude the manuscript  with Appendix \ref{appendix} that contains some subsidiary results recalled for Reader's convenience since  very relevant in the arguments used throughout the paper.

\section{Mathematical setting}\label{sec:setting_results}

In this section we recall the SDRI model from \cite{HP:2019}, collect all the definitions and \bu hypotheses \ba and state the main results of the paper. Since our model is two-dimensional, unless otherwise stated, all sets we consider are subsets of $\R^2.$ We choose the standard basis $\{{\bf e_1}=(1,0), {\bf e_2}=(0,1)\}$ in $\R^2$ and denote the coordinates of $x\in\R^2$  with respect to this basis by $(x_1,x_2).$ We denote by $\Int A$ the interior of $A\subset \R^2.$ Given a Lebesgue measurable set $E,$ we denote by $\chi_E$ its characteristic  function and by $|E|$ its Lebesgue measure. The set 
$$
E^{(\alpha)}:=
\Big\{x\in\R^2:\,\, \lim\limits_{r\to0} \frac{|E\cap B_r(x)|}{|B_r(x)|}=\alpha \Big\},
\qquad \alpha\in [0,1],
$$
where $B_r(x)$ denotes the ball in $\R^2$ centered at $x$ of radius $r>0,$ is called the set of points of density $\alpha$ of $E.$  Clearly, $E^{(\alpha)}\subset \p E$ for any $\alpha\in (0,1),$ where 
$$
\p E:=
\{x\in \R^2:\,\,
\text{$B_r(x)\cap E \ne \emptyset$ and $B_r(x)\setminus E \ne \emptyset$ for any $r>0$}\} 
$$
is the topological boundary. The set $E^{(1)}$ is the {\it Lebesgue  set} of $E$ and $|E^{(1)}\Delta E|=0.$ We denote by $\p^*E$ the {\it reduced} boundary of a set $E$  of finite perimeter \cite{AFP:2000,Gi:1984}, i.e.,
$$
\p^*E:=\Big\{x\in\R^2:\,\, \exists \nu_E(x):=-\lim\limits_{r\to0}
\frac{D\chi_E(B_r(x))}{|D\chi_E|(B_r(x))},\quad |\nu_E(x)|=1\Big\}. 
$$
The vector $\nu_E(x)$ is called the {\it generalized outer  normal} to $E.$ 

\begin{remark}

If $E$ is a set of finite perimeter, then  
\begin{itemize}

\item[{\tiny$\bullet$}]  $\cl{\p ^*E} =\p E^{(1)}$ (see e.g., \cite[Eq. 15.3]{Ma:2012});

\item[{\tiny$\bullet$}]  $\p^*E \subseteq E^{(1/2)}$ and $\cH^{1}(E^{(1/2)}\setminus \p ^*E) =0$ (see e.g., \cite[Theorem 16.2]{Ma:2012});

\item[{\tiny$\bullet$}]    $P(E,B) = \cH^{1}(B\cap \p^*E)= \cH^{1}(B\cap E^{(1/2)})$ for any Borel set $E;$
\end{itemize} 
where $P(E,B)$ and $\cH^1$ denote the \emph{perimeter} of $E$ in $B$ and the $1$-dimensional Hausdorff measure, respectively. 

\end{remark}

An $\cH^1$-measurable set $K$  is called $\cH^1$-{\it rectifiable} if $\cH^1(K)<\infty$  and there exist countably many Lipschitz functions $f_i:\R\to\R^2$ such that 
\begin{equation}\label{rectifiability_def}
\cH^1\Big(K\setminus \bigcup\limits_{i\ge1} f_i(\R)\Big) =0 
\end{equation}
(see e.g., \cite[Definition 2.57]{AFP:2000}).  Notice that one can assume in \eqref{rectifiability_def} that the functions $f_i$ are $C^1$,  since Lipschitz functions are a.e.\  differentiable. %
By the Besicovitch-Marstrand-Mattila Theorem (\cite[Theorem 2.63]{AFP:2000} a Borel set $K\subset\R^2$ with $\cH^1(K)<+\infty$ is $\cH^1$-rectifiable if and only if
$\theta^*(K,x)=\theta_*(K,x) =1$ for $\cH^1$-a.e. $x\in K,$  where 
$$
\theta^*(K,x):=\limsup\limits_{r\to0^+} \frac{\cH^1(B_r(x)\cap K)}{2r}
\quad \text{and}\quad 
\theta_*(K,x):=\liminf\limits_{r\to0^+} \frac{\cH^1(B_r(x)\cap K)}{2r}.
$$
In particular, any $\cH^1$-rectifiable set $K$ admits a approximate tangent line at $\cH^1$-a.e. $x\in K,$ see e.g., \cite[Remark 10.3]{Ma:2012}. When $\theta_*(K,x)=\theta^*(K,x)=1,$ we write for simplicity $\theta(K,x)=1.$
 A Borel set $K\subset \R^2$ with $\cH^1(K)<+\infty$ is said \emph{purely unrectifiable} if $\cH^1(K\cap \Gamma)=0$ for every $1$-dimensional Lipschitz graph $\Gamma\subset \R^2$ (see e.g., \cite[Definition 2.64]{AFP:2000}).

Moreover, by \cite[Theorem 5.7]{D:2008_book} , if $K\subset\R^2$ is   an arbitrary Borel set with $\cH^1(E)<+\infty,$ then there exist Borel subsets $K^r$ and $K^u$ of $K$ such that $K=K^r\cup K^u,$  $K^r$ is $\cH^1$-rectifiable and $K^u$ is purely unrectifiable, and such a decomposition is unique up to a $\cH^1$-negligible set. More precisely, if $K=L^r\cup L^u$ with $\cH^1$-rectifiable $L^r$ and purely unrectifiable $L^u,$ then $\cH^1(K^r\Delta L^r)=\cH^1(K^u\Delta L^u)=0$. In what follows we call $K^r$ and $K^u$ the rectifiable and purely unrectifiable parts of $K,$ respectively. When $A\subset\R^2$ with $\cH^1(\p A)<+\infty,$ we denote by $\p^rA$ and $\p^uA$ the $\cH^1$-rectifiable and purely unrectifiable parts of $\p A,$ respectively.

The notation $\dist(\cdot,E)$ stands for 
the distance function from the set $E\subset\R^2$
with the convention that $\dist(\cdot,\emptyset)\equiv+\infty.$
Given a set $A\subset\R^2,$ we consider also signed distance function
from $\p A,$ negative inside, defined as 
\begin{equation}\label{signed_distance}
\sdist(x,\p A):= 
\begin{cases}
\dist(x, A) &\text{if $x\in \R^2\setminus A,$}\\ 
-\dist(x,\R^2\setminus A) &\text{if $x\in  A.$}  
\end{cases}
\end{equation}

\begin{remark} 
The following assertions are equivalent:
\begin{itemize}  
\item[(a)] $\sdist(x,\p E_k) \to \sdist(x,\p E)$ locally uniformly in $\R^2;$

\item[(b)] $E_k\overset{K}\to \cl{E}$ and $\R^2\setminus E_k\overset{K}\to \R^2\setminus \Int E,$ \bu
where $K$ denotes the \ba Kuratowski convergence of sets \cite{D:1993}.
\end{itemize}
Moreover, either assumption  implies 
$\p E_k \overset{K}{\to} \p E.$
\end{remark} 

Given $r>0,$ $\nu\in\S^1$ and $x\in\R^2$ we denote  by $ Q_{r,\nu}(x)$ the square of sidelength $2r$ centered at $x$ whose sides are either parallel or perpendicular to $\nu$. When $\nu={\bf e_2}$ or $\nu={\bf e_1},$ \bu we \ba drop the dependence on $\nu$ and write $ Q_r(x)$\bu.    If \ba in addition $x=0,$ we write just $ Q_r.$  We also set 
\begin{equation}\label{def:I_r_Q_rpm}
\red I_r:=[-r,r]\times\{0\},\, Q_r^+ = \{x\in Q_r:\,\, x\cdot \mathbf{e}_2 > 0\},\,\text{and}\, Q_r^- = \{x\in Q_r:\,\, x\cdot \mathbf{e}_2 < 0\}. 
\end{equation}
Given $x\in\R^2$ and $r>0,$ the blow-up map $\sigma_{x,r}$ is defined as
\begin{equation}\label{blow_ups}
\sigma_{x,r} (y)= \frac{y-x}{r}. 
\end{equation}
The blow-up \bu   of \ba$ K\subset\R^2$ is defined as  $\sigma_{x,r}(K)$.

Given an open set $U\subset\R^2$ and a metric space $X$ we denote by $\Lip(U;X)$ the family of all Lipschitz functions $\psi:U\to X.$ We denote by $\Lip(\psi)$ the Lipschitz constant of $\psi\in \Lip(U;X).$
 Furthermore,  $GSBD(U;\R^2)$ denotes the collection of all \emph{generalized special functions of bounded deformation} (see \cite{ChC:2019_jems,D:2013} for their definition and properties).  Given  $u\in GSBD(U;\R^2)$ we denote with $\str{u}\in\mtwo$ the
\emph{approximate symmetric gradient} of $u$, for which  
$$
\aplim\limits_{y\to x} \frac{\red[\ba u(y) - u(x) - \str{u}(x)(y-x)\red]\ba \cdot (y-x)}{|y-x|^2}=0
$$
holds for a.e.\ $x\in U$ by \cite[Theorem 9.1]{D:2013}, and with $J_u$ the \emph{jump set} of $u$, which is  $\cH^1$-rectifiable by \cite[Theorem 6.2]{D:2013}.  Let us also define 
$$
GSBD^2(U,\R^2):=\{u\in GSBD(U;\R^2):\,\,\str{u}\in L^2(U;\mtwo)\}.
$$ 
Given a $\cH^1$-rectifiable set $M\subset \overline{U},$ we
consider a normal vector $\nu_M$ to its approximate tangent line and we  
denote by $u_M^+$ and $u_M^-$ the approximate limits of $u\in GSBD^2(U;\R^2)$ with respect to $\nu_M,$  i.e., 
\begin{equation}\label{app_limits_ofu}
u_M^+(x):=\aplim\limits_{\substack{(y-x)\cdot \nu_M>0,\\y\in U}} \,\,u(y)\quad \text{and}\quad u_M^-(x):= \aplim\limits_{\substack{(y-x)\cdot \nu_M<0\\y\in U}}\,\, u(y)  
\end{equation}
for every $x\in M$ whenever they exist (see \cite[Definition 2.4]{D:2013}). 
We refer to $u_M^+$ and $u_M^-$ as the \emph{two-sided traces} of $u$ at $M$ and we notice that  they are uniquely determined up to a permutation when changing the sign of $\nu_M.$ 
If $U=\Int{A}$ for some measurable set $A$ with $\cH^1(\p A)<+\infty$ and $M:=\p^r A$, we use the simplified notations  $u_{\p A}^\pm$ on $A^{(1)}\cap \p^r A,$ and $\tr_Au:=u_{\p A}^+$ on $\p^* A,$ where on $\p^*A$ we always choose $\nu_M$  in \eqref{app_limits_ofu} as the generalized outer unit normal to $A.$ Moreover, we define
\begin{equation}\label{defini_Su}
S_u^A:=\{x\in A^{(1)}\cap\p^r A:\, u_{\p A}^+(x) \ne u_{\p A}^-(x)\}.
\end{equation}
Note that $S_u^A$ is $\cH^1$-rectifiable. We refer to $S_u^A$ the jump set of $u$ along the cracks of $A$.

A linear function $a:\R^2\to\R^2,$ defined as $Ax =Mx+b,$ where $M$ is $2\times 2$-matrix and $b\in\R^2,$ is an (infinitesimal) rigid displacement if  $M=-M^T.$

\subsection{The SDRI model} \label{subsec:SDRI_model}
Given two  nonempty bounded \bu Lipschitz connected \ba open sets $\Omega \subset\R^2$ and $\substrate\subset\R^2\setminus\Omega$ \bu such that \ba $\overline{\Omega}\cap\overline{\substrate}\ne\emptyset$ \bu and the set $\Sigma:=\p \substrate\cap \p \Omega$ is a Lipschitz 1-manifold,  \ba we define the family of admissible regions  for the {\it free crystal} and  the space of {\it admissible  configurations} by 
$$
\fA:=\{A\subset\overline{\Omega} :\,\, \text{$A$ is $\mathcal{L}^2$-measurable and $\p A$ is $\cH^1$-rectifiable}\}
$$
and 
$$
\begin{aligned}
\admissible:=\big\{(A,u):\,\,&  A\in \fA,\\
& u\in GSBD^2(\Ins{A};\R^2) \cap H_\loc^1(\Int{A}\cup \substrate;\R^2)\big\},
\end{aligned}
$$
respectively.   
By Proposition \ref{prop:adm_sets_have_finite_per} any $A\in\fA$ has finite perimeter. Furthermore, $J_u\subset\Sigma\cap \overline{\p^* A}$ since $u\in H_\loc^1(\Int{A}\cup \substrate;\R^2).$

The {\it energy} of admissible configurations is given by $\cF:\admissible\to[-\infty,+\infty],$ 
\begin{equation}\label{SDRIenergy} 
\cF:= \cS + \cW, 
\end{equation}
where $\cS$ and $\cW$ are the surface and elastic energies of the configuration, respectively. The surface energy of $(A,u)\in\admissible$ is defined as 
\begin{align}\label{func_surface_energy}
\cS(A,u):=& \int_{\Omega \cap\p^*A} \varphi(x,\nu_A(x))d\cH^1(x) \nonumber \\
&+\int_{\Omega \cap (A^{(1)}\cup A^{(0)})\cap\p A}
\big(\varphi(x,\nu_A(x)) + \varphi(x,-\nu_A(x))\big)d\cH^1(x)\nonumber\\  
& + \int_{\Sigma\cap A^{(0)}\cap\p A} \big(\varphi(x,\nu_\Sigma(x))
+ \beta(x)\big)d\cH^1(x) \nonumber \\
& + \int_{\Sigma\cap \p^*A\setminus J_u} \beta(x) d\cH^1(x) 
 + \int_{J_u} \varphi(x,-\nu_\Sigma(x))\,d\cH^1(x), 
\end{align}
where $\varphi:\overline\Omega\times\S^1\to[0,+\infty)$ and $\beta:\Sigma\to\R$  are  Borel functions  denoting the {\it anisotropy} of crystal and the {\it relative adhesion} coefficient of the substrate, respectively, and  $\nu_\Sigma:=\nu_\substrate$. In the following we refer to the first term in \eqref{func_surface_energy} as  the {\it free-boundary energy}, to the second as the {\it energy of
internal cracks and external filaments}, to the third as the {\it wetting-layer energy},  to the fourth as the {\it contact energy}, and to the last as the {\it delamination energy}. In applications instead of $\varphi(x,\cdot)$ it is more convenient to use  its positively one-homogeneous extension $|\xi|\varphi(x,\xi/|\xi|).$ With a slight abuse of notation we denote this extension  also by $\varphi.$

The elastic energy of $(A,u)\in\admissible$ is defined as
$$
{\mathcal W}(A,u):= \int_{A\cup  \substrate } W(x,\str{u(x)}  -  M_0  (x))dx,
$$
where the elastic density $W$ is determined as the quadratic form 
$$
W(x,M): =  \C(x)M:M,
$$
by the so-called {\it stress-tensor}, a measurable function $x\in\Omega\cup\substrate\to\C(x),$ where 
$\C(x)$ is a nonnegative  fourth-order tensor in the Hilbert space  $\mtwo$ of all $2\times 2$-symmetric matrices with the natural inner product 
$$
M:N=\sum\limits_{i,j=1}^2 M_{ij}N_{ij}
$$
for $M=(M_{ij})_{i,j=1}^2,$ $N=(N_{ij})_{i,j=1}^2\in\mtwo.$

The {\it mismatch strain}  $x\in\Omega\cup\substrate\mapsto M_0(x)\in\mtwo$ is given by  
$$
M_0: = 
\begin{cases}
\str{u_0} & \text{in $\Omega ,$}\\
0 &  \text{in $\substrate,$}
\end{cases}
$$ 
for a fixed $ u_0\in H^1(\R^2;\R^2)$.

Given $m\in \N,$ let $\fA_m$ be a collection of all $A\in\fA$ such that $\p A$  has  at most $m$ connected components and let 
$$
\admissible_m:=\Big\{(A,u)\in \admissible:\,\,A\in \fA_m \Big\} 
$$ 
to be the set of constrained admissible configurations. For simplicity, we assume that $\admissible_\infty=\admissible.$ 

 \begin{remark}\label{rem:counterexample}
The reason to introduce $\admissible_m$ is that $\admissible_m$ is both closed under  $\tau_\admissible$-convergence (see \cite[Definition 2.5]{HP:2019}) and $\cF$ is lower semicontinuous with respect to $\tau_\admissible$ in  $\admissible_m$ (see  \cite[Theorems 2.7 and 2.8]{HP:2019}). Such two properties do not apply instead to $\admissible$ as the following examples show.

We begin by recalling that a sequence $\{(A_k,u_k)\}\subset\admissible$ is said to $\tau_{\admissible}$-converge
to $(A,u)\subset\admissible$  and  we denote by $(A_k,u_k)\overset{\tau_\admissible}{\to} (A,u)$, if
\begin{itemize}
\item[--] $\sup\limits_{k\ge1} \cH^1(\p A_k) < \infty,$
\item[--] $\sdist(\cdot,\p A_k)\to \sdist(\cdot,\p A)$ 
locally uniformly in $\R^2$ as $k\to \infty$,
\item[--] $u_k\to u$ a.e.\ in $\Int{A}\cup S$.
\end{itemize} 

Let $X:=\{x_n\}$ be a 
countable dense set in 
$\Omega$ and $A\in\mathcal{A}$ such that $|A|=\fm\in(0,|\Omega|]$. Then the sets  $A_k:=A\setminus\{x_1,\ldots,x_k\}\in\fA$, $k\in\mathbb{N}$, are such that  $|A_k|=\fm\in(0,|\Omega|)$, $\cH^1(\p A_k) = \cH^1(\p A),$ and 
$(A_k,0)\overset{\tau_{\admissible}}{\to} (A\setminus X,0)$ as $k\to\infty$,
but $A\setminus X\notin\fA$
since $\p (A\setminus X) = \overline{A}.$ 
Therefore, compactness with respect to $\tau_\admissible$ fails in $\admissible$. 

Furthermore, let $\Gamma\subset A$ be a segment such that $\mathcal{H}^1(\Gamma)>0$, $B:=A\setminus\Gamma$, $B_k:=A\setminus(\Gamma\cap\{x_1,\ldots,x_k\})$ for every $k\in\mathbb{N}$, and assume that $X$ is dense in $\Gamma$. We notice that $\{(B_k,0)\}\subset\admissible$, $(B,0)\in\admissible$, $|B_k|=|B|=|A|$, $(B_k,0)\overset{\tau_{\admissible}}{\to} (B,0)$ as $k\to\infty$. However, 
$$
\cF(B_k,0)=\cF(A,0)<\cF(A\setminus\Gamma,0)=\cF(B,0).
$$
Therefore, lower semicontinuity of $\cF$ with respect to $\tau_\admissible$ fails in $\admissible$. 

\end{remark}

\subsection{Localized energies} 

In this section we introduce the notion of quasi minimizers of $\cF$ and $\tilde \cF$ in $\Omega$ and the localized version $\cF(\cdot;\openset):\admissible_m\to\R$ of $\cF$ for open sets  $\openset\subset\Omega$ and for $m\in\N\cup\{\infty\}$ with the convention $\admissible_\infty:=\admissible.$   We define
\begin{equation}\label{local_cF}
\cF(A,u;\openset):=\cS(A;\openset) + \cW(A,u; \openset), 
\end{equation}
where 
$$
\cS(A;\openset):= \int_{\openset \cap \p^*A} \varphi(y,\nu_A)d\cH^1
+ 2\int_{\openset \cap (A^{(1)}\cup A^{(0)})\cap \p A} \varphi(y,\nu_A)d\cH^1 
$$
and 
$$
\cW(A,u; \openset) = \int_{\openset\cap A} \C(y)\str{u}:\str{u}dy, 
$$
are the localized versions of the surface  and elastic energies, respectively. Since we define the localized energy $\cF(\cdot;\openset)$  only for open subsets $\openset$ of $\Omega,$ the localized surface energy $\cS(\cdot;\openset)$ does not depend on $u$ and the localized elastic energy $\cW(\cdot;\openset)$ can be defined without $u_0$; see also Remark \ref{rem:passage_toU0_teng0} below.   

\begin{definition} 
Given $\Lambda\ge 0$ and $m\in \N\cup\{\infty\},$ the configuration $(A,u)\in\admissible_m$
is a local {\it $(\Lambda,m)$-minimizer} of $\cF:\admissible_m\to\R$ in $\openset$ if
\begin{equation*} 
\cF(A,u;\openset) \le \cF(B,v;\openset) + \Lambda|A\Delta B|
\end{equation*}
whenever $(B,v)\in\admissible_m$ with $A\Delta B\strictlyincluded \openset$ and $\supp(u-v)\strictlyincluded \openset.$ 
Furthermore, we define 
\begin{align}\label{minimal_with_dirixle}
\Phi(A,u;\openset):= 
\inf\Big\{ 
\cF(B,v;\openset):\,\,\nonumber  & (B,v)\in\admissible_m,\\
& B\Delta A\strictlyincluded \openset,\,\, 
\supp (u-v)\strictlyincluded \openset 
\Big\}
\end{align}
and
\begin{equation}\label{deviation}
\Psi(A,u;\openset): = \cF(A,u;\openset) - \Phi(A,u;\openset) 
\end{equation}
for every $(A,u)\in\admissible_m$ and every open set $\openset \strictlyincluded \Omega.$  
\end{definition}

 \noindent
\begin{remark}\label{rem:passage_toU0_teng0} 
\red By  \cite[Theorem 2.6]{HP:2019} (see also  \eqref{zur_tenglik} below) \ba for any minimizer $(A,u)$ of $\cF$ in $\admissible_m,$ the configuration $(A,u-u_0)$ is a $(\lambda_0,m)$-minimizer of $\cF(\cdot,\cdot;\Omega).$  Indeed, since $(A,u)$ is a minimizer of $\cF^{\lambda_0}$ in $\admissible_m,$ the function $\hat u:= u - u_0$ minimizes $\admissible_m\ni(B,v)\mapsto \hat \cF^{\lambda_0}(B,v):=\cF^{\lambda_0}(B,v + u_0).$
Hence,  
for any open set $\openset\subset\Omega$ and $(B,v)\in\admissible_m$ with $A\Delta B\strictlyincluded \openset$ and $\supp(u - u_0 - v) \strictlyincluded \openset$ we have $\hat\cF^{\lambda_0}(A,u-\hat u_0) \le \hat\cF^{\lambda_0}(B,v)$ so that  
$$
\cF(A,u - u_0;\openset) \le \cF(B,v;\openset) + \lambda_0\big||A| - |B|\big| \le \cF(B,v;\openset) + \lambda_0\big|A\Delta B\big|. 
$$
Similarly, if $(A,u)$ is a minimizer of $\tilde \cF$ in $\tilde \admissible,$ the configuration $(A,u-u_0)$ is a $\lambda_0$-minimizer of $\tilde \cF(\cdot;\openset).$
\end{remark}

\subsection{Auxiliary model} \label{sec:auxiliary}

We also introduce a \emph{weak} formulation of the SRDI model defined in Section \ref{subsec:SDRI_model} for which the more general family $\tilde \admissible$ of 
admissible 
configurations, given by
$$
\begin{aligned}
\tilde \admissible:=\big\{(A,u):\,\,& A\in \tilde \fA,\\
& u\in GSBD^2(\Ins{A};\R^2) \cap H_\loc^1(\Int{A}\cup \substrate;\R^2)\big\}, 
\end{aligned}
$$
is considered, where  
$$
\tilde \fA:=\Big\{ A\subset\bu\overline{\Omega}\ba:\, \text{$A$ is $\mathcal{L}^2$-measurable and $\cH^1(\p A)<+\infty$}\Big\}. 
$$
The auxiliary energy $\tilde\cF:\tilde\admissible\to\R$ is defined as 
$$
\tilde \cF: = \tilde \cS + \cW,
$$
where
\begin{align}\label{surface_tilde}
\tilde \cS(A,u):= &  \int_{\Omega \cap\p^*A} \varphi(x,\nu_A(x))d\cH^1(x) \nonumber \\
&+\int_{S_u^A}
\big(\varphi(x,\nu_A(x)) + \varphi(x,-\nu_A(x))\big)d\cH^1(x)\nonumber\\  
& + \int_{\Sigma\cap \p^*A\setminus J_u} \beta(x) d\cH^1(x) 
 + \int_{J_u} \varphi(x,-\nu_\Sigma(x))\,d\cH^1(x),
\end{align} 
where $S_u^A\subset\Omega$ by definition \eqref{defini_Su}.

\subsection{Main results}  \label{subsec:main_results}

We begin by stating the hypotheses which will be assumed throughout the paper:  
\begin{itemize}
\item[(H1)] $\varphi\in C(\cl{\Omega}\times \R^2)$ and is a 
Finsler norm, i.e., there exist $c_2\ge c_1>0$ such that 
for every $x\in \cl{\Omega },$ $\varphi(x,\cdot)$ is a norm in $\R^2$
satisfying  
\begin{equation}\label{finsler_norm}
c_1|\xi| \le \varphi(x,\xi) \le c_2|\xi| 
\end{equation}
for any $x\in\cl{\Omega}$ and $\xi\in\R^2;$

\item[(H2)] $\beta\in L^\infty(\Sigma)$ and 
satisfies
\begin{equation}\label{hyp:bound_anis}
-\varphi(x,\nu_\Sigma(x))\le \beta(x) \le \varphi(x,\nu_\Sigma(x)) 
\end{equation}
for $\cH^1$-a.e. $x\in \Sigma;$

\item[(H3)]   
$\C\in L^\infty(\Omega\cup\substrate)\cap C^{0}(\cl{\Omega})$ and there exists 
$c_4\ge c_3>0$ such that 
\begin{equation}\label{hyp:elastic}
2c_3\,M:M \le \C(x)M:M \le 2c_4\,M:M 
\end{equation}
for any $x\in\Omega\cup\substrate$ and $M\in\mtwo;$

\item[(H4)] Either $\fm\in (0,|\Omega |)$ or $\substrate=\emptyset.$
\end{itemize}
\bigskip

\noindent 
 Given  $\cG\in\{\cF,\tilde\cF\},$ we use the notation: 
$$
\cX_{\cG}:=\begin{cases}\admissible & \text{if $\cG=\cF,$}\\ 
\tilde\admissible & \text{if $\cG=\tilde\cF$.}
\end{cases}
$$  

 The first result is the {\it existence} of solutions without constraint on the number of free-crystal boundary components. 

\begin{theorem}[\textbf{Existence}]\label{teo:global_existence}
Assume {\rm (H1)-(H4)}. Let $\cG\in\{\cF,\tilde\cF\}.$ Then the minimum problem 
\begin{equation}\label{min_prob_globals}
\inf\limits_{(\bu B,v\ba)\in \cX_\cG,\,\,|\bu B \ba| = \fm} \cG(\bu B,v\ba) 
\end{equation}
admits a solution.  Moreover, there exists $\lambda_1>0$ such that $(A,u)\in \cX_\cG$ is a solution of \eqref{min_prob_globals} if and only if it solves  
$$
\inf\limits_{(\bu B,v\ba)\in \cX_\cG} \cG^\lambda(\bu B,v\ba)
$$
for every $\lambda\ge \lambda_1,$ where  
\begin{equation}\label{eq:flambda}
\cG^\lambda(\bu B,v\ba):=\cG(\bu B,v\ba) +\lambda\big||\bu B \ba| - \fm\big|; 
\end{equation}
\end{theorem}

\noindent For simplicity we call the solutions of \eqref{min_prob_globals}   {\it global minimizers}.

The second result is a {\it partial regularity} of the free-crystal boundaries. \bu We recall that the definition of $S_u^A$ is provided in \eqref{defini_Su}. \ba

\begin{theorem}[\textbf{Properties of global minimizers}]\label{teo:regularity_of_minimizers}   
Assume {\rm (H1)-(H4)}. Let $\cG\in\{\cF, \tilde\cF\}$ and $(A,u)\in \cX_\cG$ be a solution of \eqref{min_prob_globals}.  
Define 
\begin{equation}\label{def_a_prime}
 A':=\Int{A^{(1)}}\setminus\overline \Gamma, 
\end{equation}
where $\overline\Gamma$ is the closure of $\{x\in S_u^A:\,\theta_*(S_u^A,x)>0\},$ and, with a slight abuse of notation,  consider $u$ as defined in $A'\cup \substrate$ \emph{(}and so, also on  the $\mathcal{L}^2$-negligible set $A'\setminus \Int{A}$\emph{)}. 
Then:  
\begin{itemize}
\item[(1)]  $A'$ is open, $\theta_*(S_{u}^{A'},x)>0$ for all $x\in S_{u}^{A'},$ $|A'\Delta A|=0,$ $\cH^1(\p A\Delta\p A')=0,$  $\cH^1(S_u\Delta S_{u}^{A'})=0$, $(A',u)\in \admissible,$ and 
 $$
\cG(A,u)=\cF(A',u)=\inf\limits_{(B,v)\in\admissible,\,|B|=\fm} \cF(B,v)=\inf\limits_{(B,v)\in\tilde\admissible,\,|B|=\fm} \tilde \cF(B,v);
$$   

\item[(2)] for any $x\in\Omega$ and $r\in(0,\min\{1,\dist(x,\p \Omega)\}),$
$$
\frac{\cH^1( Q_r(x)\cap \p A')}{r} \le \frac{16c_2 + 4\lambda_1}{c_1};
$$

\item[(3)] there exist $\varsigma_0=\varsigma_0(c_3,c_4)\in(0,1)$ and $R_0=R_0(c_1,c_2,c_3,c_4,\lambda_1)>0,$ where $\lambda_1>0$ is given in Theorem \ref{teo:global_existence}, with the following property: if  $x\in \Omega\cap \p A',$ 
then 
$$
\frac{\cH^1( Q_r(x)\cap \p A')}{r}\ge \varsigma_0 
$$
for any square $ Q_r(x)\strictlyincluded\Omega$ with  $r\in(0,R_0).$

\item[(4)]$ \overline{A'^{(1)}\cap\p A'}=\cl{S_{u}^{A'}}$ and 
$$
\cH^1(\cl{S_{u}^{A'}}\setminus S_{u}^{A' })=0,
$$
hence cracks essentially coincide with the jump set for the displacement $u$;

\item[(5)] 
If $E\subset A'$ is any connected component of $A'$ with $\cH^1(\p E\cap\Sigma\setminus J_{u})=0,$ then $|E|\ge (c_1 \sqrt{4\pi} /\lambda_1)^2$  and $u=u_0 + a$ in $E,$ where $a$ is a rigid displacement. 
\end{itemize}
\end{theorem}
  
\noindent 
 In what follows we refer to the estimates in (2) and (3) as the (uniform) {\it upper and lower density estimate}, respectively. Note that by assertion (1), the assertions (3) and (5) directly hold also for solutions $(A,u)$ of \eqref{min_prob_globals}.

\subsection{Examples} \label{examples}

We recall from \cite{HP:2019} that the SDRI energy  \eqref{SDRIenergy} coincides with the functionals of the following free-boundary problems considered in the Literature when restricted to the corresponding subfamilies of admissible configurations in $\mathcal{C}$:  

\begin{itemize}[leftmargin=\dimexpr+4mm]
\item[(a)]
{\it Epitaxially strained thin films, e.g.,} 
\cite{BCh:2002,DP:2017,DP:2018_1,FFLM:2007,FM:2012,KP:2021}:
$\Omega :=(a,b)\times (0,+\infty),$  
$\substrate :=(a,b)\times(-\infty,0)$ 
for some $a<b,$  
free crystals in the subfamily
$$
\begin{aligned}
\quad \mathcal{A}_{\rm subgraph}:=\{A\subset\Omega:\, 
\text{$\exists h\in BV(\Sigma;[0,\infty))$ and 
l.s.c. such that $A=A_h$} \}\subset\fA_1,
\end{aligned}
$$ 
where $A_h:=\{(x^1,x^2)\,:\, 0<x^2<h(x^1)\}$,
and admissible configurations in the subspace 
$$
\mathcal{C}_{\rm subgraph}:=\{(A,u):\, A\in\fA_{\rm subgraph},\,
u\in H_\loc^1(\Ins{A};\R^2)\}\subset\admissible_1
$$
(see also \cite{BGZ:2015,GZ:2014});

\item[(b)] {\it Crystal cavities, e.g.,} \cite{FFLM:2011,FJM:2018,SMV:2004,WL:2003}: 
$\Omega \subset\R^2$ smooth set containing the origin, 
$\substrate :=\R^2\setminus\Omega,$ free crystals in the subfamily
$$
\mathcal{A}_{\rm starshaped}:=\{A\subset\Omega:\, 
\text{\red open  
and $\Omega\setminus A$ starshaped w.r.t. $(0,0)$} \}\subset\fA_1,
$$ 
and the space of admissible configurations
$$
\mathcal{C}_{\rm starshaped}:=\{(A,u):\, \, A\in\fA_{\rm starshaped},\,
u\in H_\loc^1(\Ins{A};\R^2)\}\subset\admissible_1;
$$ 

\item[(b)]
{\it Capillarity droplets, e.g.,} \cite{CM:2007,GB-W:2004,DPhM:2015}: 
$\Omega\subset\R^2$ is a bounded Lipschitz open set (or a cylinder), 
admissible configurations in the collection
$$
\mathcal{C}_{\rm capillarity}:=\{(A,u_0):\, \,A\in \fA\}\subset \admissible \quad\text{or} 
\quad 
 \tilde {\mathcal{C}}_{\rm capillarity}:=\{(A,u_0):\, \,A\in \tilde \fA\}\subset \tilde \admissible;
$$

\item[(d)]
{\it Griffith fracture model, e.g.,} \cite{BFM:2008,CCI:2019.jmpa,ChC:2019_arxiv,CFI:2016,CFI:2018poincare,FGL:2019,FM:1998.jmps}:
$\substrate=\Sigma=\emptyset$,   $E_0\equiv0$, and 
the space of configurations
$$
\mathcal{C}_{\rm Griffith}:=\{(\Omega \setminus K,u):\,
\text{$K$ closed, $\mathcal{H}^1$-rectifiable},\,
u\in H_\loc^1(\Omega\setminus K;\R^2)\}\subset\admissible;
$$

\item[(e)]
{\it Mumford-Shah model,  e.g.,} \cite{AFP:2000,DMMS:1992,MS:2001}:
$\substrate=\Sigma=\emptyset,$ $E_0=0,$
$\C$ is such that the elastic energy $\mathcal{W}$ reduces 
to the Dirichlet energy, and
the space of configurations
$$
\admissible_{\rm Mumfard-Shah}: = 
\{
(\Omega \setminus K,u)\in \admissible_{\rm Griffith}:\,\,
u=(u_1,0)\}\subset\admissible;
$$ 

\item[(f)] {\it Boundary delaminations, e.g.,}  
\cite{Babadjian:2016,Deng:1995,HS:1991,Baldelli:2013,Baldelli:2014,Xia:2000}:  
 the SDRI model includes also the setting of 
debonding and edge delamination in composites \cite{Xia:2000}. The focus is here on  the 2-dimensional film and substrate vertical section, while in  \cite{Babadjian:2016,Baldelli:2013,Baldelli:2014} a reduced model for the  horizontal  interface between the film and the substrate is derived. \\

\end{itemize}

For the cases (a) and (b), the existence results for the SDRI model in $\admissible_{\rm subgraph}$ and $\admissible_{\rm starshaped}$ can be found for example in  \cite[Theorem 2.9 and Remark 2.10]{HP:2019}. For (c), the same statements of Theorems \ref{teo:global_existence} and \ref{teo:regularity_of_minimizers} hold with $\cX_\cG:=\admissible_{\rm capillarity}$ if $\cG=\cF$ or $\cX_\cG:=\tilde \admissible_{\rm capillarity}$ if $\cG=\tilde\cF$  (note that $S_u$ and $\Gamma$ are empty in this case).  
For (d)-(f), we postpone the analysis  
to future investigations since some modifications in the proofs is needed to include \emph{boundary Dirichlet conditions} or {\it fidelity terms} of type 
\begin{equation}\label{eq:fidelity}
\kappa\int_{\Omega\setminus K} |u-g|^p dx
\end{equation}
for  $p\in(1,\infty)$, $\kappa>0$, and $g\in L^{\infty}(\Omega),$ which are generally considered (and needed) in these mechanical applications.

\section{Decay estimates for $m$-minimizers}\label{sec:decay_estimates} 

In this section we always assume (H4). 
We recall that by \cite[Theorem 2.6]{HP:2019}   under the hypotheses  {\rm (H1)-(H3)} both the volume-contrained minimum problem
$$ 
\inf\limits_{(A,u)\in \admissible_m,\,\,|A| = \fm} \cF(A,u), 
$$
and the unconstrained minimum problem
$$ 
\inf\limits_{(A,u)\in \admissible_m} \cF^\lambda(A,u) 
$$
admit a solution for any $m\in\N$.
Moreover, by  \cite[Theorem 2.6]{HP:2019} 
there exists $\lambda_0>0$ such that 
\begin{equation}\label{zur_tenglik}
\inf\limits_{(A,u)\in \admissible,\,\,|A|=\fm} \cF(A,u) 
= \inf\limits_{(A,u)\in \admissible} \cF^\lambda(A,u) 
= \lim\limits_{m\to\infty} \inf\limits_{(A,u)\in \admissible_m,
\,\,|A|=\fm} \cF(A,u) 
\end{equation}
for every $\lambda\ge \lambda_0.$

The main results of this section are the following density  estimates for the  quasi-minimizers of $\cF$ in $\admissible_m$ with $m\in\N\cup\{\infty\}.$

\begin{theorem}[\textbf{Density estimates for  $(\Lambda,m)$-minimizers}]\label{teo:density_estimates}
There exist  $\varsigma_*=\varsigma_*(c_3,c_4)\in(0,1)$ and $R_*=R_*(c_1,c_2,c_3,c_4,\lambda_0)>0,$ where $c_i$ are given by \eqref{finsler_norm} and \eqref{hyp:elastic}, with the following property. Let  $(A,u)\in\admissible_m$  be a $(\Lambda,m)$-minimizer of 
$\cF(\cdot,\cdot;\Omega)$ in $\admissible_m$ for some $ m\in\N\cup \{\infty\}.$ Then for any $x\in\Omega$ and $r\in(0,\dist(x,\p \Omega)),$
\begin{equation}\label{min_seq_density_up}
\frac{\cH^1( Q_r(x)\cap \p A  )}{r} \le \frac{16c_2 + 4\Lambda}{c_1}. 
\end{equation}
Moreover, if  $x\in\Omega$ belongs to the closure of the set  $\{y\in \Omega\cap \p A:\,\theta_*(\p A,y)>0\},$  
then 
\begin{equation}\label{min_seq_density_low}
\frac{\cH^1( Q_r(x)\cap \p A )}{r}\ge \varsigma_* 
\end{equation}
for any square $ Q_r(x)\strictlyincluded\Omega$ with 
 $r\in(0,R_*).$
\end{theorem}

\noindent  
To prove Theorem \ref{teo:density_estimates} we start with the following adaptation of \cite[Theorem 3]{CCI:2019.jmpa} to our setting (of set-function pairs). 

\begin{lemma}\label{lem:approximation_chambolle}
\red There exist $\eta\in(0,1/32)$  and $c_0>0$ with the following property: For any $m\in\N\cup\{\infty\}$,  any admissible $(A,u)\in \admissible_m$,  and any \ba square $ Q_R(x_0)\subset \Omega$ of sidelength $2R>0$ with 
\begin{equation}\label{delta_les_eta}
\delta:=\Big(\frac{\cH^1( Q_R(x_0)\cap  \red \p^r A )}{R}\Big)^{1/2}< \eta
\end{equation}
there exist 
$v\in GSBD^2(\Ins{\Omega};\R^2),$ $B\in\fA$ with $(B,v\big|_B)\in\admissible_m,$ $ R'\in (R(1-\sqrt{\delta}),R)$ and a Lebesgue measurable set $\omega\strictlyincluded  Q_R(x_0)$ such that
\begin{itemize}
\item[\emph{(1)}] $v\in C^\infty( Q_{R(1-\sqrt\delta)}\red (x_0)\ba),$ $A\Delta B\strictlyincluded   Q_{\red R'}(x_0)   \setminus  Q_{R(1-\sqrt{\delta})}(x_0)$ and $\supp(\tilde u-v)\strictlyincluded  Q_R(x_0),$ where 
\begin{equation}\label{new_displacement}
\tilde u:=u\chi_{ Q_R(x_0)\cap A} + \xi\chi_{ Q_R(x_0)\setminus A}, 
\end{equation}
where $\xi\in  Q_R$ is chosen such that $ Q_R\cap \p^*A\subset J_{\tilde u};$  

\item[\emph{(2)}]  $\cH^1(\p B\setminus \p A)  \le c_0\,\sqrt{\delta}\,\cH^1([ Q_R(x_0)\setminus  Q_{R(1-\sqrt{\delta})}(x_0)]\cap \p A );$ 

\item[\emph{(3)}] $|\omega|\le c_0 \delta\,\cH^1( Q_R(x_0)\cap \p A)$ and  
$$
\int_{ Q_R(x_0)\setminus \omega} |v-\tilde u|^2dx \le c_0\delta^2R^2\,\int_{ Q_R(x_0)}|\str{\tilde u}|^2dx;
$$

\item[\emph{(4)}] for any $\psi\in \Lip( Q_R;[0,1])$ and elasticity tensor $\C\in L^\infty( Q_R)$ with 
\begin{equation}\label{elastic_bounda}
d_1M:M \le \C(x)M:M\le d_2M:M,\qquad (x,M) \in  Q_R\times \mtwo,  
\end{equation}
there exist $d_3:=d_3(c_0,d_1,d_2)>0$ and $s:=s(c_0,d_1,d_2)\in(0,1/2)$ such that 
\begin{align*} 
\int_{ Q_R\red(x_0)} \psi \C(x)\str{v}:\str{v}dx \le & \int_{ Q_R\red(x_0)\ba \cap A} \psi \C(x)\str{u}:\str{u}dx \nonumber\\
&+ d_3\,\delta^s\,(1+R\,\Lip(\psi)) \int_{ Q_R\red(x_0)\ba \cap A}|\str{u}|^2dx.
\end{align*} 

\end{itemize}
\end{lemma}

The proof of Lemma \ref{lem:approximation_chambolle} is an adaptation of the arguments of \cite[Theorem 3]{CCI:2019.jmpa} to our situation of functional depending on set-function pairs with extra care paid for the constraint on the number of boundary connected components. The idea is to treat the boundary of each admissible region as a jump of a properly defined displacement. In particular, we choose such displacement of the type \eqref{new_displacement}, where $\xi$ is selected as in the construction used in the proof of \cite[Lemma 3.10]{HP:2019}. \red We also notice that the constants $\eta$ and $c:=c_0/(1+\sqrt2/24)>0$ are given by \cite[Theorem 3]{CCI:2019.jmpa}. \ba

\begin{proof}[Proof of Lemma \ref{lem:approximation_chambolle}]
  
By translating and rescaling if necessary, we assume that $x_0=0$ and $R=1.$ Notice that since $\cH^1(Q_1\cap \p A)<+\infty,$ by \blue Proposition \ref{prop:maggi_foliation} \black there exists $\xi\in(0,1)^2$ such that the set 
$$
\{x\in Q_1\cap \p^* A:\,\text{$\tr_A(u)$ exists and is equal to $\xi$}\}
$$ 
is $\cH^1$-negligible. By \cite[Theorem 4.4]{Gi:1984} up to a $\cH^1$-negligible set we can cover $Q_1\cap \p^* A$   with $C^1$-maps so that by \cite[Theorem 5.2]{D:2013} $\tr_A(u)$ \bu exists \ba $\cH^1$-a.e.\ on $Q_1\cap \p^* A.$ 

Let  
$$
\tilde u:=u\chi_{ Q_1\cap A} + \xi\chi_{ Q_1\setminus A}. 
$$
Note that $\tilde u\in GSBD^2(Q_1;\R^2)$ and by the choice of $\xi$ and \bu by \ba  \cite[Definition 2.4]{D:2013} $\red Q_1\cap \p^* A\subset J_{\tilde u}.$
In addition, by possibly   adding to  $\tilde u$ a function in $SBD^2(Q_1;\R^2)\cap W^{1,\infty}(Q_1\setminus\p A;\R^2)$ with  small $W^{1,\infty}(Q_1\setminus \p A;\R^2)$ norm, jump on the set $Q_1\cap \red \p^r A$, and   supported near $ Q_1\cap \p A$, we  can assume without loss of generality   that $ Q_1\cap J_{\tilde u}\supset Q_1\cap  \red \p^r A $ up to a $\cH^1$-negligible set\footnote{A similar argument was used in \cite[p. 1359, above Eq. 4.19]{ChC:2019_arma}}. 
Notice that
$$
\delta:=\cH^1( Q_1\cap \red \p^r A \ba)^{1/2} = \cH^1( Q_1\cap J_{\tilde u})^{1/2}
$$
and set $N:=[1/\delta]$ so that $(-N\delta,N\delta)^2\subset  Q_1.$ For $i:=0,1,\ldots,N-1$ let $Q^i:=(-(N-i)\delta,(N-i)\delta)^2$ and $C^i:=Q^i\setminus Q^{i+1}$ (assuming $C^{N-1}:=Q^{N-1}$). Up to a slight translation of $Q^i$ we assume that $\cH^1(\p A\cap \p Q^i)=0$ for all $i.$ By \cite[Lemma 3.3]{CCI:2019.jmpa} we find $i_0\ge1$ such that 
$$
\begin{cases}
\int_{C^{i_0}\cup C^{i_0+1}} |\str{\tilde u}|^2dx \le 8\sqrt{\delta} \int_{ Q_1\setminus  Q_{1-\sqrt{\delta}}} |\str{\tilde u}|^2dx,\\[2mm]
\cH^1(\p A \cap (C^{i_0}\cup C^{i_0+1})) \le 8\sqrt{\delta}\, \cH^1( \p A \cap ( Q_1\setminus  Q_{1-\delta})). 
\end{cases}
$$
We partition $Q^{i_0+1}$ into pairwise disjoint squares with sidelength $\delta$ and divide the slice $C^{i_0}$ into dyadic slices 
$$
G_j:=(-(N-i_0-2^{-j})\delta,(N-i_0-2^{-j})\delta)^2\setminus (-(N-i_0-2^{-j+1})\delta,(N-i_0-2^{-j+1})\delta)^2,
$$
then we partition each slice $G_j$ into  pairwise disjoint squares $ Q_{j,l}$ of sidelength $2^{-j}\delta$ whose sides \bu are \ba parallel to the coordinate axis. Let $\cV_0$ be the collection of all squares of sidelength $\delta$ that cover the central square $Q^{i_0+1}$ and let $\cV$ be the union of $\cV_0$ and of the collection of all $ Q_{j,l}.$ Following \cite{CCI:2019.jmpa} we differentiate between \blue ``good'' and ``bad'' \black squares in $\cV.$ A square $Q\in\cV$ is ``good'' if 
\begin{equation}\label{good_cubes}
\cH^1(Q'''\cap  \p A ) \le \eta \delta_Q^{}, 
\end{equation}
where $Q'''$ is the square with the same center as $Q$ and dilated by $7/6,$ and $\delta_Q^{}:=\delta$ if $Q\in\cV_0$ and $\delta_Q^{}:=2^{-j}\delta$ if $Q\subset G_j.$  A square $Q$ is ``bad'' if it does not satisfy \eqref{good_cubes}. By \eqref{delta_les_eta} $\delta^2=\cH^1( Q_1\cap \p A) <\eta\delta,$ hence, by definition, all squares in $\cV_0$ are good and by \cite[Eq. 12]{CCI:2019.jmpa} the sum of the perimeters of all bad squares satisfies
\begin{equation}\label{perimeter_bad_cubes}
\sum\limits_{Q\,\text{bad}} \cH^1(\p^*Q) \le \tilde c_0\sqrt{\delta}\, \cH^1(( Q_1\setminus  Q_{1-\sqrt{\delta}})\cap \p A )  
\end{equation}
for some $\tilde c_0>0$.
Since $\delta<\eta,$ by \cite[Theorem 3]{CCI:2019.jmpa} there exist $\tilde v\in GSBD^2( Q_1;\R^2),$ $r\in(1-\sqrt\delta,1)$ and a Lebesgue measurable set $\tilde \omega\strictlyincluded  Q_r$ such that 
\begin{itemize}
\item[(a1)] $\tilde v \in C^\infty( Q_{1-\sqrt{\delta}}),$ $\tilde u=\tilde v$ in $ Q_1\setminus  Q_r$ and $\cH^1(J_{\tilde u}\cap \p  Q_r)=\cH^1(J_{\tilde v}\cap \p  Q_r)=0;$

\item[(a2)]  $\cH^1(J_{\tilde v}\setminus J_{\tilde u})  \le \tilde c_0\,\sqrt{\delta}\,\cH^1(( Q_1\setminus  Q_{1-\sqrt{\delta}})\cap J_{\tilde u});$ 

\item[(a3)] $|\tilde \omega |\le \tilde c_0\delta\cH^1( Q_r\cap  \p A )$ and 
$$
\int_{ Q_1\setminus \tilde\omega}|\tilde v - \tilde u|^2dx \le \tilde c_0\delta^2 \int_{ Q_1} |\str{\tilde u}|^2dx;
$$

\item[(a4)] for any $\psi\in \Lip( Q_1;[0,1])$ and elasticity tensor $\C\in L^\infty( Q_1)$ satisfying \eqref{elastic_bounda} there exists $d_3:=d_3(\tilde c_0,d_1,d_2)>0$ such that 
$$
\int_{ Q_1} \psi \C(x)\str{\tilde v}:\str{\tilde v}dx \le \int_{ Q_1} \psi \C(x)\str{\tilde u}:\str{\tilde u}dx + d_3\delta^s\,(1+\Lip(\psi)) \int_{ Q_1}|\str{\tilde u}|^2dx
$$
\red with $s\in(0,1)$ depending only on $\tilde c_0,$ $d_1$ and $d_2;$ \ba
\item[(a5)] $J_{\tilde v} \subset \p^* D\cup (J_{\tilde u} \setminus Q^{i_0+1})$ and $J_{\tilde v}\setminus J_{\tilde u} \subset \p^*D,$ where $D$ is the union of all bad squares.
\end{itemize}

Note that for proving  (a4) in \cite{CCI:2019.jmpa} a \emph{mollifying argument} is used (together with the fact that $\C$ is assumed to be constant in   \cite{CCI:2019.jmpa}). As in our setting $\C$ is in general not constant, we revised such argument (see \cite[Eq. 23]{CCI:2019.jmpa}), by  using the fact that  the energy
$$
w\in GSBD^2(O)\mapsto \int_O \C\str{w}:\str{w}dx
$$
is quadratic with respect to the $\str{w}$ and hence, we have convexity and we can employ \emph{Cauchy-Schwartz inequality} for positive semidefinite bilinear forms to obtain
$$
\begin{aligned}
& \int_{\red O \ba} \C(x)\str{\tilde v}:\str{\tilde v}dx \le \int_{\red O \ba} \C(x)\str{\tilde u}:\str{\tilde u}dx + 2\int_{\red O \ba} \C(x)\str{\tilde v}:[\str{\tilde v} - \str{\tilde u}]dx\\
\le & \int_{\red O \ba} \C(x)\str{\tilde u}:\str{\tilde u}dx
+ 2\Big[\int_{\red O \ba} \C(x)\str{\tilde v}:\str{\tilde v}dx\Big]^{1/2}\times \\
& \hspace*{6cm}  \times \Big[\int_{\red O \ba} \C(x) [\str{\tilde v} - \str{\tilde u}]:[\str{\tilde v} - \str{\tilde u}]dx\Big]^{1/2}
\end{aligned}
$$
\red for any open set $O\subset Q_1.$ \ba Since the inequality $a^2 \le b^2 + 2ac,$ where $a,b,c\ge0,$ implies\footnote[1]{Note that $a\le b+ 2c$ follows from $a^2 \le b^2 + 2ac$ as it yields $(a-2c)^2\leq a(a-2c)\leq b^2$.} $a\le b+ 2c$, we get
$$
\begin{aligned}
\Big[\int_{\red O \ba} \C(x)\str{\tilde v}:\str{\tilde v}dx\Big]^{1/2}
\le &
\Big[\int_{\red O \ba} \C(x)\str{\tilde u}:\str{\tilde u}dx\Big]^{1/2}\\
& + 2\Big[\int_{\red O \ba} \C(x) [\str{\tilde v} - \str{\tilde u}]:[\str{\tilde v} - \str{\tilde u}]dx\Big]^{1/2}\\
\le & (1 + c\delta^s)\Big[\int_{\red O \ba} \C(x)\str{\tilde u}:\str{\tilde u}dx\Big]^{1/2}
\end{aligned}
$$
so that Eq.23 of \cite{CCI:2019.jmpa} holds also in our setting.

Let $\cV_i$ be the family of all bad squares $Q$ intersecting $\Int{A}$ and $D_i:=\bigcup_{Q\in\cV_i} Q.$  For every $Q\in \cV_i$ we define $I_Q$ as the segment of smallest length connecting $(Q'''\cap \p A) \setminus \cl{Q}$ to $\p Q$ with the convention that $I_Q=\emptyset$ if $(Q'''\cap \p A )\setminus \cl{Q}=\emptyset$ or $ Q\cap \Int{\Omega\setminus A}\ne \emptyset.$
By the definition  of $Q'''$ and $Q,$ $\cH^1(I_Q) \le \frac{\sqrt{2}}{24}\,\cH^1(\p Q).$

Let  
$$
B:=\Big[\big(A\setminus \cl{D_i}\big) \cup \p D_i\Big]\setminus \bigcup\limits_{Q\in \cV_i} I_Q
$$
and
$$
v:=\tilde v\chi_{ Q_1}^{} + \tilde u\chi_{(\Omega\cup \substrate)\setminus  Q_1}.
$$ 
 We claim that $B,$ $v$ and $\tilde \omega$ satisfy the assertions of the lemma.

Indeed, from (a4) applied with $\psi\equiv1$  and $\C=I$ it follows that $v\in GSBD^2(\Int{B};\R^2).$ Moreover, by (a5)  $v\in H^1_\loc(\Int{B};\R^2),$ thus, $(B,v)\in\admissible.$ 
 Let us show that  if $A\in\fA_m$ for some $m\in\N,$ then $B\in\fA_m.$   Indeed, by the construction of $B,$ \red 
for each bad square $Q,$ the dilated square $Q'''$ contains inside ``large'' portions of the boundary $\p A.$ Now if $\p A$ intersects $\overline{Q},$ then $I_Q=\emptyset$ and the modification $[A\setminus \overline Q]\cup \p Q\setminus I_Q$ does not increase the number of boundary components. Otherwise, if $\p A$ does not increase $\overline{Q},$ so that it intersects only $Q'''\setminus \overline{Q},$ then adding a small segment $I_Q$ to connect $\p A \cap [Q'''\setminus \overline{Q}]$ to $\overline Q$ again does not increase the number of boundary components of $[A\setminus \overline Q]\cup \p Q\setminus I_Q$. Now, from the disjointness of the cubes $Q\in\cV_i$ it follows that $B\in\admissible_m$. \ba
Therefore, if $(A,u)\in\admissible_m$ for some $m\in\N,$  then $(B,v\big|_B)\in\admissible_m.$ 

By (a1) it follows that $v\in C^\infty( Q_{1-\sqrt\delta}).$  Moreover, by the definition of $B,$  $A\Delta B\strictlyincluded  Q_{r_h}\setminus  Q_{1-\sqrt{\delta}}$  for some $r_h\in(1-\sqrt{\delta},1)$ such that $D_i\subset Q_{r_h}.$
Also, by (a1) $\supp(\tilde u-\tilde v)\strictlyincluded  Q_1$ so that $\supp(\tilde u-v)\strictlyincluded  Q_1,$ and (1) follows.

 Moreover, 
by the definition of $B,$ $I_Q$ and \eqref{perimeter_bad_cubes}
$$
\begin{aligned}
\cH^1(\p B\setminus\p A)\le & \sum\limits_{Q\in\cV_i}P(Q) + \sum\limits_{Q\in\cV_i} \cH^1(I_Q)\\
\le &\Big(1 + \frac{\sqrt2}{24}\Big)\sum\limits_{Q\in\cV_i} P(Q) \le  c_0\sqrt{\delta} \cH^1(( Q_1\setminus  Q_{1-\sqrt{\delta}})\cap \p A), 
\end{aligned}
$$
where $c_0:=\tilde c_0(1 + \sqrt2/24),$ and  
(2) follows. 

Next, by (a3) $|\omega|\le c_0\delta \cH^1( Q_1\cap  \p A ),$ and  
\begin{align*}
\int_{ Q_1\setminus \omega} |v(x) - \tilde u(x)|^2dx = & \int_{ Q_1\setminus \tilde \omega} |\tilde v(y) - \tilde u(y)|^2dy  \le \tilde c_0\delta^2\int_{ Q_1}|\str{\tilde u}|^2dy\\
 \le &  c_0\delta^2\int_{ Q_1}|\str{\tilde u}|^2(y)dy. 
\end{align*}

Finally, by (a4) and the definition of $v$  (i.e., $v=\tilde v$ in $Q_1$)  for any $\psi\in \Lip( Q_1)$ and $\C\in L^\infty( Q_1)$ satisfying \eqref{elastic_bounda} we have
\begin{align*}
&\int_{ Q_1} \psi(x) \C(x)\str{v}:\str{v}dx  =  \int_{ Q_1} \psi(x) \C(x)\str{\tilde v}:\str{\tilde v}dx \\
\le & \int_{ Q_1} \psi(x) \C(x)\str{\tilde u}:\str{\tilde u}dx + d_3\,(1+\Lip(\psi)) \int_{ Q_1}|\str{\tilde u}|^2dx\\
= & \int_{ Q_1\cap A} \psi(x) \C(x)\str{u}:\str{u}dx + d_3\delta^s\,(1+\Lip(\psi)) \int_{ Q_1\cap A}|\str{u}|^2dx,
\end{align*}
since $\tilde u$ is constant in $ Q_1\setminus A.$ Hence, (4) follows.
\end{proof}

The following proposition is a generalization to our setting of \cite[Theorem 4]{CCI:2019.jmpa} established for the Griffith model. 

\begin{proposition}\label{prop:conti_prop_3.4}
Let $ Q_R(x_0)\subset\Omega$ be a square of side length $2R>0.$ Consider 
sequences $ \{m_h\}\subset\N\cup\{\infty\},$ Finsler 
norms $\{\varphi_h\}$ and ellipticity tensors $\{\C_h\}$ such that  $\{\C_h\}$ is equicontinuous in $\cl{ Q_R(x_0)}$ and there exist $d_3,d_4,d_5>0$ with  
\begin{equation}\label{elastic_bound_1}
d_3M:M\le \C_h(x)M:M \le d_4M:M\quad\text{for all $(x,M)\in\overline{ Q_R(x_0)}\times\mtwo$}, 
\end{equation}
and 
\begin{equation} \label{good_finsler_norms}
d_5\sup\limits_{(x,\nu)\in  Q_{\red R}\times \S^1} \varphi_h(x,\nu) \le \inf\limits_{(x,\nu)\in  Q_{\red R}\times \S^1} \varphi_h(x,\nu),
\end{equation}
and define $\cF_h$ and $\Psi_h$ in $\admissible_{m_h}$ as in \eqref{local_cF} and \eqref{deviation}, respectively, with  
$\varphi_h$, $\C_h$ and $m_h$ 
in places of $\varphi$, $\C$ and $m.$ 
Let $\{(A_h,u_h)\}\subset\admissible_{m_h}$ be such that 
\begin{equation}\label{seq_almost_min}
\lim\limits_{h\to\infty} \Psi_h(A_h,u_h; Q_R(x_0)) =0,
\end{equation}
\begin{equation}\label{set_has_no_boundary}
\lim\limits_{h\to\infty} \cH^1( Q_R(x_0)\cap  \p A_h ) = 0,  
\end{equation}
\begin{equation}\label{energy_bound_uniform}
\sup\limits_{h\ge1} \cF_h(A_h,u_h; Q_R(x_0))=:M <\infty.
\end{equation}
Then there exist $u\in H^1( Q_R(x_0)),$ an elasticity tensor $\C\in C^{0}(\overline{ Q_R(x_0)};\mtwo),$  sequences $\{\xi_j\}\subset (0,1)^2$ of vectors and $\{a_j\}$ of rigid displacements and subsequences  $\{(A_{h_j},u_{h_j})\}$, $\{\varphi_{h_j}\}$ and $\{\C_{h_j}\}$  such that 
\begin{itemize}
\item[(a)] $\C_{h_j}\to \C$ uniformly in $\overline{ Q_R(x_0)}$
and $w_j:=u_{h_j} \chi_{ Q_R(x_0)\cap A_{h_j}}^{} +
\xi_j\chi_{ Q_R(x_0)\setminus A_{h_j}}^{} - a_j \to u$  pointwise
a.e. in $ Q_R(x_0),$
and $\str{w_j}\wk\str{u}$ in $L^2( Q_R(x_0))$  
as $j\to \infty;$

\item[(b)] for all $v \in u+H_0^1( Q_R(x_0))$ 
\begin{equation}\label{local_minimal_ekanku}
\int_{ Q_R(x_0)} \C(y)\str{u}:\str{u}\,dy\le 
\int_{ Q_R(x_0)} \C(y)\str{v}:\str{v}\,dy; 
\end{equation}

\item[(c)]  for any $r\in(0,R]$ 
\begin{equation}\label{functional_conv_o}
\lim\limits_{j\to\infty} 
\cF_h(A_{h_j},u_{h_j}; Q_r(x_0)) = 
\int_{ Q_r(x_0)} \C(x)\str{u}:\str{u}\,dx. 
\end{equation}
\end{itemize}
\end{proposition}

\begin{proof}
Without loss of generality, we suppose $R=1$ and $x_0=0.$ 
Let 
\begin{equation}\label{finsler_bounds_for_eucl}
c_{1,h}:=\inf\limits_{(x,\nu)\in  Q_1\times \S^1} \varphi_h(x,\nu),\qquad  c_{2,h}:=\sup\limits_{(x,\nu)\in  Q_1\times \S^1} \varphi_h(x,\nu); 
\end{equation}
by \eqref{good_finsler_norms} we have $d_5c_{2,h}\le c_{1,h}.$ 
Since 
$\sup_h \cH^1( Q_1\cap  \p A_h )< \infty,$
 by \blue Proposition  \ref{prop:maggi_foliation} \black  for every $h\ge 1$ there exists $\xi_h\in(0,1)^2$ such that 
$$
\cH^1(\{y\in  Q_1\cap \p A_h:\,\, \text{$\tr_{A_h}(u_h) $ exists and 
equals to $\xi_h$ at $y$}\})=0.
$$
Therefore  
\begin{equation}\label{jkjkdkjsdk}
\tilde u_h: = 
\begin{cases}
u_h  & \text{in  $ Q_1\cap A_h,$}\\
\xi_h  & \text{in $ Q_1\setminus A_h$}
\end{cases}  
\end{equation}
belongs to $GSBD^2( Q_1;\R^2)$ with 
$J_{\tilde u_h}\subset  Q_1\cap \p A_h$
and 
\begin{equation}\label{fjjfjf}
\lim\limits_{h\to \infty} \cH^1(J_{\tilde u_h}) = 0 
\end{equation}
in view of  \eqref{set_has_no_boundary}. Further we suppose $\cH^1(J_{\tilde u_h})<1/4$ for any $h\ge1.$

By \cite[Proposition 2]{CCF:2016} and \eqref{elastic_bound_1}, there exist a constant $c$ (depending only on $d_3$) and sequences $\{\tilde \omega_h\}$ of a Lebesgue measurable subsets \bu of \ba $ Q_1$ with $|\tilde \omega_h|\le c\cH^1( Q_1\cap  \p A_h )$ and $\{a_h\}$ of  rigid motions such that 
\begin{equation}\label{korn_poincare0}
\int_{ Q_1\setminus \tilde \omega_h} |\tilde u_h - a_h|^2dx \le c\int_{ Q_1}\C_h(x)\str{\tilde u_h}:\str{\tilde u_h}dx.  
\end{equation}
By  \eqref{elastic_bound_1} and \eqref{energy_bound_uniform}, there exists $u\in L^2( Q_1)$ such that up to a subsequence 
$
(\tilde u_h - a_h) \chi_{ Q_1\setminus \tilde \omega_h} \wk  u 
$
weakly in $L^2( Q_1).$  
Furthermore from \eqref{elastic_bound_1} and \eqref{energy_bound_uniform} we obtain
$$
\sup\limits_{h\ge1} \int_{ Q_1} |\str{\tilde u_h-a_h}|^2\,dx +\cH^1(J_{\tilde u_h}) <\infty,
$$
and hence, by \cite[Theorem 1.1]{ChC:2019_jems} there 
exist  a subsequence still denoted by $\{\tilde u_h-a_h\}$ 
for which the set 
$$
E:=\{y\in  Q_1:\,\,\lim\limits_{h\to\infty}|\tilde u_h(y) -a_h(y)|\to\infty\}
$$
has finite perimeter and $\tilde u\in GSBD^2( Q_1\setminus E;\R^2)$ with 
$\tilde u=0$ in $E$ such that 
\begin{equation}\label{safafaess}
\begin{gathered}
\tilde u_h-a_h \to \tilde u\qquad \text{a.e. in $ Q_1\setminus E$}\\
\str{\tilde u_h-a_h} \wk \str{\tilde u}\qquad
\text{in $L^2( Q_1\setminus E;\mtwo),$}\\
\cH^1(( Q_1\setminus \p^* E)\cap J_{\tilde u})+ \cH^1( Q_1\cap  \p^* E)  = 
\cH^1(J_{\tilde u} \cup \p^*E) \le
\liminf\limits_{h\to+\infty} \cH^1(J_{\tilde u_h}) =0.
\end{gathered} 
\end{equation}
In particular, $P(E,Q_1)=0$ so that by the relative isoperimetric inequality either  $|E|=|Q_1|$ or $|E|=0.$ 
By the definition of $E,$  \eqref{set_has_no_boundary}, the uniform $L^2( Q_1)$-boundedness of $\{(\tilde u_h-a_h)\chi_{ Q_1\setminus\tilde \omega_h}\}$ which is a consequence of \eqref{korn_poincare0} and \eqref{energy_bound_uniform},  and  Fatou's Lemma  it follows that  $|E|=0.$   Hence, from \eqref{safafaess} we get $
\tilde u_h-a_h \to \tilde u$ 
a.e. in $ Q_1$ and $\str{\tilde u_h-a_h} \wk \str{\tilde u}$ in $L^2( Q_1;\mtwo),$ and  all  relations in \eqref{safafaess}  hold  in $ Q_1$ and $\tilde u = u$ a.e.\ in $ Q_1.$  In particular, since $\cH^1(J_u)=0,$ by 
\blue Proposition \ref{lem:gcbd_with_zero_jumps} \black we have that $u\in H^1( Q_1;\R^2)$. In view of the fact that  our elastic energy is invariant under rigid deformations,  we suppose $a_h=0$ for any $h\ge1.$

Next we prove \eqref{local_minimal_ekanku}. Let $v\in H^1( Q_1;\R^2)$ be such that $\supp(u-v)\strictlyincluded  Q_r$ for some $r\in(0,1).$ 
Let $\psi\in C_c^1( Q_r;[0,1])$ be a cut-off function with 
$\{0<\psi<1\}\subset \{u=v\}\cap  Q_{r'}$ and $\supp(u-v)\subseteq \{\psi\equiv 1\}\subseteq  Q_{r''}$ for some $r''<r'<r.$ 
By \eqref{fjjfjf} and Lemma \ref{lem:approximation_chambolle} applied with $(A_h,u_h)$ and $ Q_r$ there exist $\tilde v_h\in GSBD^2(\Ins{\Omega};\R^2),$ $B_h\in \fA_{m_h}$ with $(B_h,\tilde v_h\big|_{B_h})\in\admissible_{m_h},$ $ r_h\in(r(1-\sqrt{\delta_h}),r)$ and a Lebesgue measurable set $\omega_h\strictlyincluded  Q_r$ such that 
\begin{itemize}
\item[(a1)] $\tilde v_h\in C^\infty( Q_{r(1-\sqrt{\delta_h})}),$ $ A_h\Delta B_h\strictlyincluded   Q_{r_h}\setminus  Q_{r(1-\sqrt{\delta_h})}$ and $\supp(\tilde u_h-\tilde v_h)\strictlyincluded  Q_r;$

\item[(a2)]  $\cH^1(\p B_h\setminus \p A_h)  \le c_0\,\sqrt{\delta_h}\,\cH^1([ Q_r\setminus  Q_{r(1-\sqrt{\delta_h})}]\cap \p A_h );$ 

\item[(a3)] $|\omega_h|\le c_0 \delta_h\,\cH^1( Q_r\cap \p A_h )$ and  
$$
\int_{ Q_r\setminus \omega_h} |\tilde v_h-\tilde u_h|^2dx \le c_0\delta_h^2r^2\,\int_{ Q_r\cap A_h}|\str{u_h}|^2dx;
$$

\item[(a4)] for any $\eta\in \Lip( Q_r;[0,1])$  
\begin{align}\label{eta_constant}
\int_{ Q_r} \eta \C_h\str{\tilde v_h}:\str{\tilde v_h}dx \le& \int_{ Q_r\cap A_h} \eta \C_h\str{u_h}:\str{u_h}dx\nonumber \\
&+ d_3\,\delta_h^s\,(1+r\,\Lip(\eta)) \int_{ Q_r\cap A_h}|\str{u_h}|^2dx, 
\end{align} 
\end{itemize}
where $\delta_h:=r^{-1/2}\cH^1( Q_r\cap  \p A_h )^{1/2}\to0,$ and $d_3$ and $s$ are constants.  We assume that $h$ is large enough so that $r_h>r'.$  
Set 
$$
v_h:=(1-\psi) \tilde v_h+ \psi v. 
$$
  We observe that $\supp (u_h - v_h\big|_{B_h})\strictlyincluded  Q_r:$ by (a1) and the definition of $\psi,$ there exists $r_0\in(r_h,r)$ such that $A_h\setminus Q_{r_0} = B_h\setminus Q_{r_0}$ and $\tilde u_h = \tilde v_h=v_h$ in $Q_r\setminus Q_{r_0}$ and hence, $u_h\big|_{Q_r\cap A_h\setminus Q_{r_0}} = \tilde u_h\big|_{Q_r\cap A_h\setminus Q_{r_0}} = v_h\big|_{Q_r\cap B_h\setminus Q_{r_0}}.$  
Thus, $(B_h,v_h)$ is an admissible configuration in \eqref{minimal_with_dirixle} and from \eqref{seq_almost_min} and the definition of deviation it follows that 
\begin{equation}\label{kdlfjakd}
\cF_h(A_h,u_h; Q_1) \le \cF_h(B_h,v_h; Q_1) +o(1), 
\end{equation}
where $o(1)\to0$ as $h\to\infty.$ We observe that
\begin{align*}
\cS_h(B_h; Q_1) - \cS_h(A_h; Q_1) \le & \cS_h(B_h; Q_r\setminus \cl{  Q_{r(1-\sqrt{\delta_h})}})   - \cS_h(A_h; Q_r\setminus \cl{  Q_{r(1-\sqrt{\delta_h})}})  \\
 \le  &  \int_{(\p^*B_h\setminus \p^*A_h) \cap Q_r\setminus \cl{  Q_{r(1-\sqrt{\delta_h})}} } \varphi(x,\nu_{B_h}) d\cH^1\\
 &  + 2 \int_{(Q_r\setminus \cl{  Q_{r(1-\sqrt{\delta_h})}})\cap (B_h^{(1)}\cup B_h^{(0)})\cap  (\p B_h\setminus \p A_h) } \varphi(x,\nu_{A_h}) d\cH^1 
\\
\le & 2c_{2,h} \cH^1(\p B_h\setminus \p A_h)
\le 2c_0c_{2,h}\sqrt{\delta_h} \cH^1([ Q_r\setminus  Q_{r(1-\sqrt{\delta_h})}]\cap \p A_h ) \\
\le  &
\frac{2c_0\sqrt{\delta_h}}{d_5}\cS_h(A_h; Q_1) =o(1) 
\end{align*}
as $h\to+\infty,$   where we used in the first inequality (a1), in the second the definition  and nonnegativity of $\cS_h$, in the third  \eqref{finsler_bounds_for_eucl}, in the fourth (a2)  in the last again \eqref{finsler_bounds_for_eucl} and the definition of $\cS_h,$ and finally in the equality we used \eqref{energy_bound_uniform}.  
Thus, \eqref{kdlfjakd} is rewritten as 
\begin{equation}\label{kdlfjakd1}
\cW_h(A_h,u_h; Q_1) \le \cW_h(B_h,v_h; Q_1) +o(1). 
\end{equation}

Note that
by (a1), (a3), \eqref{fjjfjf}, \eqref{safafaess} and Fatou's Lemma,  $\tilde v_h\chi_{Q_r\setminus \omega_h}\to u$ a.e.\ in $ Q_r$ and by (a3) $\chi_{Q_r\setminus \omega_h} \to 1$ a.e.\ in $Q_r.$ Therefore, for a.e.\ $x\in Q_r$ there exists $h_x\ge1$ such that $\chi_{Q_r\setminus \omega_h}(x)=1$  for every $h>h_x$ and $\tilde v_h(x)=\tilde v_h(x)\chi_{Q_r\setminus \omega_h}(x)\to u(x).$ So
\begin{equation}\label{a.e._converge}
\text{$\tilde v_h\to u$ a.e.\ in $Q_r.$} 
\end{equation}

We claim that $\tilde v_h\to u$ strongly in $L_\loc^2( Q_r).$   To see this we fix $\rho\in(0,r),$ and, since $\delta_h\to0$  by (a1),   there exists $h_\rho\ge1$ such that  $\tilde v_h\in H^1(Q_\rho)$ for every $h>h_\rho.$ 
From \eqref{elastic_bound_1}, \eqref{energy_bound_uniform}  and  \eqref{eta_constant} as well as the Korn-Poincar\'e inequality 
$$
\sup\limits_{h>h_\rho}  \|\tilde v_h - b_h\|_{H^1(Q_\rho)} <\infty 
$$
for some sequence $\{b_h\}$ of rigid displacements. 
On the one hand, by Rellich-Kondrachov Theorem there exist $z\in H^1( Q_\rho;\R^2)$ and not relabelled subsequence such that 
$\tilde v_h - b_h \to z$ in $L^2( Q_\rho;\R^2)$ and a.e.\ in $Q_\rho.$ On the other hand, by \eqref{a.e._converge}  $b_h = \tilde v_h - (\tilde v_h -b_h)$ converges to $b:=u-z$ a.e.\ in $Q_\rho.$ Since $b_h$ is a rigid displacement, so is $b$ and hence $b_h\to b$ uniformly in $Q_\rho.$ Therefore, 
$$
\limsup\limits_{h\to\infty} \|\tilde v_h - u\|_{L^2( Q_\rho)}^{} 
\le  
\limsup\limits_{h\to\infty} \|\tilde v_h -b_h - z\|_{L^2( Q_\rho)}^{} 
+
\limsup\limits_{h\to\infty} \|b_h -b\|_{L^2( Q_\rho)}^{} =0,  
$$
and the claim follows.

Since $u=v$ out of $\{\psi=1\},$ the claim implies 
$\tilde v_h\to v$  strongly in $L^2(\{0<\psi<1\}),$ and hence,
\begin{align}\label{sgsgsg}
\lim\limits_{h\to\infty} \int_{ Q_r} |\nabla \psi\odot (v- \tilde v_h)\big|_{A_h}|^2 \le & 
\liminf\limits_{h\to\infty}  \int_{\{0<\psi <1\}} |\nabla \psi\odot (v- \tilde v_h)|^2 = 0,
\end{align}
where $X\odot Y = (X\otimes Y +Y\otimes X)/2,$
Thus, by the definition of $v_h$  and the equality 
$$
\str{v_h} = (1-\psi)\str{\tilde v_h} + \psi \str{v} + 
\nabla \psi\odot (v- \tilde v_h), 
$$
we estimate
\begin{align}\label{rtetyeye}
&\int_{ Q_r}\C_h \str{v_h}:\str{v_h}dx\nonumber\\
=
& \int_{ Q_r} (1-\psi)^2\C_h \str{\tilde v_h}:\str{\tilde v_h}dx
+ \int_{ Q_r}\psi^2\C_h \str{v}:\str{v}dx\nonumber \\
& + \int_{ Q_r} \C_h (\nabla \psi\odot (v- \tilde v_h)) :( \nabla \psi\odot (v- \tilde v_h))dx\nonumber \\
& + 2\int_{ Q_r} (1-\psi) \C_h \str{\tilde v_h}:(\nabla \psi\odot (v- \tilde v_h))dx\nonumber \\
& + 
2\int_{ Q_r} \psi \C_h\str{v}:
(\nabla \psi\odot (v- \tilde v_h))dx\nonumber \\
=& \int_{ Q_r} (1-\psi)^2\C_h \str{\tilde v_h}:\str{\tilde v_h}dx
+ \int_{ Q_r}\psi^2\C_h \str{v}:\str{v}dx +
o(1)\nonumber\\
\le & 
\int_{ Q_r\cap A_h} (1-\psi)^2\C_h \str{u_h}:\str{u_h}dx 
+ \int_{ Q_r}\psi^2\C_h \str{v}:\str{v}dx + o(1), 
\end{align}
where in the second equality we use \eqref{energy_bound_uniform}, \eqref{eta_constant} with $\eta\equiv 1,$ \eqref{sgsgsg}, \eqref{elastic_bound_1} and   the H\"older inequality, while in the last inequality we use \eqref{eta_constant} with $\eta=(1-\psi)^2$ and \eqref{jkjkdkjsdk}. 
Now  \eqref{kdlfjakd1}, \eqref{rtetyeye} and \eqref{jkjkdkjsdk} imply
\begin{align}\label{djksdfkj}
\int_{ Q_r}(2\psi -\psi^2)\C_h \str{\tilde u_h}:\str{\tilde u_h}dx \le 
\int_{ Q_r}\psi^2\C_h \str{v}:\str{v}dx + o(1). 
\end{align}
Since $\{\C_h\}$ is equibounded (see \eqref{elastic_bound_1}) and equicontinuous, by the Arzela-Ascoli Theorem, there exist  a (not relabelled) subsequence and an elasticity tensor $\C\in C^0( Q_1;\mtwo)$ such that  $\C_h\to \C$ uniformly in $ Q_1.$ Hence, 
letting $h\to\infty$ in \eqref{djksdfkj} and using the convexity of the elastic energy  and \eqref{safafaess}, we obtain
\begin{equation}\label{shapatoyoqcha}
\int_{ Q_r}(2\psi -\psi^2) \C(y)\str{u}:\str{u}dy \le 
\int_{ Q_r} \psi^2 \C(y)\str{v}:\str{v}\,dy.
\end{equation}
By the choice of $\psi,$ \eqref{shapatoyoqcha} implies 
\begin{equation}\label{ajsdhjss}
\int_{ Q_{r''}} \C(y)\str{u}:\str{u}dy \le 
\int_{ Q_r} \C(y)\str{v}:\str{v}\,dy. 
\end{equation}
Since $r''$ is arbitrary, letting $r''\nearrow r$ we deduce that \eqref{ajsdhjss} holds also with $r''=r.$   Since $\supp(u-v)\strictlyincluded  Q_r$, this implies \eqref{local_minimal_ekanku}. 

It remains to prove \eqref{functional_conv_o}.
If we take $v=u$ in \eqref{djksdfkj} and use $0\le \psi\le1$ and $\psi=1$ in $ Q_{r''}$
we get  
\begin{align*} 
&\int_{ Q_{r''}}\C \str{u}:\str{u}dx\le 
\liminf\limits_{h\to\infty} \int_{ Q_{r''}}\C_h \str{\tilde u_h}:\str{\tilde u_h}dx \nonumber \\
&\le 
\limsup\limits_{h\to\infty} \int_{ Q_{r''}}\C_h \str{\tilde u_h}:\str{\tilde u_h}dx \le
\int_{ Q_r}\C \str{u}:\str{u}dx. 
\end{align*}
Since $r''$ is arbitrary, letting $r''\nearrow r $ we deduce 
\begin{equation}\label{elastcic_coveee} 
\lim\limits_{h\to\infty} \int_{ Q_r}\C_h \str{\tilde u_h}:\str{\tilde u_h}dx =
\int_{ Q_r}\C \str{u}:\str{u}dx.  
\end{equation}

Now we prove that 
\begin{equation}\label{surface_jugoldi}
\lim\limits_{h\to\infty}  \cS_h(A_h; Q_r) =0 
\end{equation}
for any $r\in(0,1).$
By \eqref{set_has_no_boundary}, we can find $h_r>0$ such that 
\begin{equation}\label{length_small_Ur}
  \cH^1( Q_1\cap  \p A_h )<(1-r)/5 
\end{equation}
for any $h>h_r,$ and hence 
there is no connected component of $\p A_h$ intersecting both $\p  Q_r$ and $\p  Q_1.$ Also by the relative isoperimetric inequality, passing to further subsequence we suppose that either 
\begin{equation}\label{set_disappear}
\lim\limits_{h\to\infty} | Q_1\cap A_h| =0 
\end{equation}
or 
\begin{equation}\label{complement_disappear}
\lim\limits_{h\to\infty} | Q_1\setminus A_h| =0. 
\end{equation}

First assume that \eqref{set_disappear} holds. Let $E_h\subset A_h$  be the set consisting of all connected components of $\cl{A_h}$ not intersecting $\p  Q_1.$ 
Then,   $(A_h\setminus E_h,u_h\big|_{A_h\setminus E_h})$  is an admissible configuration in \eqref{minimal_with_dirixle},  
thus, 
\begin{equation}\label{tsfde}
\cF_h(A_h,u_h;  Q_1) \le \Phi_h(A_h,u_h; Q_1) +o(1) \le \cF_h(A_h\setminus E_h,u_h;  Q_1) + o(1), 
\end{equation}
where in the first inequality we use \eqref{seq_almost_min} and in the second we use the definition of $\Phi_h.$  Hence, 
$$
\begin{aligned}
\cS(A_h; Q_r)\le & \cS(E_h;Q_1) = \cS_h(A_h; Q_1) - \cS_h(A_h\setminus E_h; Q_1)  \\
\le & \cF_h(A_h; Q_1) - \cF_h(A_h\setminus E_h; Q_1) \le  o(1), 
\end{aligned}
$$
where we used in the first inequality the definition of $E_h,$ which entitles that $U_r\cap  \p A_h \subset  \p E_h $, in the equality the disjointness of $\cl{A_h\setminus E_h}$ and $\cl{E_h}$ which follows by \eqref{length_small_Ur},  and in the second inequality the nonnegativity of the elastic energy and in the third \eqref{tsfde}.
Hence, \eqref{surface_jugoldi} follows.

Now assume that \eqref{complement_disappear} \blue  holds and let $\delta_h:=  r^{-1/2}\sqrt{\cH^1(Q_r\cap  \p A_h )}\to0.$  \black 
Fix any $\rho\in(0,r).$ By \eqref{set_has_no_boundary}, we can find $h_{r,\rho}>0$ such that 
$\delta_h<\min\{1-r,r-\rho\}/5$ for any $h>h_{r,\rho}.$ Since $ A_h\in\fA_{m_h},$ no connected component of $\p A_h$ \bu intersects \ba both $\p  Q_r$ and $\p  Q_\rho.$ 
Let $ F_h\subset  Q_1\setminus A_h$ be the union of all connected components of $\cl{ Q_1\setminus A_h}$ lying strictly inside $ Q_1$ (so $ F_h$ is a union of ``holes'' and  $ \p F_h\subset \p A_h$). 
Let $\psi\in C_c^1( Q_r;[0,1])$ be a cut-off function with  \blue $\{0<\psi<1\}\subset Q_{r'}$ and $\{\psi\equiv 1\}\subseteq  Q_{r''}$ \black for some $r''<r'<r.$ Set $A_h':=A_h\cup \overline{ F_h}.$ Applying Lemma  \ref{lem:approximation_chambolle} with $(A_h',\tilde u_h\big|_{A_h'}),$ $ Q_r$ and $m=m_h$ we find $\tilde v_h'\in GSBD^2(\Ins{\Omega};\R^2),$  $B_h'\in \fA_{m_h}$ with $(B_h',\tilde v_h'\big|_{B_h})\in\admissible_{m_h},$ $ r_h\in(r(1-\sqrt{\delta_h}),r)$ and a Lebesgue measurable set $\omega_h'\strictlyincluded  Q_r$ such that 
\begin{itemize}
\item[(b1)] $\tilde v_h'\in C^\infty( Q_{r(1-\sqrt{\delta_h})}),$ $A_h'\Delta B_h'\strictlyincluded   Q_{r_h}\setminus  Q_{r(1-\sqrt{\delta_h})}$ and $\supp(\tilde u_h-\tilde v_h')\strictlyincluded  Q_r;$

\item[(b2)]  $\cH^1(\p B_h'\setminus \p A_h')  \le c_0\,\sqrt{\delta_h}\,\cH^1([ Q_r\setminus  Q_{r(1-\sqrt{\delta_h})}]\cap \p A_h' );$ 

\item[(b3)] $|\omega_h'|\le c_0 \delta_h\,\cH^1( Q_r\cap \p A_h')$ and  
$$
\int_{ Q_r\setminus \omega_h'} |\tilde v_h'-\tilde u_h|^2dx \le c_0\delta_h^2r^2\,\int_{ Q_r\cap A_h'}|\str{u_h}|^2dx;
$$

\item[(b4)] for any $\eta\in \Lip( Q_r;[0,1])$  
\begin{align*} 
\int_{ Q_r} \eta \C\str{\tilde v_h'}:\str{\tilde v_h'}dx \le& \int_{ Q_r\cap A_h} \eta \C\str{u_h}:\str{u_h}dx\nonumber \\
&+ d_3\,\delta_h^s\,(1+r\Lip(\eta)) \int_{ Q_r\cap A_h}|\str{u_h}|^2dx, 
\end{align*} 
\end{itemize}
where  $d_3$ and $s$ are constants.
Set 
\begin{equation*} 
v_h':=(1-\psi) \tilde v_h'+ \psi u. 
\end{equation*}
By the definition of $A_h' $ and (b1)  $(B_h',v_h'\big|_{B_h'})$ is an admissible configuration for $\Phi_h(A_h,u_h; Q_1)$ in \eqref{minimal_with_dirixle}. Thus from \eqref{seq_almost_min} and \eqref{complement_disappear} 
\begin{equation}\label{energy_difference} 
\cF_h(A_h,u_h; Q_1) \le \cF_h(B_h',v_h'\big|_{B_h'}; Q_1) +o(1). 
\end{equation}
Now as in the proof of \eqref{djksdfkj}  
\begin{align}\label{elastic_energy_difference}
&\cW_h(B_h',v_h'\big|_{B_h'}; Q_1) - \cW_h(A_h,u_h; Q_1)\nonumber \\
\le & 
\int_{ Q_r}\psi^2\C_h \str{u}:\str{u}dx -\int_{ Q_r}(2\psi -\psi^2)\C_h \str{\tilde u_h}:\str{\tilde u_h}dx  + o(1) \nonumber \\
\le &  \int_{ Q_r} \C_h \str{u}:\str{u}dx -  \int_{ Q_{r''}} \C_h \str{\tilde u_h}:\str{\tilde u_h}dx + o(1). 
\end{align}
Moreover, 
\begin{align}\label{surface_energy_difference}
& \cS_h(B_h'; Q_1) - \cS_h(A_h; Q_1)=  \Big(\cS_h(B_h'; Q_1) - \cS_h(A_h'; Q_1) \Big) + \Big(\cS_h(A_h'; Q_1) - \cS_h(A_h; Q_1) \Big)  \nonumber \\
\le & \cS_h(B_h'; Q_r\setminus  Q_{r(1-\sqrt{\delta_h})}) - \cS_h(A_h; Q_\rho) \le 2c_{2,h} \cH^1(\p B_h'\setminus \p A_h')- \cS_h(A_h; Q_\rho) \nonumber \\
\le & 2c_0c_{2,h}\sqrt{\delta_h} \cH^1([ Q_r\setminus  Q_{r(1-\sqrt{\delta_h})}]\cap \p A_h) - \cS_h(A_h; Q_\rho) \nonumber \\
\le & \frac{2c_0\sqrt{\delta_h}}{d_5}\cS_h(A_h; Q_1) - \cS_h(A_h; Q_\rho) = o(1) - \cS_h(A_h; Q_\rho),
\end{align}
where  we used in the first inequality (b1) and the definition of $A_h',$ in the second and in the last inequalities the definition of  $\cS_h,$ \eqref{finsler_bounds_for_eucl} and \eqref{good_finsler_norms},  in the third inequality (b2), and in the last equality \eqref{energy_bound_uniform} and that $\delta_h\to0$ by \eqref{set_has_no_boundary}. 
Hence, \eqref{energy_difference}, \eqref{elastic_energy_difference} and   \eqref{surface_energy_difference} imply 
\begin{align*} 
\cS_h(A_h; Q_\rho) +  \int_{ Q_{r''}} \C_h \str{\tilde u_h}:\str{\tilde u_h}dx  \le \int_{ Q_r} \C_h \str{u}:\str{u}dx + o(1). 
\end{align*}
Thus, letting $h\to\infty$ and using \eqref{elastcic_coveee} we get 
$$ 
\limsup\limits_{h\to\infty} \,\cS_h(A_h; Q_\rho) + \int_{ Q_{r''}} \C \str{u}:\str{u}dx  \le \int_{ Q_r} \C \str{u}:\str{u}dx.
$$ 
Now letting $r''\to r$ we get 
\begin{equation}\label{shapatellomxss}
\limsup\limits_{h\to\infty} \cS_h(A_h; Q_\rho) =0. 
\end{equation}
Observe that the function $B\mapsto \cS_h(A_h;B)$ defined for Borel sets $B\subset  Q_1$ extends to a bounded  nonnegative  Radon measure $\mu_h$ in $ Q_1$. Since \eqref{shapatellomxss} holds for any $\rho\in(0,r),$ $\mu_h$ converges to $0$ in the weak* sense, and thus \eqref{surface_jugoldi} follows. 
\end{proof}

Recall that by \cite[Proposition 3.4]{CFI:2018ccm} 
if the elasticity tensor $\C$ is constant,
then for any $\gamma\in(0,2)$ there exists 
$c_\gamma:=c_\gamma(c_3,c_4)>0$ such that 
for every  local minimizer $(\Omega,u)\in \admissible$ 
of $\cF(\cdot;\openset),$  $u$ is analytic in $\openset$ 
and for any square
$ Q_R(x)\strictlyincluded \openset$ and $r\in(0,R),$
\begin{equation}\label{c_gamma}
\int_{ Q_r(x)}\C\str{u}:\str{u}\,dx \le 
c_\gamma\, \Big(\frac{r}{R}\Big)^{2-\gamma}
\int_{ Q_R(x)} \C\str{u}:\str{u}\,dx. 
\end{equation}
Given $\gamma\in(0,1)$ let 
\begin{equation*} 
\tau_0=\tau_0(\gamma,c_3,c_4):=\min\{1,\tfrac12 c_{\gamma}^{-\frac{1}{4-2\gamma}}\}, 
\end{equation*}
where  $c_{\gamma}$ is the constant appearing in \eqref{c_gamma}.
Using Proposition \ref{prop:conti_prop_3.4} and repeating similar arguments of \cite{ChC:2019_arxiv,CFI:2018poincare} we get the following  decay property of the functional $\cF$.

\begin{proposition}\label{prop:functional_decay}
For any $\tau\in(0,\tau_0)$ there exist  
$\varsigma = \varsigma(\tau)\in(0,1)$ and $\vartheta:=\vartheta(\tau)\in(0,1)$ with the following property: If there exist $m\in \N\cup \{\infty\},$ $(A,u)\in\admissible_m$ and a square $Q_\rho(x)\strictlyincluded \Omega$ such that
$$
\cH^1( Q_\rho(x)\cap \p A) \le 2\varsigma \rho \quad \text{and}\quad
\cF(A,u; Q_\rho(x)) \le (1+\vartheta)\Phi(A,u; Q_\rho(x)),  
$$
then 
$$
\cF(A,u; Q_{\tau \rho}(x)) \le  \tau^{2-\gamma}\cF(A,u; Q_\rho(x)).
$$
\end{proposition}

\begin{proof}
  We argue by contradiction. Assume that there exists $\tau\in(0,\tau_0)$ such that for all $\varsigma,\vartheta\in(0,1)$ we can find $m:=m(\varsigma,\vartheta)\in\N\cup\{\infty\},$ $(A,u):=(A(\varsigma,\vartheta),u(\varsigma,\vartheta))\in \admissible_m$  and $Q_\rho(x)\strictlyincluded\Omega$ with $\rho:=\rho(\varsigma,\vartheta)$ and $x:=x(\varsigma,\vartheta)$  satisfying
\begin{align}\label{telefon}
\cH^1( Q_{\rho}(x)\cap \p A) \le 2\varsigma\rho \quad \text{and} \quad\cF(A,u_; Q_{\rho}(x)) \le (1 + \vartheta)\Phi(A,u; Q_{\rho}(x)),
\end{align} 
but 
\begin{equation}\label{papae}
\cF(A,u; Q_{\tau\rho}(x))>\tau^{2-\gamma} \cF(A,u; Q_{\rho}(x)).
\end{equation}
Let us choose any positive real numbers $\varsigma_h,\vartheta_h\to0,$ and denote for simplicity $m_h:=m(\varsigma_h,\vartheta_h),$ $(A_h,u_h)= (A(\varsigma_h,\vartheta_h),u(\varsigma_h,\vartheta_h)),$ $\rho_h:=\rho(\varsigma_h,\vartheta_h),$ $x_h=x(\varsigma_h,\vartheta_h).$ By \eqref{telefon} and \eqref{papae},    
\begin{align}
& \cH^1( Q_{\rho_h}(x_h)\cap \p A_h) \le 2\varsigma_h\rho_h, \label{contradiction_ass1}\\
& \cF(A_h,u_h; Q_{\rho_h}(x_h)) \le (1 + \vartheta_h)\Phi(A_h,u_h; Q_{\rho_h}(x_h)), \label{contradiction_ass2} 
\end{align}
but 
\begin{equation}
\cF(A_h,u_h; Q_{\tau\rho_h}(x_h))>\tau^{2-\gamma} \cF(A_h,u_h; Q_{\rho_h}(x_h))  \label{contradiction_ass3} 
\end{equation}
for any $h.$ Note that $\cF(A_h,u_h; Q_{\rho_h}(x_h))>0.$ Let us define the rescaled energy $\cF_h(\cdot; Q_1):\admissible_{m_h}\to\R$ as in \eqref{local_cF} with 
$$
\varphi_h(y,\nu):= \frac{\rho_h\varphi(x_h+\rho_hy,\nu)}{ \cF(A_h,u_h; Q_{\rho_h}(x_h))}
$$
in place of $\varphi(y,\nu)$ and 
$$
 \C_h(y):=\C(x_h+\rho_hy)
$$
in place of  $\C(y)$, for $y\in  Q_1$.
We notice that
\begin{equation}\label{decay_bound}
\cF_h(E_h,v_h; Q_1)=1
\end{equation}
for 
$$
E_h:= \sigma_{x_h,\rho_h}(A_h)
$$
(see definition of blow-up map $\sigma_{x,r}$ at \eqref{blow_ups}) and
$$
  v_h(y):=\frac{u_h(x_h+\rho_hy)}{\sqrt{\cF(A_h,u_h;B_{\rho_h}(x_h))}}.
$$
By \eqref{contradiction_ass1} we obtain 
$$
\cH^1( Q_1\cap \p E_h) <2\varsigma_h
$$
while \eqref{contradiction_ass2} and  \eqref{decay_bound} entails
$$
\Psi_h(E_h,v_h; Q_1) \le \vartheta_h \Phi_h(E_h,v_h; Q_1) \le \vartheta_h \cF_h(E_h,v_h; Q_1)=\vartheta_h,
$$
where $\Phi_h$ and $\Psi_h$ are defined as in \eqref{minimal_with_dirixle} and \eqref{deviation} (again with $\varphi_h$ and $\C_h$ in places of $\varphi$ and $\C, $ respectively).  By \eqref{hyp:elastic} $\{\C_h\}$ is equibounded. Since $\Omega$ is bounded, there exists $x_0\in \cl{\Omega}$ such that, up to extracting a subsequence, $x_h\to x_0$ as $h\to+\infty.$ As $\rho_h\to0,$ one has $x_h+ \rho_h y\to x_0$ for every $y\in \cl{ Q_1}.$  Thus $\{\C_h\}$ is also equicontinuous and $\C_h\to \C_0:=\C(x_0)$ uniformly in $ \cl{ Q_1}.$ 
In view of  \eqref{contradiction_ass1}, \eqref{contradiction_ass2} and \eqref{decay_bound}, we can apply Proposition \ref{prop:conti_prop_3.4}  to find $\bu v\ba\in H^1( Q_1;\R^2),$ vectors $\xi_h\in (0,1)^2$, and infinitesimal rigid displacements $a_h$ such that, up to a subsequence, 
$$
w_h:=v_h\chi_{ Q_1\cap E_h}^{} + \xi_h\chi_{ Q_1\setminus E_h}^{} - a_h \to v
$$ 
pointwise a.e.\ in $ Q_1$, $\str{w_h} \wk \str{v}$ in $L^2( Q_1)$ as $h\to+\infty,$
and  
\begin{equation}\label{decay_limit}
\lim\limits_{h\to+\infty} \cF_h(E_h,w_h; Q_r) =  \lim\limits_{h\to+\infty} \cF_h(E_h,v_h; Q_r) =\int_{ Q_r} \C_0(x)\str{v}:\str{v}dx 
\end{equation}
for any $r\in(0,1].$
In particular, from \eqref{decay_limit} and \eqref{contradiction_ass3} it follows that
\begin{align*}
\int_{ Q_\tau} \C_0(x)\str{v}:\str{v}dx&=\lim\limits_{h\to+\infty}\cF(E_h,v_h; Q_{\tau})\\
&\geq\lim\limits_{h\to+\infty}\tau^{2-\gamma} \cF(E_h,v_h; Q_{1}) = \tau^{2-\gamma} \int_{ Q_1} \C_0(x)\str{v}:\str{v}dx.
\end{align*}
Since $\C_0$ is constant, applying \eqref{c_gamma} with $r:=\tau$ and $R:=1$ we get  
\begin{align*}
c_\gamma\tau^{2-\gamma}\int_{ Q_1} \C_0(x)\str{v}:\str{v}dx &\ge \int_{ Q_{\tau}}  \C_0(x)\str{v}:\str{v}dx\\
 &\ge  \tau^{\gamma-2}\int_{ Q_1} \C_0(x)\str{v}:\str{v}dx. 
\end{align*}
Now recalling that $\cF_h(E_h,v_h; Q_1)=1,$ by \eqref{decay_limit} we get $\int_{ Q_1} \C_0(x)\str{v}:\str{v}dx =1,$ thus, $\tau^{2-\gamma}\ge c_\gamma^{-1/2}>\tau_0^{2-\gamma},$  a contradiction.
\end{proof}

By employing the arguments of \cite[Section 4.3]{P:2012} and using Proposition \ref{prop:functional_decay}
we establish the following lower bound for $\cF$.  

\begin{proposition} \label{prop:lower_density_cG}
Given $\tau\in (0,\tau_0)$, let $\varsigma:=\varsigma(\tau)\in(0,1)$ and $\vartheta:=\vartheta(\tau)\in(0,1)$ be as in Proposition \ref{prop:functional_decay}. Let $(A,u)\in\admissible_m$ be a $(\Lambda,m)$-minimizer of $\cF$ in $ Q_{r_0}(x_0)$ for some $ m\in\N\cup\{\infty\}$ and $r_0>0,$ and let 
$$
J_A^*:=\{y\in  Q_{r_0}(x_0) \cap  \p A:\,\, \theta_*(\p A,y)>0\}. 
$$
 Then\bu, \ba
\begin{equation}\label{lower_dens_functiona}
\cF(A,u; Q_\rho(x)) \ge 2c_1\varsigma\rho
\end{equation}
for every $ x\in  \overline{J_A^*}$ and for every square $ Q_\rho(x)\subset  Q_{r_0}(x_0)$  with $\rho\in(0,R_0),$ where 
$$
R_0:=R_0(r_0,\Lambda,c_1,\tau):=\min\Big\{r_0,\frac{\sqrt{\pi}\,c_1\vartheta}{\Lambda(2+\vartheta)}\Big\}.
$$ 
\end{proposition}

\begin{proof}
 Fix $m\in\N\cup\{\infty\}.$ 
Note that for any $(C,w),(D,v)\in\admissible_m$ and $\openset\subset\Omega$ with $C\Delta D\strictlyincluded \openset$ 
\begin{align}\label{relate_isopp}
\sqrt{4\pi}\,|C\Delta D|^{1/2} \le  & \cH^1(\p^*(C\Delta D)) \le \cH^1(\openset\cap \p^* C) + \cH^1(\openset\cap \p^* D)\nonumber\\
\le &\frac{\cS(C,\openset) + \cS(D,\openset)}{c_1} \le \frac{\cF(C,w;\openset) + \cF(D,v;\openset)}{c_1},
\end{align}
where in the first inequality we used the isoperimetric inequality, in the second  $\p^*(C\Delta D) \subset \openset\cap (\p^*C\cup\p^*D),$ in the third \eqref{finsler_norm} and the definition of $\cS(\cdot;\openset)$ and in the last the nonnegativity of $\cW(\cdot;\openset).$ 
Thus, from the $(\Lambda,m)$-minimality of $(A,u)$ in $ Q_{r_0}(x_0)$ we deduce that 
\begin{align}\label{almost_min_dan_ketamiz}
\!\cF(A,u; Q_r(x)) \le & \cF(B,v; Q_r(x)) + \Lambda|A\Delta B|^\frac{1}{2}|A\Delta B|^\frac{1}{2}\nonumber \\
\le & \cF(B,v; Q_r(x)) + \frac{\Lambda r}{\sqrt{\pi}\,c_1}\Big(\cF(A,u; Q_r(x)) + \cF(B,v; Q_r(x))\Big)  
\end{align}
for any $ Q_{\bu r\ba}(x)\subset  Q_{r_0}(x_0)$ and $(B,v)\in \admissible_m$ with $A\Delta B\strictlyincluded  Q_r(x)$ and $\supp(u-v)\strictlyincluded  Q_r(x)$, where in the last inequality we used \eqref{relate_isopp} and the inequality $|A\Delta B| \le | Q_r|=4r^2.$
Let $r>0$ be small enough so that $\frac{\Lambda r}{\sqrt{\pi}\,c_1} \le \frac{\vartheta}{2+\vartheta},$ where $\vartheta:=\vartheta(\tau)\in(0,1)$ is given by Proposition \ref{prop:functional_decay}. From \eqref{almost_min_dan_ketamiz} we obtain 
\begin{equation*}
\cF(A,u; Q_r(x)) \le (1+ \vartheta)\cF(B,v; Q_r(x)), 
\end{equation*}
which by the arbitrariness of $(B,v)$  is equivalent to
\begin{equation}\label{almost_minimal_mish}
\cF(A,u; Q_r(x)) \le (1+ \vartheta)\Phi(A,u; Q_r(x)). 
\end{equation}

Now we prove \eqref{lower_dens_functiona}.
Let $x\in J_A^*.$ For simplicity we suppose that $x=0.$  Assume by contradiction   that for such  $m\in\N\cup\{\infty\},$ $(A,u)\in\admissible_m$ and for some $ Q_\rho\strictlyincluded  Q_{r_0}(x_0)$ with $\rho\in(0,R_0)$  we have  
\begin{equation*} 
\cF(A,u; Q_\rho)< 2c_1\varsigma \rho.
\end{equation*}
Then by the nonnegativity of the elastic energy and 
\eqref{finsler_norm},
\begin{equation*} 
2c_1\varsigma \rho> \cF(A,u; Q_\rho) \ge \int_{ Q_\rho\cap \p A} \varphi(x,\nu_A)d\cH^1\ge 
c_1\cH^1( Q_\rho\cap \p A ) 
\end{equation*}
so that 
\begin{equation}\label{ahahaha1}
\cH^1( Q_\rho\cap \p A  )< 2\varsigma \rho. 
\end{equation}
By \eqref{ahahaha1} and \eqref{almost_minimal_mish} we can apply Proposition \ref{prop:functional_decay}  and obtain that
$$
\cF(A,u; Q_{\tau \rho}) \le \tau^{2-\gamma} 
\cF(A,u; Q_\rho) \le  2c_1\varsigma \tau^{2-\gamma} \rho 
$$
Hence, 
$$
\cH^1( Q_{\tau\rho}\cap \p A) \le  2\varsigma \tau^{2-\gamma} \rho < 2\varsigma \tau \rho,
$$
where we used $\gamma,\tau\in(0,1)$, 
and by induction
$$
\cH^1( Q_{\tau^n\rho}\cap \p A) \le 2\varsigma \tau^{(2-\gamma)n} \rho<2\varsigma \tau^n \rho,\quad n\in\N.
$$
However,  by the choice of $x$  
$$
 0< \theta_*(\p A,x)=   \liminf\limits_{n\to+\infty} 
\frac{\cH^1( Q_{\tau^n\rho}\cap \p A)}{2\tau^n\rho} 
\le \lim\limits_{n\to+\infty}\frac{2c_1\varsigma \tau^{(1-\gamma)n}}{2c_1} = 0, 
$$
a contradiction. This contradiction implies \eqref{lower_dens_functiona} for $x\in J_A^*.$ 

\red Now consider any $x\in Q_{r_0}(x_0)\cap \overline{J_A^*}$ and $\rho\in(0,R_0)$ with $Q_\rho(x)\subset Q_{r_0}(x_0),$ and let us choose a sequence $\{Q_{\rho_k}(x_k)\}$ of squares with $x_k\in J_A^*$ and $\rho_1\le \rho_2\le \ldots \le\rho$ such that 
$$
Q_{\rho_1}(x_1)\subseteq Q_{\rho_2}(x_2)\subseteq \ldots Q_\rho(x)\qquad\text{and}\qquad 
Q_\rho(x) = \bigcup_k Q_{\rho_k}(x_k).
$$
Notice that $x_k\to x$ and $\rho_k\to\rho.$ By De Giorgi-Letta Theorem \cite[Theorem 1.53]{AFP:2000}, both maps
$$
O\mapsto \int_{O\cap \partial^* A} \varphi(x,\nu_A)d\mathcal{H}^1 + 2\int_{O\cap (A^{(0)}\cup A^{(1)})\cap \partial A} \varphi(x,\nu_A)d\mathcal{H}^1 
$$
and 
$$
O\mapsto \int_{O\cap A} \mathbb{C}(y)e(u):e(u)d y,
$$
defined at open sets $O\subset\subset\Omega,$ uniquely extend to positive Borel measures $\mu_1$ and $\mu_2$ in $\Omega.$ Therefore, from the continuity of $\mu_1$ and $\mu_2$ (see e.g. \cite[Remark 1.3]{AFP:2000}) and the validity of \eqref{lower_dens_functiona} with $x_k$ and $\rho_k$ it follows that
\begin{align*}
\cF(A,u;Q_\rho(x)) = & \mu_1(Q_\rho(x)) + \mu_2(Q_\rho(x)) = 
\lim\limits_{k\to+\infty} [\mu_1(Q_{\rho}(x_k)) + \mu_2(Q_{\rho_k}(x_k))] \\
= & \lim\limits_{k\to+\infty} \cF(A,u;Q_{\rho_k}(x_k)) \ge  \lim\limits_{k\to+\infty} (2c_2\varsigma\rho_k) = 2c_2\varsigma\rho.
\end{align*} 
\end{proof}

Now we are ready to prove \eqref{min_seq_density_up} and \eqref{min_seq_density_low}.

\begin{proof}[Proof of Theorem \ref{teo:density_estimates}]
  Let $m\in\N\cup \{\infty\}$ and $(A,u)$ be a $(\Lambda,m)$-minimizer of $\cF(\cdot,\cdot;\Omega).$   
We begin by establishing \eqref{min_seq_density_up}. Let $x\in \Omega$, $r\in \red(0,\min\{1,\dist(x,\p \Omega)\}\ba)$, and $ Q_r:= Q_r(x)$. \bu Since \eqref{min_seq_density_up} is trivial if  $Q_r\cap\partial A=\emptyset$, then we assume that  $Q_r\cap\partial A\neq\emptyset$ and so \ba   $E:=(A\setminus \overline{ Q_r})\cup \p  Q_r \in \fA_m\bu.$ By \ba the $(\Lambda,m)$-minimality of $(A,u)$ 
$$
\cF(A,u; Q_r) \le\cF(E,u; Q_r) + \Lambda | Q_r|. 
$$
Hence, by the nonnegativity $\cW(A\cap Q_r,u; Q_r)$
$$
\int_{ Q_r\cap \p A} \varphi(x,\nu_{A})d\cH^1 \le  2\int_{\p  Q_r} \varphi(x,\nu_{ Q_r})d\cH^1 +4 \Lambda r^2  
$$
and hence \eqref{finsler_norm} entails \eqref{min_seq_density_up}. 
In particular, since 
$E\Delta A\strictlyincluded  Q_\rho$ for every $\rho\in(r,\dist(x,\p \Omega)),$
we also have 
\begin{align*}
\cF(A,u; Q_\rho) \le  & \cF(E,u;  Q_\rho) +\Lambda | Q_r| = \cF(E,u;  Q_\rho\setminus \overline{ Q_r}) 
+\cS(E,u;\overline{ Q_r}) +4\Lambda r^2\\
\le & \cF(E,u;  Q_\rho\setminus \overline{ Q_r}) + 2\int_{\p  Q_r} \varphi(x,\nu_{ Q_r})d\cH^1 + 4\Lambda r^2\\
\le & \cF(E,u;  Q_\rho\setminus \overline{ Q_r}) + 16c_2r + 4\Lambda r^2
\end{align*}
and hence, letting $\rho\searrow r$ and using $r\le1$ we get 
\begin{equation}\label{upper_bound_f}
\cF(A,u;\overline{ Q_r}) \le (16c_2 + 4\Lambda )r. 
\end{equation}

Now   assuming that $x$ belongs to the closure of the set  $\{y\in \Omega\cap \p A:\,\theta_*(\p A,y)>0\},$   we prove \eqref{min_seq_density_low}. 
For $\tau_o:=\tau_0/2,$ let $\varsigma_o=\varsigma(\tau_o)\in(0,1)$ and $R_o=R_0(1,\Lambda,c_1,\tau_o)>0$ be as in Proposition \ref{prop:lower_density_cG}. Then by \eqref{lower_dens_functiona},
\begin{equation}\label{low_bound_f}
\cF(A,u;  Q_{\kappa r}) \ge 2c_1\varsigma_o \kappa r  
\end{equation}
for $\kappa\in(0,1]$ and for any square $ Q_r\subset\Omega$ with $r\in(0,R_o).$  We consider $\varsigma_*:=\varsigma(\tau_*),$ $\vartheta_*:=\vartheta(\tau^*)$, and $R_*:=\min\{R(1,\Lambda,c_1,\tau_*),R_o\}$ as given by Proposition \ref{prop:functional_decay} for
$\tau_*:=\min\{\frac{\tau_0}{2},\big(\frac{c_1\varsigma_o}{16c_2+4\Lambda }\big)^{\frac{1}{1-\gamma}}\}$. By contradiction, if $\cH^1( Q_r\cap \p A)<\varsigma_*r,$ 
then by applying \eqref{almost_minimal_mish} with $\kappa=\tau_*$ we obtain
$$
\cF(A,u; Q_r) \le (1 + \vartheta_*) \Phi(A,u; Q_r).
$$
Then by Proposition \ref{prop:functional_decay},  
$$
\cF(A,u; Q_{\tau_* r}) \le \tau_*^{2-\gamma} \cF(A,u; Q_r) 
$$
so that by \eqref{low_bound_f} and \eqref{upper_bound_f} 
$$
\tau_*^{1-\gamma} \ge \frac{2c_1\varsigma_o}{16c_2 + 4\Lambda }, 
$$
which is a contradiction. 
\end{proof}

\section{Compactness and lower-semicontinuity properties}\label{sec:compact_lsc_property}

 For the convenience of the reader, we divide the prove into several propositions. 
We start by  showing the compactness of free crystal regions of the sequence of constrained minimizers $\{(A_m,u_m)\}$.   
 
\begin{proposition}\label{prop:compact_A_m}
 Assume that either $\fm\in(0,|\Omega|)$ or $\substrate=\emptyset.$  
There exist $m_h\nearrow +\infty,$ $(A_{m_h},u_{m_h})\in \admissible_{m_h}$ and  $A\in  \widetilde{\fA} $ such that  
\begin{itemize}
 \item[\rm(a)] for any $h\in\N$\bu, \ba  $(A_{m_h},u_{m_h})$ is a minimizer of $\cF$ in $\admissible_{m_h}$ with $|A|=\fm$ such that $\p A_{m_h}$ does not contain isolated points; 
 
 \item[\rm(b)] $\sdist(\cdot,\p A_{m_h}) \to \sdist(\cdot,\p A)$ locally uniformly in $\R^2$ as $h\to\infty;$
 
 \item[\rm(c)] for any $x\in \Omega\cap\p A$ and $r\in (0,\min\{R_*,\dist(x,\p\Omega)\})$
\begin{equation}\label{density_pA}
 \frac{c_1\varsigma_*}{8\pi c_2} \le  \frac{\cH^1( Q_r(x) \cap \p A )}{2r} \le \frac{2\pi c_2}{c_1\varsigma_*}, 
\end{equation}
where $\varsigma_*:=\varsigma_*(c_3,c_4)\in(0,1)$ and $R_*:=R_*(c_1,c_2,c_3,c_4)>0$  are given in Theorem \ref{teo:density_estimates}.
\end{itemize}
\end{proposition}

\begin{proof}
 
By \cite[Theorem 2.6]{HP:2019} there exists a minimizer $(A_m,u_m)\in\admissible_m$ for every $m\in\N$.  Without loss of  generality we assume that $\p A_m$ does not contain isolated points. In fact, if $\p A_m$ has a isolated point $x$ in $A_m^{(0)},$ then $A_m\setminus \{x\}\in\fA_m$  and $\cF(A_m,u_m) = \cF(A_m\setminus\{x\},u_m).$ Analogously, if $\p A_m$ has an isolated point in $A_m^{(1)},$ then there exists $r>0$ such that $B_r(x)\cap \p A_m=\{x\}$ (and $B_r(x)\subset A_m\cup\{x\} \in\admissible_m$). In view of Proposition  \ref{lem:gcbd_with_zero_jumps} the function $u_m,$ arbitrarily extended to $x$ belongs to $H_\loc^1(B_r(x)),$ hence, the configuration $(A_m\cup\{x\},u_m)\in\admissible_m$ and satisfies $\cF(A_m,u_m)=\cF(A_m\cup\{x\},u_m).$  

In view of Remark \ref{rem:passage_toU0_teng0} $(A_m,u_m-u_0)$ is a $(\lambda_0,m)$-minimizer of $\cF(\cdot,\cdot;\Omega).$  Moreover, since $\p A_m$ does not contain isolated points $\theta_*(\p A_m,x)>0$ for any $x\in \p A_m,$ hence by 
Theorem \ref{teo:density_estimates} the density estimates \eqref{min_seq_density_up} and \eqref{min_seq_density_low} hold for all $x\in\Omega\cap \p A_m.$ 

By \cite[Proposition 3.1]{HP:2019}, there exist  $A\subset\Omega$ and a 
subsequence $\{(A_{m_h},u_{m_h})\}$ such that 
$\sdist(\cdot,\p A_{m_h})\to \sdist(\cdot,\p A)$ as $h\to\infty.$
Consider the sequence $\mu_h:=\cH^1\res  \p A_{m_h}$ of positive Radon measures. By Theorem \ref{teo:density_estimates}
\begin{equation}\label{dens_esss}
\frac{\varsigma_*}{2} \le \frac{\mu_h( Q_r(x))}{2r} \le \frac{2\pi c_2}{c_1} 
\end{equation}
for every $x\in \Omega\cap \p A_{m_h}$ and $ Q_r(x)\strictlyincluded \Omega$ with $r\in(0,R_*).$ By \eqref{finsler_norm}, \eqref{hyp:bound_anis} and \eqref{zur_tenglik},
$$
\begin{aligned}
\mu_h(\R^2) =   \cH^1(\p A_{m_h}) \le &\cH^1(\p\Omega)+ \frac{\cF(A_{m_h},u_{m_h}) + 2c_2\cH^1(\Sigma)}{c_1}  \\
\le &\cH^1(\p\Omega)+ \frac{\cF(A_1,u_1) + 2c_2\cH^1(\Sigma) }{c_1}, 
\end{aligned}
$$
hence, by compactness, there exist a not relabelled subsequence and a positive Radon measure $\mu$ in $\R^2$ such that 
$\mu_h\wk^*\mu$ as $h\to\infty.$ 
We claim that 
$$ 
\overline{\Omega\cap \p A} \subseteq \supp\mu \subseteq\p A. 
$$ 
Indeed, let $x\in \Omega\cap \p A$ and $r\in(0,\min\{\dist(x,\p \Omega),R_*\}).$
By the $\sdist$-convergence, there exists $x_h\in  Q_r(x)\cap \p A_{m_h}$ with  $x_h\to x$ such that $ Q_{r/2}(x_h)\subset Q_r(x)$ and hence, by the weak* convergence and \eqref{dens_esss}, 
$$
\mu(\overline{ Q_r(x)}) \ge \limsup\limits_{h\to\infty}\mu_h(\overline{ Q_r(x)})\ge
\limsup\limits_{h\to\infty}\mu_h( Q_{r/2}(x_h)) \ge  \varsigma_* r.
$$
This implies $x\in \supp\mu.$
Conversely, 
if, by contradiction, there exists $x\in \supp\mu\setminus \p A,$ 
then we can find $r>0$ such that $ Q_r(x)\cap \p A=\emptyset.$ From  the  $\sdist$-convergence
it follows that $ Q_{r/2}(x)\cap \p A_{m_h}=\emptyset$ for  $h$ large enough, and hence,
$$
0<\mu( Q_{r/2}(x)) \le \liminf\limits_{h\to\infty} \mu_h( Q_{r/2}(x)) =0,
$$
which is a contradiction.

From \eqref{dens_esss} it follows that 
\begin{equation}\label{dens_esss_mu}
\frac{\varsigma_*}{2} \le \frac{\mu( Q_r(x))}{2r} \le \frac{2\pi c_2}{c_1} 
\end{equation}
for any $x\in \Omega\cap \supp\mu$ any $r\in(0,R_*)$ with $ Q_r(x)\strictlyincluded \Omega$. Indeed,  let $x\in \Omega\cap \supp\mu$ and let $R(x):=\min\{R_*,\dist(x,\p\Omega)\bu\}\ba.$ Then by the weak* convergence $\mu_h(Q_r(x))\to \mu(Q_r(x))=0$ for a.e.\ $r\in(0,R(x)).$ In particular, \eqref{dens_esss_mu} holds for a.e.\ $r\in(0,R(x)).$ Since $\mu$ is a Radon measure, \eqref{dens_esss_mu} extends to all $r\in(0,R(x))$ by the left-continuity of the map $r\mapsto \mu(Q_r(x))$. 

From \eqref{dens_esss_mu} and \cite[Theorem 2.56]{AFP:2000} it follows that 
\begin{equation}\label{haus_bouns}
\varsigma_*\cH^1\res (\Omega\cap \supp\mu) \le \mu\res\Omega \le  \frac{4\pi c_2}{c_1}\,\cH^1\res (\Omega\cap \supp\mu).  
\end{equation}
Thus, $ \mu\res\Omega$ is absolutely continuous with respect to $\cH^1\res (\Omega\cap \supp\mu)$ 
and $\cH^1(\supp\mu)<\infty.$ 
By \eqref{haus_bouns}, 
$$
\cH^1(\p A) \le \cH^1(\Omega\cap \p A) +\cH^1(\p\Omega\cap \p A) \le 
\frac{1}{\varsigma_*}\,\mu(\Omega) +\cH^1(\p \Omega) <\infty.
$$ 

Finally let us prove \eqref{density_pA}. Fix any $x\in\Omega\cap\p A$ and let $R(x):=\min\{R_*,\dist(x,\p \Omega)\}.$ Then by \eqref{haus_bouns}
$$
\frac{\varsigma_* \cH^1(Q_r(x)\cap \p A)}{2r} \le \frac{\mu(Q_r(x))}{2r} \le \frac{4\pi c_2}{c_1}\,\frac{\cH^1(Q_r(x)\cap\p A)}{2r}.
$$
This and  \eqref{dens_esss_mu} imply 
$$
\frac{\varsigma_* \cH^1(Q_r(x)\cap \p A)}{2r} \le \frac{2\pi c_2}{c_1}\quad\text{and}\quad \frac{\varsigma_*}{2}\le \frac{4\pi c_2}{c_1}\,\frac{\cH^1(Q_r(x)\cap\p A)}{2r},
$$
and hence, \eqref{density_pA} follows. 
\end{proof}

\blue
We notice that by Proposition \ref{prop:adm_sets_have_finite_per} the limit set $A$ in Proposition \ref{prop:compact_A_m} is of finite perimeter. However, a priori, by the arguments of Proposition \ref{prop:compact_A_m}, its topological boundary $\p A$ does not need to be $\cH^1$-rectifiable, and so in $\mathcal{A}$.  This issue is overcome by introducing the extended class $\widetilde{\mathcal{A}}$ and the auxiliary model $\widetilde{F}$ in Section \ref{sec:auxiliary}.
\black 

\begin{corollary}\label{cor:A_ning_xossasi}
Let $\{A_{m_h}\}$ and $A$ be as in Proposition \ref{prop:compact_A_m}. Then $A_{m_h}\to A$ in  $L^1(\R^2)$ as $h\to\infty.$
\end{corollary}

\begin{proof}
Since $\cH^1(\p A)<\infty$ and $A_{m_h} \overset{K}{\to} \overline{A}$ as $h\to\infty,$ one has $\chi_{A_{m_h}} (x) \to \chi_A(x)$ as $h\to\infty$ for a.e.\ $x\in \R^2.$ Now Corollary \ref{cor:A_ning_xossasi} follows from the Dominated Convergence Theorem.
\end{proof}

The following result generalizes \cite[Theorem 4.2]{Gi:2002}  since it applies to set $\Gamma$ a priori not connected \bu and \black even not necessarily $\cH^1$-rectifiable),  but satisfying uniform density estimates. Recall that we denote by $\Gamma^r$ and $\Gamma^u$ the $\cH^1$-rectifiable and purely unrectifiable parts of a Borel $1$-set $\Gamma.$

\begin{proposition} \label{prop:convergence_tangent_line}
Let $\Gamma\subset\R^2$ be a Borel set such that  $\cH^1(\Gamma)<+\infty$ and  for some $r_0,c,C>0$ and for all $x\in\Gamma$
\begin{equation}\label{uniform_density_estimates}
c\le \frac{\cH^1( Q_r(x))}{2r} \le C,\qquad r\in(0,r_0). 
\end{equation}
Then for any $R>0$ and a.e.\ $x\in \Gamma^r$ one has
\begin{equation}\label{blowup_kuratowki}
\overline{  Q_{R,\nu_\Gamma(x)}  (x)} \cap \sigma_{x,\rho}(\Gamma) \overset{K}{\to} \cl{  Q_{R,\nu_\Gamma(x)} (x) }  \cap T_x 
\end{equation}
and 
\begin{equation}\label{blowup_weak}
\cH^1\res ( \sigma_{x,\rho}(\Gamma) ) \overset{*}{\wk} \cH^1\res T_x 
\end{equation}
as $\rho\to0,$ where \blue $\sigma_{x,r}$ is \black the blow-up map defined in \eqref{blow_ups} and $T_x$ is the generalized tangent line to $\Gamma$ at $x.$  Moreover,  for any $\cH^1$-measurable $\Gamma'\subset\Gamma$ and $\cH^1$-a.e. $x\in [\Gamma']^r$ the relations \eqref{blowup_kuratowki} and \eqref{blowup_weak} hold with $\Gamma'$ in place of $\Gamma.$
\end{proposition}

\begin{proof}
By \cite[Theorem 3.3]{Fa:1985}, $ \Gamma^r$ (and hence $ [\Gamma']^r$)  has a approximate tangent line  at $\cH^1$-a.e. $x,$ therefore,  \eqref{blowup_weak} follows from \cite[Remark 2.80]{AFP:2000}. To prove \eqref{blowup_kuratowki} with $\Gamma$ 
choose $x\in\Gamma$ such that $\theta(\Gamma,x)=1$ and $T_x$ exists.  Without loss of generality we assume that $x=0$ and $\nu_\Gamma(x)={\bf e_2}$ is  the unit normal to $T_x.$  
First we prove 
\begin{equation}\label{kuratowski_tangent}
\sigma_{0,r}(\Gamma) \overset{K}{\to} T_0 
\end{equation}
as $r\searrow0.$ Indeed, let $\mu_r:=\cH^1\res (\sigma_{0,r}(\Gamma))$ and $\mu_0:=\cH^1\res T_0.$ Given $r>0,$ since $\mu_r( Q_\rho(x))=\frac{\cH^1( Q_{\rho r}(rx))}{r},$ by  \eqref{uniform_density_estimates} for all $x\in \sigma_{0,r}(\Gamma)$ and $\rho\in(0,r_0/r)$ one has 
\begin{equation}\label{density_balls}
c\le \frac{\mu_r( Q_\rho(x))}{2\rho} \le C.  
\end{equation}
Let $r_k\searrow 0$ be any sequence. By compactness of sets in the Kuratowski convergence, passing to a further not relabelled subsequence if necessary we suppose that 
\begin{equation} \label{kuratowski_tangent1}
\sigma_{0,r_k}(\Gamma) \overset{K}{\to} L
\end{equation}
for some closed set $L\subset \R^2$ as $k\to\infty.$ We claim that $L=T_0.$ If there exists $x\in T_0\setminus L,$ then for some $\rho>0,$ $ Q_\rho(x)\cap L=\emptyset.$
By \eqref{kuratowski_tangent1}, $ Q_{\rho/2}(x)\cap \sigma_{0,r_k}(\Gamma)=\emptyset$ for all large $k$ so that $\mu_{r_k}( Q_{\rho/2}(x))=0.$  Then by \eqref{blowup_weak}
$$
0=\lim\limits_{k\to\infty} \mu_{r_k}( Q_{\rho/2}(x)) \ge \mu_0( Q_{\rho/2}(x))\ge\rho, 
$$
a contradiction. If there exists $x\in L\setminus T_0,$ then for some $ Q_\rho(x)\cap T_0=\emptyset$ for some $\rho>0$ and  there exists a sequence $x_k\in \sigma_{0,r_k}(\Gamma)$ such that $x_k\to x.$ Then  $ Q_{\rho/2}(x_k)\subset  Q_\rho(x)$ for all large $k$ so that by \eqref{blowup_weak} and \eqref{density_balls},
$$
0=\mu_0(\cl{ Q_\rho(x)}) \ge \limsup\limits_{k\to\infty} \mu_{r_k}(\cl{ Q_\rho(x)}) \ge \limsup\limits_{k\to\infty} \mu_{r_k}( Q_{\rho/2}(x_k))\ge c\rho,
$$
a contradiction. Thus, $L=T_0.$ Since the sequence $r_k\searrow 0$ is arbitrary, \eqref{kuratowski_tangent} follows. Now \eqref{blowup_kuratowki} is obvious.

\blue To prove the assertion for $\Gamma',$ fix any $x\in \Gamma'$ such that $\theta(\Gamma,x)=\theta(\Gamma',x)=1$ and both generalized tangents $T_x^\Gamma$ and $T_x^{\Gamma'}$ of $\Gamma$ and $\Gamma'$ exist. Note that $T_x^\Gamma=T_x^{\Gamma'} =:T_x.$ For shortness, assume that $x=0$ and $\nu_\Gamma(x)={\bf e_2}.$ Since in general $\Gamma'$ does not satisfy the uniform density estimates of type \eqref{uniform_density_estimates}, we cannot argue as above. 

Let $r_k\searrow0$ be arbitrary sequence such that $\sigma_{0,r_k}(\Gamma')\to L$ for some closed set $L\subset\R^2.$ Then for every $x\in L$ there exists a sequence $x_k\in \sigma_{0,r_k}(\Gamma')$ such that $x_k\to x.$ Since $\Gamma'\subset\Gamma$ and by  \eqref{kuratowski_tangent} $\sigma_{0,r_k}(\Gamma)\overset{K}{\to} T_0,$ we have  $x_k\in \sigma_{0,r_k}(\Gamma)$ and $x_k\to x\in T_0.$ Thus, $L\subset T_0.$ 
To prove the converse inclusion, assume that there exists $x\in T_0\setminus L.$ Since $L$ is closed there exists $r>0$ such that $B_{2r}(x)\cap L=\emptyset.$ As we mentioned in the beginning of the proof, for $\mu_k:=\cH^1\res (\sigma_{0,r_k}(\Gamma')) $ we have 
$\mu_k \overset{*}{\to} \cH^1(T_0).$ In particular, for every $\rho\in (0,r)$
$$
\lim\limits_{k\to+\infty} \mu_k(B_\rho(x))  = \cH^1(B_\rho(x)\cap T_0) =2\rho.
$$
Hence, $B_\rho(x)\cap \sigma_{0,r_k}(\Gamma')\ne\emptyset$  for each such $\rho$ and thus, taking a sequence $\rho_n\to0$ and using a diagonal argument we obtain a sequence $x_n\in \sigma_{0,r_{k_n}}(\Gamma')$ converging to $x.$ So $x\in L,$ a contradiction. 

Since $r_k\to0$ is arbitrary, one has $\sigma_{0,r}(\Gamma')\overset{K}{\to} T_0$ as $r\to0.$
\black 
\end{proof}

Next we turn to the compactness of displacements of the sequence of constrained minimizers $\{(A_m,u_m)\}.$ 

\begin{proposition}\label{prop:existence_of_u}
Let $A_{m_h}$ and $A$ be as in Proposition \ref{prop:compact_A_m}. Let $\{E_i\}_{i\in\N}$ be  \bu the family of \ba all connected components of $\Int{A}.$ There \bu exist \ba a further (not relabelled) subsequence of $\{(A_{m_h},u_h)\}$, a sequence $\{a_h\}$ of rigid displacements, a subset $\indexset$ of $\N$, a function $v_0\in H^1(\substrate)$ and a family $\{v_i\in GSBD^2(\Int{E_i})\cap H_\loc^1(\Int{E_i}\cup\substrate)\}_{i\in\indexset}$ such that  
$$
|u_{m_h} + a_h | \to +\infty 
$$
a.e. in $\bigcup_{i\in\N\setminus \indexset  }E_i,$
$$
u_{m_h} + a_h \wk v_0\chi_{\substrate} + \sum\limits_{i\in\indexset } v_i\chi_{E_i}
$$
weakly in $H_\loc^1((\cup_{i\in\indexset }E_i)\cup \substrate)$ (and hence a.e.\ in $(\cup_{i\in\indexset }E_i)\cup \substrate$), 
$$
\str{u_{m_h}} \to \str{ v_0}\chi_{\substrate} + \sum\limits_{i\in\indexset } \str{v_i}\chi_{E_i}
$$
weakly in $L_\loc^2((\cup_{i\in\indexset }E_i)\cup \substrate).$ 
\end{proposition}

The main difference of our compactness result from \cite[Theorem 1.1]{ChC:2019_jems} is not only that in our setting we have the set-function coupling, but also we need to select those components of limiting free crytal region where the displacements diverge and those in which they don't. This first requires to actually  prove that the behavior is consistent inside each component of the limiting free-crystal region, which is achieved using \cite[Proposition 3.7]{HP:2019}. 

\begin{proof}
Since $\substrate$ is connected and Lipschitz, by the Korn-Poincar\'e inequality and the Rellich-Kondrachov Theorem  there exists a further not relabelled subsequence $\{u_{m_h}\},$ a sequence $\{a_h\}$ of infinitesimal rigid displacements and $v_0\in H^1(\substrate;\R^2)$ such that $u_{m_h}  + a_h \to v_0$  weakly in $H^1(\substrate;\R^2)$ and a.e.\ in $\substrate.$  

\red We define the set $\indexset\subset \N$ as follows: For each $i\in\N$ fix some ball $B_i\strictlyincluded E_i.$ Since $A_{m_h} \overset{K}{\to} A,$ there exists $h_i^0>0$ such that $B_i\strictlyincluded A_{m_h}$ for all $h> h_i^0.$ By \eqref{hyp:elastic} and \eqref{zur_tenglik} 
$$
\sup_{h> h_i^0}\int_{B_i} |\str{u_{m_h} + a_h}|^2 dx \le \frac{1}{2c_3}\sup_{h> h_i^0} \int_{A_{m_h} \cup S} \C(x)\str{u_{m_h}}:\str{u_{m_h}}dx <+\infty,
$$
and thus, by \cite[Proposition 3.7]{HP:2019} either $|u_{m_h} + a_h|\to +\infty$ a.e.\ in $B_i$ or up to a subsequence, $u_{m_h} + a_h$ converges a.e.\ in $B_i.$ By a diagonal argument, we choose a further not relabelled subsequence $\{u_{m_h}\}$ and the subset $\indexset$ of indices $i\in \N$  such that for every $i\in \indexset$ the sequence $w_h:=u_{m_h}+ a_h \to v_i$  converges a.e.\ in $B_i$ as $h\to+\infty.$ \ba

We claim that for every $i\in\indexset$ there exists $v_i\in H_\loc^1(E_i;\R^2)\cap GSBD^2(E_i;\R^2)$ such that $w_h \to v_i$ weakly in $H_\loc^1(E_i;\R^2)$ and a.e.\ in $E_i$ as $h\to\infty$. To prove the claim we fix $i\in\indexset $ and let $D\strictlyincluded E_i$ be an arbitrary connected open set containing $B_i.$ Since $\sdist(\cdot,\p A_{m_h})\to\sdist(\cdot,\p A)$ locally uniformly in $\R^2,$ there exists $h_D>0$ such that $D\strictlyincluded \Int{A_{m_h}}$ for all  $h>h_D.$ 
Note that $w_h\in H^1(D)$ and
\begin{equation}\label{elastic_seq_kokok}
\sup\limits_{h>h_D} \int_D |\str{w_h}|^2dx \le 
C:=\frac{1}{2c_3}\sup\limits_{h>h_D} \int_{A_{m_h}\cup\substrate} \C(x)\str{u_{m_h}}:\str{u_{m_h}}dx<+\infty, 
\end{equation}
where in the first inequality we used \eqref{hyp:elastic} and in the second \eqref{zur_tenglik}. 
Since $w_h$ has finite limit a.e.\ in $B_i\subset D,$ by \cite[Proposition 3.7]{HP:2019} there exists $v_i^D\in H_\loc^1(D)\cap GSBD^2(D)$ and a  subsequence $\{w_h^D\}$ of $\{w_h\}$ such that $w_h^D \to v_i^D$ weakly in $H_\loc^1$ and a.e.\ in $D.$ Now choosing a sequence $D_1\strictlyincluded D_2\strictlyincluded \ldots \strictlyincluded E_i$ of connected open sets such that $B_i\subset D_1$ and $E_i = \cup_j D_j$ and using a diagonal argument we choose a (not relablled) subsequence $\{w_h\}$ and $v_i\in H_\loc^1(E_i)\cap GSBD_\loc^2(E_i)$ such that $w_h\to v_i$ weakly in $H_\loc^1(E_i)$ and a.e.\ in $E_i.$ In particular, $\str{w_h}\to \str{v_i}$  weakly in $L_\loc^2(E_i)$ and hence, by convexity and \eqref{elastic_seq_kokok}
$$
\int_{D_j} |\str{v_i}|^2dx \le \liminf\limits_{h\to+\infty} \int_{D_j} |\str{w_h}|^2dx \le C.
$$
Hence, letting $j\to\infty$ we get $v_i\in GSBD^2(E_i).$

\red Let us now show that by \ba the choice of $\indexset,$ for every $j\in \N\setminus \indexset $ one has $|u_{m_h}+ a_h|\to+\infty$ a.e.\ in $E_j$ as $h\to+\infty.$  
\red Indeed, by definition, if $i\notin N,$ then $|u_{m_h} + a_h|\to +\infty$ a.e.\ in $B_i\strictlyincluded E_i.$ Let $D\strictlyincluded E_i$ be any connected open set containing $B_i.$ As in \eqref{elastic_seq_kokok} we can show $\|\str{u_{m_h}+a_h}\|_{L^2(D)}^2$ is uniformly bounded for all sufficiently large $h,$ and therefore, by \cite[Proposition 4.7]{HP:2019} $|u_{m_h} + a_h|\to +\infty$ a.e.\ in $D.$ \ba

Finally, since $u_{m_h} + a_h \to u$ weakly in $H^1_\loc((\cup_{i\in \indexset } E_i)\cup \substrate),$  it follows that $\str{u_{m_h}}=\str{u_{m_h} + a_h} \to \str{u}$ weakly in $L^2_\loc((\cup_{i\in \indexset } E_i)\cup\substrate).$ 
\end{proof}

Proposition \ref{prop:existence_of_u} allows us to define a ``limit'' displacement.

\begin{proposition} \label{prop:def_of_u}
Let $\{(A_{m_h},u_{m_h})\},$ $\{a_h\},$  $A,$ $\indexset $ and $\{v_i\}_{i\in\indexset \cup\{0\}}$ satisfy the assertion of Proposition \ref{prop:existence_of_u} and let 
$$ 
u:=v_0 \chi_{\substrate} + \sum\limits_{i\in\indexset } v_i\chi_{E_i} + \sum\limits_{j\in\N\setminus \indexset } u_0\chi_{E_j}, 
$$ 
where $u_0$ is the displacement defining the mismatch strain $M_0$. Then 
\begin{equation}\label{lsc_elastic1}
\liminf\limits_{h\to\infty} \cW(A_{m_h},u_{m_h})
\ge \cW(A,u).
\end{equation}
\end{proposition}

\begin{proof}
Fix arbitrary open set $D\strictlyincluded \Int{A}\cup \substrate.$  By Proposition \ref{prop:existence_of_u} $u_{m_h} +a_h \to u$  weakly in $L^2(D\cap [(\cup_{i\in\indexset } E_i) \cup\substrate]),$ hence, by the convexity of the elastic energy 
\begin{align*}
&\liminf\limits_{h\to\infty} \cW(A_{m_h},u_{m_h})   = 
\liminf\limits_{h\to\infty} \int_{A_{m_h}\cup \substrate} W(x,\str{u_{m_h}}-M_0)dx\nonumber \\
\ge &  \liminf\limits_{h\to\infty} \Big(\int_{D\cap \substrate} W(x,\str{u_{m_h}} - M_0)dx +\sum\limits_{j\in\indexset } \int_{D\cap E_i} W(x,\str{u_{m_h}}-M_0)dx\Big)\nonumber \\
\ge & \int_{D\cap \substrate} W(x,\str{u}-M_0)dx +\sum\limits_{i\in\indexset } \int_{D\cap E_i} W(x,\str{u} - M_0)dx, 
\end{align*}
where we recall that $M_0=\str{u_0}.$ Since $\str{u} - M_0 =0$ a.e.\ in $\cup_{j\in \N\setminus \indexset } E_j,$ this inequality can also be rewritten as 
$$
\liminf\limits_{h\to\infty} \cW(A_{m_h},u_{m_h})
\ge \int_{D\cap (A\cup \substrate)} W(x,\str{u}-M_0)dx. 
$$
Now letting $D\nearrow \Int{A}\cup \substrate$ and using $|A\setminus\Int{A}|\le |\p A| =0$ we get \eqref{lsc_elastic1}.
\end{proof}

Now we establish the following ``lower semicontinuity'' of $\cF(A_m,u_m).$ 

\begin{proposition}\label{prop:lsc_surface}
Let $\{(A_{m_h},u_{m_h})\},$ $A$ and $u$ be as in Proposition \ref{prop:def_of_u}. 
Then $(\Int{A},u)\in \tilde \admissible$ and
\begin{equation}\label{lsc_surface_energy}
\liminf\limits_{h\to\infty} \cS(A_{m_h},u_{m_h}) \ge \tilde \cS(\Int{A},u),
\end{equation}
where $\tilde \cS$ is defined in \eqref{surface_tilde}.
\end{proposition}

We postpone the proof of this proposition after the following auxiliary lemma, needed to treat the delamination and jumps along the cracks.

\begin{lemma} \label{lem:internal_cracks}
\red Recall the definition of the sets  $I_r$ and $Q_r^\pm$ from \eqref{def:I_r_Q_rpm}. \black Let $\phi$ be any norm in $\R^2.$  Let $\{D_k\}$ and $\{m_k\}$ be sequences  of subsets of $Q_4$ and of natural numbers, respectively,  satisfying 
\begin{itemize}
\item[(a)] the number of connected components $\p D_k$ lying strictly inside $Q_4$ does not exceed $m_k;$

\item[(b)]  $\sdist(\cdot,\p D_k)\to  - \dist(\cdot,I_4)$ uniformly in $Q_4$ and 
$$
\sup\limits_k \cH^1(Q_1\cap \p D_k)<+\infty;
$$

\item[(c)] there exists a sequence $\{w_k\}\subset GSBD^2(Q_4)$ such that $J_{w_k}\subset Q_1\cap \p D_k$ and 
$$
\sup\limits_k  \int_{Q_1} |\str{w_k}|^2dx<+\infty;
$$

\item[(d)]  there exist $\xi^\pm\in\R^2$ such that 
$$
w_k \to w_0:= \xi^+ \chi_{Q_1^-} + \xi^-\chi_{Q_1^+\setminus U_1^\infty}\quad\text{a.e. in $Q_1\setminus U_1^\infty$}
$$
and 
$$
|w_k|\to+\infty \quad\text{a.e. in $U_1^\infty,$}
$$
 where $U_1^\infty$ is either $\emptyset$ or $Q_1^+.$ 
\end{itemize}
Then  there exists a subsequence $\{k_h\}\subset\N$ such that for  any $\delta\in(0,1)$ we can find  $h_\delta>0$ for which 
\begin{equation}\label{something_decrease}
\int_{Q_1\cap \p^*D_{k_h}} \phi(\nu_{D_{k_h}}) d\cH^1 + 2\int_{Q_1\cap D_{k_h}^{(1)} \cap \p D_{k_h}} \phi(\nu_{D_{k_h}})d\cH^1 \ge 2\int_{I_1}\phi({\bf e_2})d\cH^1 -\delta  
\end{equation}
for all $h>h_\delta.$ 
\end{lemma}

\bu 
Before the proof of Lemma \ref{lem:internal_cracks} we recall some notations and results from \cite{ChC:2019_jems}. Given $\xi\in \R^2\setminus\{0\},$ let $\Pi_\xi:=\{y\in\R^2:\,y\cdot \xi=0\}.$ For  every set $B\subset\R^2$ and for every $y\in\Pi_\xi$ we define 
$$
B_y^\xi:=\{t\in\R:\, y+t\xi \in B\}.
$$
Moreover, for every $u:B\to\R^2$ we define $\hat u_y^\xi:B_y^\xi\to\R$ by 
$$
\hat u_y^\xi (t):= u(y + t\xi)\cdot \xi.
$$
When $u\in GSBD^2(Q_1),$ then $\hat u_y^\xi \in SBV_\loc^2([Q_1]_y^\xi)$ for $\cH^1$-a.e. $\pi_\xi(Q_1)$ and for all $\xi\in \R^2\setminus\{0\}.$ In this case we define 
$$
I_y^\xi(u):=\int_{[Q_1]_y^\xi} |(\dot{u})_y^\xi|^2 dt,
$$
where $(\dot{u})_y^\xi$ is the density of the absolutely continuity part of $D \hat u_y^\xi$ and also 
$$
II_y^\xi(u):=|D(\tau(u\cdot \xi)_y^\xi)|([Q]_1^\xi),
$$
where $\tau(t):=\arctan(t).$ Recall that  
\begin{equation*}
\int_{\Pi_\xi} I_y^\xi(u) \cH^1(y) + 
\int_{\Pi_\xi} II_y^\xi(u) \cH^1(y) 
\le \int_{Q_1} |\str{u}|dx +  \int_{Q_1} |\str{u}|^2dx + \cH^1(J_u)
\end{equation*}
(see e.g. \cite[Eq. 3.8 and 3.9]{ChC:2019_jems}). \ba

\begin{proof}
The proof is similar  to \cite[Lemma 4.7]{HP:2019}. 
Since $\phi$ is even, 
$$
\phi(\xi) = \sup\limits_{\eta\in\R^2,\,\phi^o(\eta)=1} |\xi\cdot\eta|,\quad \xi\in\R^2,
$$
where $\phi^o$ is the dual norm of $\phi.$ By the compactness of $B^{\phi^o}:=\{\eta\in\R^2:\,\phi^o(\eta)=1\},$ for any countable set $\{\eta_i\}$ dense in $B^{\phi^o}$   and for any $\cH^1$-rectifiable set $K\subset\R^2$ 
$$
\phi(\nu_K(x)) = \sup\limits_{i\ge1} |\nu_K(x)\cdot\eta_i|\quad \text{for $\cH^1$-a.e. $x\in K.$}
$$
Hence, by \cite[Lemma 6]{DBD:1983} for any open set $U\subset\R^2$  
$$
\int_{U\cap K} \phi(\nu_K)d\cH^1 = \sup\limits_k \sup\Big\{\sum\limits_{i=1}^k\int_{A_i\cap K} |\nu_K\cdot\eta_i|d\cH^1:\,\text{$A_i\strictlyincluded U$ open and pairwise disjoint}\Big\}. 
$$
Moreover, by the area formula for any Borel set $B$
$$ 
\int_{B\cap K} |\nu_K\cdot \xi|d\cH^1 = |\xi| \int_{\pi_\xi(B)} \cH^0(K\cap B_y^\xi)d\cH^1(y),
$$ 
where 
$
\pi_\xi(z) = z - \big(z\cdot \tfrac{\xi}{|\xi|}\big)\,\tfrac{\xi}{|\xi|} 
$
and given $y\in\pi_\xi(B),$ $B_y^\xi\bu=\ba\pi_\xi^{-1}(y)\cap B.$

{\it Step 1:} There exists an at most countable set $\varUpsilon\subset B^{\phi^o}$ such that 
\begin{equation}
\lim\limits_{k\to+\infty} \cH^1(\pi_\xi(I_1)\setminus \pi_\xi(J_{w_k})) =0 
\label{good_projectinss}
\end{equation}
for any $\xi\in B^{\phi^o}\setminus \varUpsilon.$

Indeed, let $\varUpsilon$ be the set of all $\xi\in B^{\phi^o}$ for which there exists $y\in \pi_\xi(I_1)$ such that $\cH^1(\pi_\xi^{-1}(y)\cap \p D_k)>0.$  
By assumption (b) and \blue Proposition \ref{prop:maggi_foliation} \black the set $\varUpsilon$ is at most countable.
Let $\{w_{k_l}\}$ be arbitrary not relabelled subsequence of $\{w_k\}.$ In view of \cite[Eq. 3.23]{ChC:2019_jems} (applied with $A=U_1^\infty$) for any $\xi\in B^{\phi^o}\setminus \varUpsilon,$ $\epsilon>0$ and for $\cH^1$-a.e. $y\in \pi_\xi(Q_1)$ there exists a further subsequence $w_{k_{l_h}}$ (possibly depending on $\xi,$ $\epsilon$ and $y$)
\begin{gather}\label{soni_dayi}
\cH^0(J_{[\hat w_0]_y^\xi} \cap [Q_1\setminus U_1^\infty]_y^\xi) + \cH^0([\p U_1^\infty]_y^\xi) \le 
\liminf\limits_{h\to+\infty} \Big[
\cH^1(J_{[w_{k_{l_h}}]_y^\xi}) + \epsilon (I_y^\xi(w_{k_{l_h}}) + II_y^\xi{w_{k_{l_h}}})
\Big].
\end{gather}
By the definition of $w_0$ and $U_1^\infty,$ the left-side of \eqref{soni_dayi} is equal to $1$ for $\cH^1$-a.e. $y\in \pi_\xi(I_1).$ Theorefore, for such $y$ and for sufficiently small $\epsilon>0$ we have
$\liminf\limits_{h\to+\infty} \cH^1(J_{[w_{k_{l_h}}]_y^\xi})\ge1.$ Hence, for $\cH^1$-a.e. $y\in \pi_\xi(I_1)$ the line $\pi_\xi^{-1}(y)$ intersects $J_{w_{k_{l_h}}}$ for all $h$ and  \eqref{good_projectinss} follows.

Note that by \cite[Proposition 4.6]{HP:2019}  
\begin{equation}\label{eq:lscsd}
\liminf\limits_{k\to +\infty} \int_{Q_1\cap J_{w_k}} \phi(\nu_{J_{w_k}}) d\cH^1 \ge \int_{I_1} \phi({\bf e_2}) d\cH^1. 
\end{equation}

{\it Step 2:} Now we improve \eqref{eq:lscsd} by including coefficient 2 on the right-hand side of the inequality in the presence of a small error term. 

We proceed in three substeps. We redefine the displacement $w_k$ in the convex envelope $V_k^i$ of each connected component $K_k^i$ of $\p D_k$ in such a way that $\p V_k^i$ become jump sets with the left-hand side of \eqref{something_decrease} lowered up to a small error. 
\smallskip

{\it Substep 2.1:} First we identify $\{V_k^i\}.$

Fix any $\delta\in(0,1).$ 
By (b) there exists $k_\delta^1>0$ such that $([-2,2]\times [-2,-\delta]) \cup ([-2,2]\times [\delta,2])\subset \Int{D_k}$ for any $k\ge k_\delta^1.$ Let $F_k:=Q_1\cap D_k.$ Note that $\p F_k \subset (Q_1\cap \p D_k) \cup (\{\pm1\}\times[-\delta,\delta])$ and since $D_k\in \fA_{m_k},$ the  number of connected components $\{L_k^j\}_{j\ge1}$ of $\p F_k$ does not exceed $m_k$. Note that $F_k\subset [-1,1]\times[-\delta,\delta]$ and 
\begin{align}
\alpha_k: = & \int_{Q_1\cap \p^*E_{k}} \phi(\nu_{E_{k}}) d\cH^1 + 2\int_{Q_1\cap E_{k}^{(1)} \cap \p E_{k}} \phi(\nu_{E_{k}})d\cH^1 \nonumber \\ 
\ge & \int_{Q_2\cap \p^*F_k} \phi(\nu_{F_k}) d\cH^1 + 2\int_{Q_2\cap F_k^{(1)} \cap \p F_k} \phi(\nu_{F_k})d\cH^1 - 4\delta \nonumber\\
= & \sum\limits_{j\ge1}  \Big[\int_{Q_2\cap \p^*F_k\cap L_k^j} \phi(\nu_{F_k}) d\cH^1 + 2\int_{Q_2\cap F_k^{(1)} \cap \p F_k\cap L_k^j} \phi(\nu_{F_k})d\cH^1\Big] - 4\delta:=\alpha_k'.
\label{atlichna}
\end{align}
Next repeating the same arguments of Step 1 in the proof  of \cite[Lemma 4.7]{HP:2019} we can find a family $\{V_k^i\}_i$ of at most countably many pairwise disjoint closed convex sets with non-empty interior such that for each $L_k^j$ there exists a unique $V_i$ with  $L_k^j\subset V_k^i$ and 
\begin{equation}\label{pospsle}
\alpha_k' \ge \sum\limits_{i\ge 1} \int_{\p V_k^i} \phi(\nu_{V_k^i}^{}) d\cH^1 - 6\delta  
\end{equation}
see e.g. \cite[Eq. 4.34]{HP:2019}

{\it Substep 2.2:} Now we replace $w_k$ with another function $v_k$ associated to $V_k^i.$ Fix $\xi_0\in\R^2$ such that the jump set of the function 
$$
v_k:= w_k\chi_{Q_1\setminus \cup_i V_k^i} + \xi_0\chi_{\cup_i V_k^i}
$$
coincide with $\cup_i \p V_k^i$ (up to a $\cH^1$-negligible set).

By assumption (b) $\cup_i\p V_k^i \overset{K}\to I_1$ as $k\to+\infty.$ Moreover, as in Step 1 we can find a countable set $\varUpsilon'\subset B^{\phi^o}$ such that  
by assumption (b) and \eqref{good_projectinss}  
\begin{align*} 
\limsup\limits_{k\to+\infty} \cH^1(\pi_\xi(I_1)\setminus \pi_\xi(\cup_i\p V_k^i)) \le &  \limsup\limits_{k\to+\infty} \cH^1(\pi_\xi(I_1)\setminus \pi_\xi(\p D_k)) \nonumber \\
\le  & 
\lim\limits_{k\to+\infty} \cH^1(\pi_\xi(I_1)\setminus \pi_\xi(J_{w_k})) =0 
\end{align*}
for all $\xi \in B^{\phi^o}\setminus (\varUpsilon\cup\varUpsilon').$
Moreover, by assumption (d) $v_k \to w_0$ a.e.\ in $Q_1\setminus U_1^\infty$ and $|v_k|\to+\infty$ a.e.\ in $U_1^\infty.$

{\it Substep 2.3:}  By convexity of each $V_k^i$ we observe that 
$$
\liminf\limits_{k\to+\infty} \cH^0(\pi_\xi^{-1}(y) \cap J_{v_k}) \ge2 = 2\cH^0(J_{[\hat w_0]_y^\xi} \cap [Q_1\setminus U_1^\infty]_y^\xi). 
$$
for all $\xi\in B^{\phi^o}\setminus(\varUpsilon\cup\varUpsilon')$ and $\cH^1$-a.e.  $y\in\pi_\xi.$ 
Thus, by repeating the arguments of Step 1 in the proof of \cite[Proposition 4.6]{HP:2019} we get 
$$ 
\liminf\limits_{k\to+\infty} \int_{\cup_i \p V_k^i} \phi(\nu_{\cup_iV_k^i}^{})d \cH^1 = \liminf\limits_{k\to+\infty} \int_{J_{v_k}} \phi(\nu_{J_{v_k}}^{})d \cH^1 \ge 2\int_{I_1} \phi({\bf e_2})d\cH^1,
$$ 
which together with \eqref{atlichna} and \eqref{pospsle} implies the assertion of the lemma.
\end{proof}

Now we are ready to prove \eqref{lsc_surface_energy}. 

\begin{proof}[Proof of Proposition \ref{prop:lsc_surface}]
For shortness, let 
$$
G:=\Int{A}.
$$
We define
$$
\tilde u_h := (u_{m_h} + a_h)\chi_{A_{m_h}}^{} + \eta\chi_{\Omega\setminus A_{m_h}} 
$$
and 
$$
\tilde u:= u\chi_G^{} + \eta\chi_{\Omega\setminus G}^{}
$$
for $\eta\in (0,1)^2$ such that $\Omega\cap \p^*A_{m_h}\subset J_{\tilde u_h}$ and $\Omega\cap \p^*G\subset J_{\tilde u}$ up to an $\cH^1$-negligible set. Such $\eta$ exists by \blue Proposition \ref{prop:maggi_foliation} \black in view of the estimate 
$$
\cH^1(\p A_{m_h}) \le \frac{\cS(A_{m_h},u_{m_h})}{c_1} + \frac{2c_2\cH^1(\p \Omega)}{c_1}\le  \frac{\cF(A_1,u_1)}{c_1} + \frac{2c_2\cH^1(\p \Omega)}{c_1}, 
$$
which holds for every $h\ge1.$ 
Notice that $\tilde u_h\in GSBD^2(\Ins{\Omega})\cap  H_\loc^1((\Omega\cup \substrate) \setminus \p A_{m_h}),$ $\tilde u\in GSBD^2(\Ins{\Omega})\cap  H_\loc^1((\Omega\cup \substrate) \setminus \p G),$ $J_{\tilde u}\subset (\Omega\cap \p G)\cup (\Sigma\cap J_u)$  and
\begin{equation}
\cH^1(J_{\tilde u_h}) + \int_\Omega |\str{\tilde u_h}|^2dx
\le  \cF(A_{m_h},u_{m_h}) + \cH^1(\Sigma)\le M:= \cF(A_1,u_1)  + \cH^1(\Sigma)<\infty 
\label{good_estimatessssad}
\end{equation}
for every $h\ge1.$ Moreover, by Proposition \ref{prop:existence_of_u}, the definitions of $u,$ $\tilde u_h$ and $\tilde u,$
\begin{equation}\label{aeconverg}
\tilde u_h \to \tilde u \quad\text{a.e. in $[S\cup (\Omega\setminus G)]\cup  \bigcup\limits_{i\in\indexset } E_i$} 
\end{equation}
and 
\begin{equation}\label{aediverg}
|\tilde u_h| \to +\infty \quad\text{a.e. in $\bigcup\limits_{j\in\N\setminus\indexset } E_j,$} 
\end{equation}
where $\{E_i\}$ and $\indexset$ are provided by Proposition \ref{prop:existence_of_u}. 

We \blue recall \black that a priori $\p A,$ and hence $\p G$, does not need to be $\cH^1$-rectifiable. Therefore, by \cite[Theorem 6.2]{D:2013} $J_{\tilde u}\subset (\Omega\cap \p^r G)\cup (\Sigma\cap J_u),$ where we recall that $\p^rG$ is $\cH^1$-rectifiable part of $\p G.$

To prove \eqref{lsc_surface_energy} we use similar arguments as in \cite[Proposition 4.1]{HP:2019}. Let
$g \in L^\infty(\Sigma\times\{0,1\})$ be such that 
$$
g(x,s): = \varphi(x,\nu_\Sigma (x)) + s\beta(x)
$$  
for which we know by \eqref{hyp:bound_anis} that $g\ge0$ and 
\begin{equation}\label{g_ning_sharti}
|g(x,1) - g(x,0)| \le \varphi(x,\nu_\Sigma(x))\qquad\text{for  a.e.\ $x\in\Sigma.$ } 
\end{equation}
Let $\mu_h$ be the sequence of positive Radon measures defined at Borel sets $B\subset\R^2$ as 
$$
\begin{aligned}
\mu_h(B) &: =  \int_{B\cap \Omega \cap\p^*A_{m_h}} \varphi(x,\nu_{A_{m_h}})d\cH^1 + 2\int_{B\cap \Omega \cap (A_{m_h}^{(1)}\cup A_{m_h}^{(0)})\cap \p A_{m_h}}
\varphi(x,\nu_{A_{m_h}})d\cH^1\nonumber\\  
& + \int_{B\cap \Sigma\cap A_{m_h}^{(0)}\cap  \p A_{m_h} } \big[\varphi(x,\nu_\Sigma)
+ g(x,1)\big]d\cH^1(x)  + \int_{B\cap \Sigma\setminus \p A_{m_h}} g(x,0)d\cH^1 \nonumber \\
& + \int_{B\cap \Sigma\cap \p^*A_{m_h} \setminus J_{u_{m_h}} } g(x,1) d\cH^1 
 + \int_{B\cap\Sigma\cap J_{u_{m_h}}} \big[g(x,0) + \varphi(x,\nu_\Sigma)\big]\,d\cH^1 
\end{aligned}
$$
and
let $\mu$ be the positive measure defined at Borel sets $B\subset\R^2$ as 
$$
\begin{aligned}
& \mu(B) : =  \int_{B\cap \Omega \cap\p^*G} \varphi(x,\nu_A)d\cH^1 +2\int_{B\cap \Omega \cap G^{(1)} \cap  \p G\cap J_{\tilde u}} \varphi(x,\nu_G)d\cH^1\\
 +& \int_{B\cap \Sigma\setminus \p G} g(x,0)d\cH^1 + \int_{B\cap \Sigma\cap \p^*G \setminus J_u} g(x,1) d\cH^1 
 + \int_{B\cap\Sigma\cap J_{\tilde u}} \big[g(x,0) + \varphi(x,\nu_\Sigma)\big]\,d\cH^1.
\end{aligned}
$$  
Since  $S_{\tilde u}^A:=G^{(1)}\cap\p G\cap J_{\tilde u}$ and $\Sigma\cap J_{\tilde u} = \Sigma\cap J_u,$
we have
$$
\mu_h(\R^2) = \cS(A_{m_h} ,u_{m_h}) + \int_{\Sigma} \varphi(x,\nu_\Sigma)d\cH^1
$$
and 
$$
\mu(\R^2) = \tilde \cS(G,u) + \int_{\Sigma} \varphi(x,\nu_\Sigma)d\cH^1.
$$
Hence, to establish \eqref{lsc_surface_energy} it suffices to prove 
\begin{equation}\label{lsc_of_muh}
\liminf\limits_{h\to\infty} \mu_h(\R^2) \ge \mu(\R^2). 
\end{equation}
Since $\sup_h \mu_h(\R^2)<+\infty,$ by compactness, there exists a positive Radon measure $\mu_0$  in $\R^2$ such that (up to a subsequence) $\mu_h\wk^* \mu_0$ as $h\to\infty.$  
We show 
\begin{equation}\label{mu0_katta_mu}
\mu_0\ge \mu 
\end{equation}
and we observe that \eqref{lsc_of_muh} immediately follows from \eqref{mu0_katta_mu}.
To establish \eqref{mu0_katta_mu} it suffices to prove  
\begin{subequations}
\begin{align}
& \frac{d\mu_0}{d\cH^1\res (\Omega\cap\p^*G)}(x) \ge \varphi(x,\nu_G(x))\qquad \text{for a.e.\ $x\in \Omega\cap \p^*G,$}\label{eq:at_reduced_boundary}\\ 
& \frac{d\mu_0}{d\cH^1\res (\Sigma\cap \p^* G)}(x) \ge g(x,1) \qquad \text{for a.e.\ $x\in \Sigma\cap \p^* G,$}\label{eq:contact_of_G}\\ 
& \frac{d\mu_0}{d\cH^1\res (\Sigma\setminus \p G)}(x) \ge \varphi(x,\nu_\Sigma(x)) \qquad \text{for a.e.\ $x\in \Sigma\setminus \p G,$}\label{eq:density_Sigma}  \\
& \frac{d\mu_0}{d\cH^1\res S_{\tilde u}^A }(x) \ge 2\varphi(x,\nu_G(x))\qquad \text{for a.e.\ $x\in   S_{\tilde u}^A,$ }\label{eq:at_internal_crack}\\ 
& \frac{d\mu_0}{d\cH^1\res (\Sigma\cap J_{\tilde u})}(x) \ge 2\varphi(x,\nu_\Sigma(x)) \qquad \text{for a.e.\ $x\in \Sigma\cap  J_{\tilde u}$}\label{eq:at_delaminations} 
\end{align}
\end{subequations}
since $g(x,0)=\varphi(x,\nu_\Sigma).$ 

The proof of the estimates \eqref{eq:at_reduced_boundary}-\eqref{eq:at_delaminations} follows from similar arguments used in \cite[Proposition 4.1]{HP:2019} with special care needed for \eqref{eq:at_internal_crack}. In fact for \eqref{eq:at_internal_crack} we cannot employ the strategy used for  \cite[Eq. 4.40c]{HP:2019} that was hinged on the uniform bound on the number of boundary components, which here we do not have.  We instead adapt the arguments employed in \cite[Eq. 4.40g]{HP:2019} by using Lemma \ref{lem:internal_cracks}.  

Next we detail the proofs of \eqref{eq:at_reduced_boundary}-\eqref{eq:at_delaminations}.
\smallskip 

{\it Proof of \eqref{eq:at_reduced_boundary}.} Note that $A =G$ up to a negligible set.  
By Corollary \ref{cor:A_ning_xossasi}  $A_{m_h}\to A$ in $L^1(\R^2),$ thus, the proof of \eqref{eq:at_reduced_boundary} can be done following the arguments of  \cite[Eq. 4.40a]{HP:2019} using Reshetnyak lower semicontinuity Theorem \cite[Theorem 2.38]{AFP:2000}.
\smallskip

{\it Proof of \eqref{eq:contact_of_G}.}
Since $A_{m_h} \to G$ in $L^1(\R^2),$ we  have   $D\chi_{A_{m_h}}\wk^* D\chi_G.$ Thus, the proof of \eqref{eq:contact_of_G} directly follows from \cite[Lemma 3.8]{ADT:2017}  (see also the proof of \cite[Eq. 4.40d]{HP:2019}). 
\smallskip

{\it Proof of \eqref{eq:density_Sigma}.}
Let $x_0\in \Sigma\setminus \p G$ and let $r_0:=\dist(x_0,\p G) >0.$ Since  $\R^2\setminus \Int{A_{m_h}}\overset{K}{\to} \R^2\setminus \Int{A}=\R^2\setminus G,$ there exists $r_1\in(0,r_0)$ such that $B_r(x_0)\cap \Int{A_{m_h}}$ for any $r\in(0,r_1).$ Therefore, for any $r\in(0,r_1)$ 
$$
\begin{aligned}
\mu_h(\cl{B_r(x_0)}) = & \int_{\cl{B_r(x_0)}\cap \Sigma\cap A_{m_h}^{(0)}\cap  \p A_{m_h} } \big[\varphi(x,\nu_\Sigma)
+ g(x,1)\big]d\cH^1(x)  + \int_{\cl{B_r(x_0)}\cap \Sigma\setminus \p A_{m_h}} g(x,0)d\cH^1\\
\ge & \int_{\cl{B_r(x_0)}\cap \Sigma\cap A_{m_h}^{(0)}\cap  \p A_{m_h} } g(x,0)d\cH^1(x)  + \int_{\cl{B_r(x_0)}\cap \Sigma\setminus \p A_{m_h}} g(x,0)d\cH^1 \\
= & \int_{\cl{B_r(x_0)}\cap\Sigma} g(x,0)d\cH^1,
\end{aligned}
$$
where in the inequality we used \eqref{g_ning_sharti}. 
Thus, taking $limsup$ as $h\to+\infty$ and using $\mu_h \wk^*\mu_0$ we get 
$$
\mu_0(\cl{B_r(x_0)}) \ge \int_{B_r(x_0)\cap\Sigma} g(x,0)d\cH^1.
$$
Now \eqref{eq:density_Sigma} follows from the Besicovitch Differentiation Theorem.
\smallskip

{\it Proofs of \eqref{eq:at_internal_crack} and \eqref{eq:at_delaminations}.} We establish
\begin{equation}\label{eq:at_jump_of_extension}
\frac{d\mu_0}{d\cH^1\res K } \ge 2\phi(x,\nu_K) \quad\text{for $\cH^1$-a.e. $x\in K,$}
\end{equation}
where 
$$
K=S_{\tilde u}^A \cup (\Sigma\cap  J_{\tilde u}).
$$

Let $x\in K$ be such that $\theta(K,x)=1.$ Then either $x\in S_{\tilde u} \subset G^{(1)}\cap \p^r G$ or $x\in \Sigma\cap J_{\tilde u}.$  By setting $E_0:=\substrate$ and \bu recalling \ba that $\Int{A}=\cup_{i\in\N} E_i,$  in view of Proposition \ref{prop:existence_of_u}  
we \bu have  one \ba of the following:
\begin{itemize}
\item[(b1)] there exists $i_0\in\indexset $ such that  $x\in E_{i_0}^{(1)}\cap \p^r E_{i_0},$  
$
\theta(E_{i_0}^{(1)}\cap \p^r E_{i_0},x)=1
$ 
and $u_{m_h} + a_{m_h} \to u$ a.e.\ in $E_{i_0};$ 

\item[(b2)] there exist $i_0\in \indexset \cup \{0\}$ and $j_0\in \N\setminus \indexset $ 
such that $x\in \p^*E_{i_0}\cap\p^* E_{j_0}$ and $u_{m_h} + a_{m_h}\to u$ a.e.\ in $E_{i_0}$ and $|u_{m_h} + a_{m_h}|\to\infty$ a.e.\ in $E_{j_0};$

\item[(b3)] there exist $i_1,i_2\in\indexset \cup\{0\}$ with $i_1\ne i_2$ such that $x\in \p^*E_{i_1}\cap\p^* E_{i_2}$ and $u_{m_h} + a_{m_h}\to u$ a.e.\ in $E_{i_1}\cup E_{i_2}.$ 
\end{itemize}

\noindent 
Let  $L$ denote the set among $E_{i_0}^{(1)}\cap \p^r E_{i_0},$ $\p^*E_{i_0}\cap\p^* E_{j_0}$ and $\p^*E_{i_1}\cap\p^* E_{i_2}$ containing $x.$ Without loss of generality we assume that $x\in Y\subset L,$
where $Y$ is defined as the  set of points $y\in L\subset\p A$ satisfying: 
\begin{itemize}
\item[(c1)] $\theta(\p G,y) = \theta(\p A,y) =\theta(L,y)= 1$ and $\nu_G(y)=\nu_A(y)=\nu_L(y)$ exists. If $y\in\Sigma,$ then additionally, $\theta(\Sigma,x) =1$ and $\nu_\Sigma$ also exists;

\item[(c2)] as $\rho\to0$ the sets $\overline{ Q_{R,\nu_L}(y)}\cap \sigma_{\rho,x}(\p A), $ $\overline{ Q_{R,\nu_L}(x)}\cap \sigma_{\rho,y}(\p G) $ and $\overline{ Q_{R,\nu_L}(y)}\cap \sigma_{\rho,y}(L) $ converge  $\overline{Q_{R,\nu_L}(y)} \cap T_y$ in the Kuratowski sense, where  $R>0$ and $T_y$ is the generalized tangent line to $\p A$ at $y;$

\item[(c3)] one-sided traces $\tilde u^+(y)$ and $\tilde u^-(y)$ of $\tilde u$ w.r.t. $L$ exist and are not equal;

\item[(c4)] $\frac{d\mu_0}{ d\cH^1\res K}\,(y)$ exists and is finite.
\end{itemize}
In fact, $\cH^1(L\setminus Y)=0$ since for (c1) we notice that  $Y\subset L \subset \p^rA$ and $\p ^rA$ is $\cH^1$-rectifiable, for (c2) we use Proposition \ref{prop:convergence_tangent_line}  by observing that the points of  $\Sigma$ and $\Omega\cap \p A$ satisfy uniform density estimates  in view of the Lipchitzianity of $\Sigma$ and Proposition \ref{prop:compact_A_m}, respectively, for (c3) we use
\cite[Definition 2.4]{D:2013} and the existence of traces of $GBD$-functions along $C^1$-manifolds \cite[Theorem 6.2]{D:2013} and the fact that being a jump set of $\tilde u,$ the set  $K$ (and also $L$) can be covered by at most countably many one-dimensional $C^1$-graphs (up to a $\cH^1$-negligible set), and finally for (c4) we use Besicovitch  Differentiation Theorem.

Without loss of generality, we assume  $x=0,$ $\nu_K(x)={\bf e_2},$ $T_x$ is the $x_1$-axis and ${\bf e_2}$ is the outer normal of $E_{i_0}$. 

Let $4r_0:=\dist(0,\p\Omega)$ if $0\in\Omega$ and $4r_0:=\dist(0,\p\Sigma)$ if $0\in\Sigma;$ since $\Sigma$ is Lipschitz, it consists of at most countably many open connected components in $\p \Omega,$ and hence, $r_0>0$.  By weak convergence,
\begin{equation}\label{mu_r_convergence}
\lim\limits_{h\to\infty} \mu_h(\cl{ Q_r}) = \mu_0( Q_r) 
\end{equation}
for a.e.\ $r\in(0,r_0).$  By assumption (b3), \cite[Definition 2.4]{D:2013} and \cite[Remark 2.2]{D:2013} separately applied to $ Q_1^+:=Q_1\cap\{x_2>0\}$ and $ Q_1\setminus  Q_1^+$  we have
\begin{equation}\label{approximate_jump}
\lim\limits_{r\to 0} \int_{ Q_1} |\tau(\tilde u(rx)) - \tau(u_0(x))| dx =0,  
\end{equation}
where 
$$
u_0 := \tilde u^+(0)\chi_{ Q_1^+} + \tilde u^-(0)\chi_{ Q_1\setminus  Q_1^+} 
$$ 
and
$$
\tau(z) = (\arctan z_1,\arctan z_2),\quad z=(z_1,z_2)\in\R^2.
$$
For every $r\in(0,r_0)$ let 
$$
 U_r^\infty:=\{x\in  Q_1:\,\, \liminf\limits_{h\to\infty} |\tilde u_h(rx)|=+\infty\}.
$$
Unlike the proof of \cite[Eq. 4.40g]{HP:2019}, \eqref{aediverg} implies that  $U_r^\infty $ can have positive measure. 
By \eqref{aeconverg} and  the Dominated Convergence Theorem  
\begin{equation}\label{almost_convergence}
\lim\limits_{h\to\infty} \int_{ Q_1\setminus   U_r^\infty } |\tau(\tilde u_h(rx)) - \tau(\tilde u(rx))| dx = 0.  
\end{equation}
By (c2) applied with $R=8$, Proposition \ref{prop:convergence_tangent_line} and (c1)-(c3)  
\begin{align*}
 &  Q_8\cap \sigma_r(\p A) \overset{K}{\to} I_8\qquad \text{and}\qquad \cH^1\res ( Q_8\cap \sigma_r(\p A)) \overset{*}{\wk} \cH^1\res I_8,  
 \\
&  Q_8\cap \sigma_r(L) \overset{K}{\to} I_8\quad \text{and}\quad \cH^1\res ( Q_8\cap \sigma_r(L) )) \overset{*}{\wk} \cH^1\res I_8  
\end{align*}
as $r\to0.$ Hence, by \cite[Proposition A.5]{HP:2019} 
\begin{subequations}
\begin{align}
& \sdist(\cdot, \sigma_r(\p A)) \to - \dist(\cdot,T_0), \label{sdist_a_convergence}\\ 
& \sdist(\cdot, \sigma_r(\p E_{i_0})) \to -\dist(\cdot,T_0),\label{sdist_g_case_a}\\
&\sdist(\cdot, \sigma_r(\p [E_{i_0}\cup E_{j_0}])) \to -\dist(\cdot,T_0), \label{sdist_g_case_b}\\
&\sdist(\cdot, \sigma_r(\p [E_{i_1}\cup E_{i_2}])) \to -\dist(\cdot,T_0) \label{sdist_g_case_c}
\end{align}
\end{subequations}
locally uniformly in $\cl{ Q_4}$ as $r\to0.$ 
Let 
$$ 
U_0^\infty=
\begin{cases}
\emptyset & \text{in cases (c1) and (c3)},\\
Q_1^+ & \text{in case (c2)}. 
\end{cases} 
$$ 
By the definitions of $E_{i_0},$ $E_{j_0},$ $E_{i_1}$ and $E_{i_2}$ and \eqref{sdist_g_case_a}-\eqref{sdist_g_case_c}  
\begin{equation}\label{small_r_u_cheksiz}
\lim\limits_{r\to0} | U_r^\infty\Delta  U_0^\infty| =0.  
\end{equation}

{\it Step 1:} We choose sequences $h_k\nearrow \infty$ and $r_k\searrow 0$ as follows.  By \eqref{mu_r_convergence}, \eqref{approximate_jump}, \eqref{sdist_a_convergence} and \eqref{small_r_u_cheksiz} for any $k\in\N$ there exists  $r_k\in(0,\frac1k)$ such that \eqref{mu_r_convergence} holds with $r=r_k$ and 
\begin{subequations}
\begin{align}
&\|\sdist(\cdot, \sigma_{r_k} (\p A)) + \dist(\cdot,  T_0) \|_{L^\infty( Q_4)} <\frac{1}{k^2},\label{bound_A_conv_tangent}\\
&\int_{ Q_1} |\tau(\tilde u(r_k x)) - \tau(u_0(x))| dx <\frac{1}{k^2}, \label{defi_jump_set_trace}\\
&| U_{r_k}^\infty \Delta  U_0^\infty| < \frac{1}{k^2} \label{small_vol_disp_infinite}.
\end{align}
\end{subequations}

\noindent 
Given $k\ge1$ and $r_k$, since $A_{m_h}$ $sdist$-converges to  $A$ and the function $\tau$ is bounded, by \eqref{almost_convergence}, \eqref{small_vol_disp_infinite} and \eqref{mu_r_convergence} we can choose  $h_k$ such that 
\begin{subequations}
\begin{align}
&\frac{1}{h_kr_k} <\frac{1}{k},\label{biggg_h_k}\\
&\|\sdist(\cdot,   \sigma_{r_k} (\p A_{m_{h_k}} )) - \sdist(\cdot,   \sigma_{r_k} ( \p A))\|_{L^\infty( Q_4)}  <\frac{1}{k},\label{A_k_close_to_A}\\
&\int_{ Q_1\setminus  U_0^\infty} |\tau(\tilde u_{h_k}(r_kx)) - \tau(\tilde u(r_kx))| dx < \frac{1}{k},\label{w_h_close_to_w}\\
& \mu_{h_k}(\cl{ Q_{r_k}}) \le \mu_0( Q_{r_k}) +r_k^2.  \label{mu_k_close}
\end{align}
\end{subequations}
Notice that by \eqref{biggg_h_k}, $h_k\to\infty$ as $k\to\infty.$ 

Let 
$$ 
D_k:= \sigma_{r_k} (A_{m_{h_k}}\cup \substrate)
$$
and
$$
w_k(x):=\tilde u_{h_k}(r_kx),\quad x\in Q_1.
$$  
Then the number of connected components of $\p D_k$ lying strictly inside $Q_4$ does not exceed $m_{h_k},$ and $w_k\in GSBD^2(Q_1)$ with $J_{w_k}\subset Q_1\cap \p D_k.$
By \eqref{A_k_close_to_A} 
and \eqref{bound_A_conv_tangent}, 
\begin{equation*}
\sdist(\cdot,\p D_k) \to -\dist(\cdot, T_0) \quad \text{uniformly in $Q_4$ as $k\to\infty.$} 
\end{equation*}
Moreover, by \eqref{defi_jump_set_trace} and  \eqref{w_h_close_to_w}  $w_k\to u_0$ a.e.\ in $Q_1\setminus U_1^\infty$ and $|w_k|\to+\infty$ a.e.\ in $U_1^\infty.$ 
By the finiteness of 
$$
\frac{d\mu_0}{\cH^1\res L}(0) = \lim\limits_{k\to\infty} \frac{\mu_0( Q_{r_k})}{2r_k} 
$$
and \eqref{mu_k_close} 
\begin{equation}\label{vashmulle}
 \cH^1( Q_1\cap \p D_k) =
\frac{\cH^1( Q_{r_k}\cap \p A_{m_{h_k}} )}{r_k}  \le  \frac{\mu_{h_k}(\cl{ Q_{r_k}})}{c_1 r_k} \le C:=\frac{2}{c_1}\frac{d\mu_0}{\cH^1\res L}(0) + 1 
\end{equation}
for all large $k.$ 
Moreover, by changing variables as $x=r_ky$  and using \eqref{good_estimatessssad} we get 
$$ 
\int_{Q_1} |\str{w_k}|^2dx = \int_{Q_{r_k}} |\str{\tilde u_k}|^2dy \le M  
$$ 
for all $k;$ note that the first equality holds only in dimension two.

Fix $\delta\in(0,1).$ 
Since $\varphi$ is uniformly continuous, there exists $k_\delta^0>0$ such that  
$$
|\varphi(x,\nu) - \varphi(0,\nu)| < \delta,\qquad x\in \cl{ Q_{r_k}},\,\,\nu\in\S^1.
$$
Therefore, by the definitions of $D_k$ and $\mu_h,$ the nonnegativity of $g$ as well as \eqref{vashmulle} 
\begin{align}\label{shapat_estimates}
\frac{\mu_{h_k}(\cl{ Q_{r_k}})}{r_k} \ge 
&
\int_{ Q_1\cap \p^* D_k} \phi(\nu_{D_k})d\cH^1+ 2\int_{ Q_1\cap D_k^{(1)} \cap \p D_k} \phi(\nu_{D_k})d\cH^1 -  2Cc_2\delta,
\end{align}
where 
$$
\phi(\nu) = \varphi(0,\nu).
$$

By Lemma \ref{lem:internal_cracks} applied with sequences $\{D_k\}$ and $\{m_{h_k}\}$ we find $k_\delta^2>k_\delta^1$ such that 
$$
\int_{ Q_1\cap \p^* D_k} \phi(\nu_{D_k})d\cH^1+ 2\int_{ Q_1\cap D_k^{(1)} \cap \p D_k} \phi(\nu_{D_k})d\cH^1 \ge 
2\int_{I_1} \phi({\bf e_2})d\cH^1 - \delta.
$$
Thus, by \eqref{shapat_estimates}  and \eqref{mu_k_close} we get 
$$
\frac{\mu_0(Q_{r_k})}{2r_k} + \frac{r_k}{2} \ge  \int_{I_1} \phi({\bf e_2})d\cH^1 - \frac{2Cc_1 + 1}{2}\,\delta 
$$
for all $k>k_\delta^2.$ Now letting first $k\to+\infty$ and then $\delta\to0$ we get \eqref{eq:at_jump_of_extension}.
\end{proof}

\section{Proof of the main results}\label{sec:proof_thm1}

 The aim of this section is to prove theorems of Section \ref{subsec:main_results}.
We start by showing that the volume-constraint infima of $\cF$ in $\admissible$ and of $\tilde \cF$ in $\tilde\admissible$ in fact coincide.

\begin{proposition}\label{prop:min_extend}
Assume hypotheses (H1)-(H3) and let $\fm\in(0,|\Omega|)$ or $\substrate=\emptyset.$ 
Then 
\begin{equation}\label{equiv_min_problem}
\inf\limits_{(A,u)\in \admissible,\,|A|=\fm}\, \cF(A,u)\,\, = \inf\limits_{(A,u)\in\tilde  \admissible,\,|A|=\fm}\, \tilde \cF(A,u) = \inf\limits_{(A,u)\in\tilde  \admissible}\, \tilde \cF^\lambda(A,u) 
\end{equation}
for any $\lambda\ge \lambda_0,$ where $\lambda_0$ is given by \cite[Theorem 2.6]{HP:2019} and $\tilde \cF^\lambda$ is given by \eqref{def_cf_tilde_lambda}. 
\end{proposition}

\begin{proof}
We repeat similar arguments to \cite[Section 5]{HP:2019}. 
Note that for any $\lambda>0$
\begin{equation}\label{easy_part}
\inf\limits_{(A,u)\in \admissible,\,|A|=\fm}\, \cF(A,u)\,\, \ge  \inf\limits_{(A,u)\in\tilde  \admissible,\,|A|=\fm}\, \tilde \cF(A,u) \ge \inf\limits_{(A,u)\in \tilde \admissible}\, \tilde \cF^{\lambda}(A,u). 
\end{equation}
Further we fix any $\lambda\ge\lambda_0.$
Recall that  from \cite{HP:2019} for such $\lambda$ 
$$
\inf\limits_{(A,u)\in \admissible,\,|A|=\fm}\, \cF(A,u) =\lim\limits_{m\to+\infty}
\,\min\limits_{(A,u)\in \admissible_m,\,|A|=\fm}\, \cF(A,u) 
=\lim\limits_{m\to+\infty}\,\min\limits_{(A,u)\in \admissible_m}\, \cF^{\lambda}(A,u),
$$
where $\cF^\lambda$ is given by \eqref{eq:flambda}. Thus, in view of \eqref{easy_part} to prove \eqref{equiv_min_problem}  it is enough to establish that \bu for \ba $\epsilon>0$ there exists $n_\epsilon\in\N$ and $(A_\epsilon,u_\epsilon)\in \admissible_{n_\epsilon}$ such that 
\begin{equation}\label{shshshshka}
\inf\limits_{(A,u)\in \tilde \admissible}\, \tilde \cF^{\lambda}(A,u) + \epsilon > \cF^{\lambda}(A_\epsilon,u_\epsilon). 
\end{equation}
To prove the existence of $n_\epsilon$ and $(A_\epsilon,u_\epsilon)\in\admissible_{n_\epsilon}$, we repeat essentially the same arguments of the proof of \cite[Eq. 5.4]{HP:2019}. For the convenience of the reader we give the detailed proof. Given $\epsilon>0$ let $(B_1,v_1)\in\tilde \admissible$ be such that 
\begin{equation}\label{defini_b1v1}
\inf\limits_{(A,u)\in \tilde\admissible}\, \cF^{\lambda}(A,u)> \cF^{\lambda}(B_1,v_1) -\epsilon. 
\end{equation}
Since $|B_1| = |\Int{B_1}|$ and $\cF^{\lambda}(B_1,v_1)\ge \cF^{\lambda}(\Int{B_1},v_1),$ we may assume that $B_1=\Int{B_1},$ i.e., $B_1$ is open.

{\it Step 1:} First we remove the jump set $J_{v_1}$ of $v_1$ on $\Sigma$ making a hole in $\Omega.$ Recall that by our choice, $\nu_\Sigma$ is always directed towards $\Omega.$ Since $\Sigma$ is Lipschitz, by the regularity of $\cH^1\res\Sigma,$ there  exists  \bu a relatively \ba open set $\Sigma'\subset \Sigma$ such that $\cH^1(J_{v_1}\setminus\Sigma')=0$ and $\cH^1(\Sigma'\setminus J_{v_1})<\frac{\epsilon}{c_2}.$ 

Let $r\in(0,\frac{\epsilon}{\lambda\cH^1(\Sigma)})$ be such that 
\begin{equation}\label{cont_phi_poop}
|\varphi(x,\nu) - \varphi(y,\nu)|<\frac{\epsilon}{\cH^1(\Sigma)}  
\end{equation}
whenever $|x-y|<4r.$ Since $\Sigma$ is Lipschitz, by Vitaly Covering Lemma we can find an at most countable family $\{Q_{r_j,\nu_{\Sigma}(x_j)}(x_j)\}_{j\ge1}$ of disjoint open squares such that $x_j\in\Sigma,$ $r_j\in(0,r),$ $\Sigma\cap Q_{r,\nu_{\Sigma}(x_j)}(x_j)$ is a graph in $\nu_\Sigma(x_j)$-direction, 
$\Sigma$ crosses  two opposite sides of each $Q_{r,\nu_{\Sigma}(x_j)}(x_j)$ parallel to $\nu_\Sigma(x_j)$ and 
\begin{equation}\label{good_cover_sigmaprime}
\cH^1\Big(\Sigma'\setminus\bigcup_j \cl{Q_{r_j,\nu_{\Sigma}(x_j)}(x_j)}\Big)=0.   
\end{equation}
Note that $\sum\limits_j r_j< \cH^1(\Sigma).$ 
For each $j$ define 
$$
\Sigma_j:=(\Sigma\cap \cl{Q_{r,\nu_{\Sigma}(x_j)}(x_j)}) + \rho_j\nu_\Sigma(x_j),
$$
where $\rho_j\in(0,r_j)$ is such that  $\Sigma_j$ still connects two vertical sides of $Q_{r,\nu_{\Sigma}(x_j)}(x_j)$ and $\sum_j \rho_j <\frac{\epsilon}{2c_2}.$ Let $U_j$ be the  
open set whose boundaries are $\Sigma_j,$ $\Sigma\cap \cl{Q_{r,\nu_{\Sigma}(x_j)}(x_j)}$ and two vertical sides of $Q_{r,\nu_{\Sigma}(x_j)}(x_j).$ Note that $\{U_j\}_j$ is a countable  family of pairwise disjoint open sets.

Let $B_2:=B_1\setminus \cl{\cup_j U_j}$ and $v_2:=v_1\Big|_{B_2\cup \substrate}.$ Then  using the localized version of $\cS$ we get
\begin{align}\label{letyouseee}
\cS(B_2,v_2) \le & \cS(B_1,v_1-u_0; \Omega\setminus \cl{\cup_j U_j}) + \sum\limits_{j} \Big(\int_{\Sigma_j} \varphi(x,\nu_\Sigma(x))d\cH^1 + 2c_2 \rho_j\Big). 
\end{align}
By \eqref{cont_phi_poop} and the definition of $\Sigma_j$
$$
\int_{\Sigma_j} \varphi(x,\nu_\Sigma(x))d\cH^1  \le \int_{\Sigma\cap Q_{r,\nu_{\Sigma}(x_j)}(x_j)} \varphi(y,\nu_\Sigma(y))d\cH^1  + \frac{\epsilon\cH^1(\Sigma\cap Q_{r,\nu_{\Sigma}(x_j)}(x_j))}{\cH^1(\Sigma)}.
$$
Thus summing this inequality in $j$ and using pairwise disjointness of $Q_{r,\nu_{\Sigma}(x_j)}(x_j)$ and \eqref{good_cover_sigmaprime} we get
$$
\sum\limits_j \int_{\Sigma_j} \varphi(x,\nu_\Sigma(x))d\cH^1 \le \int_{\Sigma'} \varphi(y,\nu_\Sigma(y))d\cH^1 + \frac{\epsilon\cH^1(\Sigma')}{\cH^1(\Sigma)}.
$$
Using the definition of $\Sigma'$  we obtain
$$
\sum\limits_j \int_{\Sigma_j} \varphi(x,\nu_\Sigma(x))d\cH^1 \le \int_{J_{v_1}} \varphi(y,\nu_\Sigma(y))d\cH^1 
+ 2\epsilon.
$$
Inserting this in \eqref{letyouseee} and using the inequality $\sum_j \rho_j <\frac{\epsilon}{2c_2}$  we get
$$
\cS(B_2,v_2) \le  \cS(B_1,v_1-u_0; \Omega\setminus \cl{\cup_j U_j}) + \int_{J_{v_1}} \varphi(y,\nu_\Sigma(y))d\cH^1 
+ 3\epsilon \le \cS(B_1,v_1) + 3\epsilon.
$$
Then by the nonnegativity of the elastic energy, for $(B_2,v_2)$ we get 
$$
\tilde \cF(B_2,v_2) \le \tilde \cF(B_1,v_1) + 3\epsilon.
$$
Notice that by our construction $\Sigma\cap J_{v_2}$ is $\cH^1$-negligible, hence by Proposition \ref{lem:gcbd_with_zero_jumps} $v_2\in H^1_\loc(\Ins{B_2}).$ 

Finally we estimate the volume contribution of $B_2.$ Since $U_j\subset Q_{r,\nu_{\Sigma}(x_j)}(x_j)$ and 
$r_j\le r<\frac{\epsilon}{\lambda\cH^1(\Sigma)},$ using $\sum_j r_j <\cH^1(\Sigma)$ we get 
$$
|B_1\setminus B_2| \le\sum\limits_j |U_j| \le \sum\limits_j r_j^2 \le r\sum\limits_j r_j < \frac{\epsilon}{\lambda}.
$$
Therefore, 
\begin{equation}\label{defini_b2v2}
 \tilde \cF^{\lambda}(B_1,v_1) \ge \tilde \cF^{\lambda}(B_2,v_2) - 4\epsilon.
\end{equation}

{\it Step 2:} Let $\{E_i\}_{i\ge1}$ be all open connected components of $B_2$ (recall that $B_2$ is open). We remove all sufficiently small connected components of $B_1.$ Using the localized versions of $\cS$ and $\cW$  we have
$$
\cW(B_2,v_2-u_0;\Omega) = \sum\limits_{i\ge1} \cW(E_i,v_2-u_0;\Omega). 
$$
Since $\p E_i\cap\p E_j\subset B_2^{(1)}\cap\p B_2$ and $\varphi(x,\cdot)$ is even,
$$
\cS(B_2,v_2;\Omega) = \sum\limits_{i\ge1} \cS(E_i,v_2;\Omega).
$$
Hence, there exists $N_1\in\N$ such that the set 
$
B_3:=\cup_{i=1}^{N_1} E_i 
$
satisfies 
$$
\begin{gathered}
\cS(B_2,v_2;\Omega) +\cW(B_2,v_2-u_0;\Omega)+ \epsilon >\cS(B_3,v_2;\Omega)+\cW(B_3,v_2-u_0;\Omega),\\
0 \le  |B_2| - |B_3| < \frac{\epsilon}{\lambda}.
\end{gathered}
$$
Thus, 
\begin{equation}\label{defini_b3v3}
\cF^{\lambda}(B_2,v_2) > \cF^{\lambda}(B_3,v_3) - 2\epsilon , 
\end{equation}
where $v_3:=v_2\Big|_{B_3}.$

{\it Step 3:} Let $\{F_j\}_{j\ge1}$ be all connected components of $\Omega\setminus \cl{B_3}$ such that $\p F_j\subset \p B_3$ (hence, $F_i$ are holes in $B_3$).  We fill in all sufficiently small holes. Since $\cS(B_3,v)<+\infty,$ there exists $N_2\ge1$ such that 
$$
\sum\limits_{i>N_2} \cS(F_i,v_3;\Omega) + \sum\limits_{i>N_2} \cW(F_i,v_3-u_0;\Omega) <\epsilon,\quad \sum\limits_{i>N_2} |F_i|<\frac{\epsilon}{\lambda}.
$$
Then the set $B_4:= B_3\cup (\cup_{i>N_2} F_i)$ and the function $v_4:=v_3\chi_{B_2\cup \substrate} + u_0\chi_{\cup_{i>N_2} F_i}$ satisfies 
\begin{equation}\label{defini_b4v4}
\cF^{\lambda}(B_3,v_3)>\cF^{\lambda}(B_4,v_4) - 2\epsilon .
\end{equation}
By construction, $\overline{\p^* B_4}$ has at most $N_1 + N_2$ connected components.  

{\it Step 5:} Finally we construct $(A_\epsilon,u_\epsilon)\in\admissible_{n_\epsilon}$ satisfying \eqref{shshshshka} for some $n_\epsilon\in\N.$ Let $B_5:=\Int{\cl{B_4}}.$ Since  $B_4$ can have finitely many ``substantial'' holes  $B_5\cap \p B_4=\emptyset.$ In particular, if we extend $v_4$ arbitrarily to the set $B_4^{(1)}\cap \p B_4$ and denote the extension by $v_5,$ then $v_5\in GSBD^2(\Ins{B_5})$ and $J_{v_5}=S_{v_4}^{B_4}$ up to a $\cH^1$-negligible set, where $S_u^A$ is defined in \eqref{defini_Su}. 
Since $v_5=v_4$ a.e.\ in $B_5,$ by \eqref{defini_b1v1}-\eqref{defini_b4v4}  
$$ 
\int_{B_5\cup\substrate} \C(x)\str{v_5}:\str{v_5} =  \cW(B_4,v_4) \le \tilde \cF(B_4,v_4) + c_2\cH^1(\Sigma) \le C+ 9\epsilon, 
$$ 
where $C:=\max\{1,\inf_{\tilde\admissible}\tilde \cF\}$ is independent of $\epsilon.$

By   \cite[Theorem 1.1]{ChC:2019_arma} there exists $u_\epsilon\in SBV^2(\Ins{B_5})\cap L^\infty(\Ins{B_5})$ such that $J_{u_\epsilon}$ is contained in a union $\Gamma$ of finitely many closed connected pieces of $C^1$-curves in $\Ins{B_5},$ $u_\epsilon\in W^{1,\infty}(\Ins{B_5})$ and 
\begin{equation}\label{elastic_close}
\int_{B_5\cup\substrate} |\str{u_\epsilon} - \str{v_5}|^2 dx \le \frac{\epsilon}{4(C+11\epsilon)(\|\C\|_\infty +1)}  
\end{equation}
and 
\begin{equation}\label{jump_approx000}
\cH^1(J_{u_\epsilon}\Delta J_{v_5})<\frac{\epsilon}{2c_2}. 
\end{equation}
Since $J_{v_5}\subset B_5,$ we can assume that the squares $\{Q_j\}_{j\ge1}$ of Vitali cover in \cite[Eq. 4.3a]{ChC:2019_arma} satisfies $Q_j\strictlyincluded B_5.$ Therefore, we may assume that $\Gamma\subset \cl{B_5}.$ Since $\cH^1\res \Gamma$ is regular, we may extract finitely many intervals of $\Gamma$ whose union $\Gamma'$ still covers $J_{u_\epsilon}$ and satisfies $\cH^1(\Gamma'\setminus J_{u_\epsilon})<\frac{\epsilon}{2c_2}.$ 
Now we define 
$$
A_\epsilon:=B_5\setminus \cl{\Gamma'}.
$$
Recall that both $\Sigma\cap J_{v_5}$ and $\Sigma\cap J_{e-\epsilon}$ are $\cH^1$-negligible.
By the definition of $B_5$ and $\Gamma',$ there exists $n_\epsilon\in\N$ such that $(A_\epsilon,u_\epsilon)\in \admissible_{n_\epsilon}.$ By the definition of $\tilde \cS,$ $B_5$ and $v_5$ as well as by \eqref{jump_approx000} we have
\begin{align*} 
\tilde \cS(B_4,v_4) = & \int_{\Omega\cap \p^* B_5} \varphi(x,\nu_{B_5})d\cH^1 + 2\int_{B_5\cap J_{v_5}} \varphi(x,\nu_{J_{v_5}})d\cH^1 + \int_{\Sigma\cap \p^*B_5} \beta d\cH^1\nonumber\\
\ge & \int_{\Omega\cap \p^* B_5} \varphi(x,\nu_{B_5})d\cH^1 + 2\int_{B_5\cap J_{u_\epsilon}} \varphi(x,\nu_{J_{u_\epsilon}})d\cH^1 + \int_{\Sigma\cap \p^*B_5} \beta d\cH^1 - \epsilon.
\end{align*}
Thus, by the definition of $A_\epsilon$ and $\Gamma'$ 
\begin{align} \label{surface_epsilon}
\tilde \cS(B_4,v_4) \ge & \int_{\Omega\cap \p^* A_\epsilon} \varphi(x,\nu_{A_\epsilon})d\cH^1 + 2\int_{A_\epsilon^{(1)}\cap \Gamma'} \varphi(x,\nu_{\Gamma'})d\cH^1 + \int_{\Sigma\cap \p^*A_\epsilon} \beta d\cH^1 - 2\epsilon\nonumber \\
= & \cS(A_\epsilon,u_\epsilon) - 2\epsilon.
\end{align}
Moreover, using the relations $|A_\epsilon\Delta B_4|=0$ and $v_4=v_5$ a.e.\ in $B_5$  and Cauchy-Schwartz inequality for nonnegative symmetric forms we obtain
\begin{align} \label{ahahaha010101}
\cW(A_\epsilon,u_\epsilon) \le & \cW(B_4,v_4) + 2\int_{B_5\cup\substrate} \C(x)\str{u_\epsilon}:(\str{u_\epsilon} - \str{v_5})\nonumber \\
 \le & \cW(B_4,v_4) + 2\Big(\int_{B_5\cup\substrate} \C(x)\str{u_\epsilon}:\str{u_\epsilon}dx\Big)^{1/2}\times \nonumber \\
 &\times 
\Big(\int_{B_5\cup\substrate} \C(x) (\str{u_\epsilon} - \str{v_5}):(\str{u_\epsilon} - \str{v_5})dx\Big)^{1/2}.
\end{align}
Similarly,
\begin{align*}
& \int_{B_4\cup\substrate} \C(x)\str{u_\epsilon}:\str{u_\epsilon}dx\\
\le &  \cW(B_4,v_4) + 2\Big(\cW(B_4,v_4)\Big)^{1/2}
\Big(\int_{B_5\cup\substrate} \C(x) (\str{u_\epsilon} - \str{v_5}):(\str{u_\epsilon} - \str{v_5})dx\Big)^{1/2}\\
\le & (C+9\epsilon) + 2\sqrt{(C+9\epsilon)\|\C\|_\infty}\,\, 
\|\str{u_\epsilon} - \str{v_5}\|_{L^2} \le C+10\epsilon,
\end{align*}
where in the last inequality we used \eqref{elastic_close}. 
Therefore, by \eqref{ahahaha010101} and again by \eqref{elastic_close}
\begin{align}\label{elastic_epsilon}
\cW(A_\epsilon,u_\epsilon) \le & \cW(B_4,v_4) + 2\sqrt{(C+10\epsilon)\|\C\|_\infty}\,\, 
\|\str{u_\epsilon} - \str{v_5}\|_{L^2} \le \cW(B_4,v_4) +\epsilon.
\end{align}
Now combining \eqref{surface_epsilon} and \eqref{elastic_epsilon} as well as using $|B_5|=|A_\epsilon|$ we get 
\begin{equation}\label{defini_beps}
\tilde \cF^{\lambda}(B_4,v_4) \ge  \cF^{\lambda}(A_\epsilon,u_\epsilon) - 3\epsilon.
\end{equation}
Since $(A_\epsilon,u_\epsilon)\in\admissible_{n_\epsilon},$ by \eqref{defini_b1v1}, \eqref{defini_b2v2}, \eqref{defini_b3v3}, \eqref{defini_b4v4} and \eqref{defini_beps} we get 
$$
\inf\limits_{(A,u)\in\tilde\admissible} \,\tilde\cF(A,u) + 12 \epsilon \ge \cF(A_\epsilon,u_\epsilon),
$$
and \eqref{shshshshka} follows.
\end{proof}
 
 Proposition \ref{prop:min_extend} implies that the configuration $(A,u)$ given by Proposition \ref{prop:lsc_surface} is a volume-constraint minimizer of $\tilde \cF$ in $\tilde\admissible.$

\begin{proposition}\label{prop:minimizers_tilde_cf}
Let $(A,u)\in\tilde\admissible$ be given by Proposition \ref{prop:lsc_surface}. Then 
$(\Int{A},u)$ is a minimizer of $\tilde \cF$ in $\tilde \admissible$ under the volume constraint $|A|=\fm.$ Moreover, let $\lambda_0$ be as in Proposition \ref{prop:min_extend} and let $\bu(\tilde A,\tilde u)\ba\in \tilde\admissible$ be any volume-constraint minimizer of $\tilde\cF.$ Then $\bu(\tilde A,\tilde u)\ba$ is a minimizer of $\tilde \cF^\lambda$ for all $\lambda\ge\lambda_0,$ where 
\begin{equation}\label{def_cf_tilde_lambda}
\tilde \cF^{\lambda}(B,v):=\tilde \cF(B,v) + \lambda\big||B| - \fm\big|,\quad (B,v) \in\tilde \admissible,\quad\lambda>0. 
\end{equation}
\end{proposition}
 
\begin{proof}
Note that since $|\Int{A}\Delta A|=0$ and $(\Int{A},u)\in\tilde\admissible,$ by  Propositions \ref{prop:def_of_u}, \ref{prop:lsc_surface} and \ref{prop:min_extend}  
$$ 
\tilde \cF(\Int{A},u) = \inf\limits_{(B,v)\in\admissible}\, \cF(B,v) = \inf\limits_{(B,v)\in\tilde \admissible} \tilde \cF(B,v) = \inf\limits_{(B,v)\in \admissible} \tilde  \cF^{\lambda}(B,v)
$$ 
for all $\lambda\ge\lambda_0.$
Thus, $(\Int{A},u)$ is a minimizer of both $\tilde \cF$ and $\tilde\cF^{\lambda_0}.$ The same is true for every minimizer $(B,v)$ of $\tilde \cF.$
\end{proof}

\begin{theorem}[\textbf{Density estimates for minimizers of $\tilde \cF^\lambda$}]\label{teo:density_estimates_tilde}
Given $\lambda>0,$ let $(A,u)\in\tilde \admissible$  be any minimizer of
$\tilde \cF^{\lambda}(\cdot,\cdot)$ in $\tilde \admissible$ and let $\xi\in\R^2$ be such that for the function 
$$
\tilde u:=u\chi_{A\cup\substrate} + \xi\chi_{\Omega\setminus A}
$$
one has $\Omega\cap\p^*A \subset J_{\tilde u}.$  Then for any $x\in \Omega$ and $r\in(0,\dist(x,\p \Omega)),$
\begin{equation}\label{min_seq_density_up_Ju}
\frac{\cH^1( Q_r(x)\cap J_{\tilde u})}{r} \le \frac{16c_2 + 4\lambda}{c_1}. 
\end{equation}
Moreover, there exist $\varsigma^*=\varsigma^*\in(0,1)$ and $R^*>0$ not depending on $(A,u)$ with the following property. If $x\in\Omega$ belongs to the closure $J^c_{\tilde u}$ of the set  $\{y\in \Omega\cap J_{\tilde u}:\,\theta_*(J_{\tilde u},y)>0\},$  
then 
\begin{equation}\label{min_seq_density_low_Ju}
\frac{\cH^1( Q_r(x)\cap J_{\tilde u} )}{r}\ge \varsigma^* 
\end{equation}
for any square $ Q_r(x)\strictlyincluded\Omega$ with 
 $r\in(0,\min\{R^*,\dist(x,\p \Omega)\}),$ and if $x\in\Omega$ belongs to the closure $S_u^c$ of $\{x\in S_u^A:\,\theta_*(S_u^A,x)>0\},$ then 
\begin{equation}\label{min_seq_density_low_Su}
\frac{\cH^1( Q_r(x)\cap S_u^A)}{r}\ge \varsigma^* 
\end{equation}
for any $r\in (0,\min\{R^*,\dist(x,\cl{\p^*A})\}.$
In particular, 
\begin{equation}\label{essential_closed_things} 
\cH^1(\Omega\cap (J_{\tilde u}^c\setminus J_{\tilde u})=\cH^1(\Int{A^{(1)}}\cap (S_u^c\setminus S_u^A)=0.  
\end{equation}
\end{theorem}

\begin{proof}[Proof of Theorem \ref{teo:density_estimates_tilde}]
As in Remark \ref{rem:passage_toU0_teng0} $(A,u)$ is a minimizer of $\tilde \cF^{\lambda }$ if and only if 
$(A,u+u_0)$ minimizes the $ \widehat{\tilde  \cF^{\lambda }}(\cdot):=\tilde \cF^{\lambda }(\cdot -u_0).$ Thus, we can introduce the following localized version of $\tilde \cF$ in open subsets $\openset$ of $\Omega$ which does not see the substrate:
$$ 
\tilde \cF(B,v;\openset):=\tilde \cS(B;\openset) + \cW(B,v; \openset) 
$$ 
where
$$
\tilde \cS(B,v;\openset):= \int_{\openset \cap \p^*B} \varphi(y,\nu_B)d\cH^1
+ 2\int_{\openset \cap B^{(1)}\cap \p B\cap S_v} \varphi(y,\nu_B)d\cH^1, 
$$
the $\cW(\cdot;\openset)$ is given as in \eqref{local_cF}
and $S_v^A$ is defined as in \eqref{defini_Su}.
Then the minimality of $(A,u)$ implies that $(A,u+u_0)$ is a quasi-minimizer of $\tilde \cF(\cdot;\openset)$  
in $\openset,$ namely, 
\begin{equation*} 
\tilde \cF(A,u +u_0;\openset) \le \tilde \cF(B,v;\openset) + \lambda_0|A\Delta B|
\end{equation*}
whenever $(B,v)\in\tilde \admissible$ with $A\Delta B\strictlyincluded \openset$ and $\supp(u+u_0-v)\strictlyincluded \openset.$ 
Now the proof of the existence of $\varsigma^*$ and $R^*$ satisfying \eqref{min_seq_density_up_Ju} and \eqref{min_seq_density_low_Ju} runs along the same lines of the proof of Theorem \ref{teo:density_estimates} for $m=\infty,$ therefore, we do not repeat it here. Note that $\varsigma^*$ and $R^*$ depend only on $c_i$ and $\lambda.$

Let $A_{\circ}:=\Int{A^{(1)}}.$ We claim that 
$$ 
\p A_{\circ} = \overline{\p^*A}. 
$$ 
Indeed, note that $A^{(1)}\setminus A_{\circ}\subset \p A^{(1)} =\cl{\p^*A},$ where in the equality we used $\cl{\p^*A} = \cl{\p^*A^{(1)}}=\p A^{(1)}$ see e.g., \cite[Eq. 15.3]{Ma:2012}. Thus, $A_{\circ}$ is also equivalent to $A,$ and hence, $\p^*A_{\circ}=\p^*A=\p^*A^{(1)}.$ In particular, $\p A^{(1)}=\overline{\p^*A_{\circ}}\subset \p A_{\circ}.$ On the other hand, assume that there exists $x\in \p A_{\circ}\setminus \p A^{(1)}.$ Since $\p A^{(1)}$ is closed, there exists $r>0$ such that $\cl{B_r(x)}\cap \p A^{(1)}=\emptyset.$ Hence, either $B_r(x)\subset \Int{A^{(1)}}=A_{\circ}$ or $\overline{B_r(x)}\cap \overline{A^{(1)}}=\emptyset.$ Since $A_{\circ}$ is open and $x\in\p A_{\circ},$ the inclusion $B_r(x)\subset A_{\circ}$ is not possible. On the other hand, since $\cl{A_{\circ}}\subset \cl{A^{(1)}}$ and $x\in \p A_{\circ},$ the relation  $\overline{B_r(x)}\cap \overline{A^{(1)}}=\emptyset$ is also not possible. Thus, $\p A_{\circ}\subseteq\p A^{(1)}.$

To prove \eqref{min_seq_density_low_Su}  we fix $\Omega'\strictlyincluded\Omega.$ 
We claim that $\tilde u\big|_{A_{\circ}}$  is a minimizer of Griffith functional $\cG:GSBD^2(\Ins{A_{\circ}}) \to\R,$
$$
\cG(v):=\int_{A_{\circ}\cap J_{v}} \varphi(x,\nu_{J_v}) d\cH^1 + \int_{A_{\circ}} \C(x)\str{v}:\str{v}dx 
$$ 
with Dirichlet boundary condition $v=\tilde u=u$ in  $A_{\circ}\setminus\Omega'.$  Indeed, for every $v\in GSBD^2(A_{\circ})$ with $\tilde u=v$ in $A_{\circ}\setminus \Omega'$ we define $B:=A_{\circ}\setminus \cl{J_v}.$ Then $(B,v)\in\tilde\admissible$ and by the minimality of $(A,u)$ 
$$
\cG(u) - \cG(v) =\tilde \cF(A,u) - \tilde \cF(B,v)\le 0.
$$
Since $S_v^B = J_{\tilde u\big|_{A_{\circ}}}$ up to a $\cH^1$-negligible set,  \eqref{min_seq_density_low_Su} directly follows from the density estimates for the jump set of  Griffith minimizers (see e.g. \cite{ChC:2019_arxiv}) with possibly smaller $\varsigma^*\in(0,1)$ and $R^*>0.$

Finally, we prove \eqref{essential_closed_things} only for $S_u^A,$ the other being similar. Let $\Gamma:=\{x\in S_u^A:\, \theta_*(S_u^A,x)>0\}.$ Note that $S_u^c=\overline \Gamma.$

We claim that 
\begin{equation}\label{essential_closed_gamma}
\cH^1(A_{\circ}\cap (\overline\Gamma \setminus \Gamma))=0. 
\end{equation}
Indeed, let $\mu:=\cH^1\res \Gamma.$  Then $\mu(\overline{\Gamma}\setminus \Gamma)=0.$ By the regularity of $\mu,$ for every $\epsilon>0$ there exists an open set $U\subset\R^2$ such that $L:=A_{\circ}\cap (\overline{\Gamma}\setminus \Gamma) \subset U$ and $\mu(U)=\cH^1(U \cap \Gamma)<\epsilon.$ Note that $\overline\Gamma\subset \cl{\{y\in\Omega\cap J_{\tilde u}:\,\theta_*(J_{\tilde u},x)>0\}},$ where $\tilde u$ is given by Theorem \ref{teo:density_estimates_tilde}. Hence, for \eqref{min_seq_density_low_Ju} holds for all points of $\overline\Gamma.$ By the definition of the closure, and Vitali Covering Lemma we can find at most countable pairwise disjoint family $\{\cl{B_{r_i}(x_i)}\}_i$ of closed balls $\cl{B_{r_i}(x_i)}$ with $x_i\in A_{\circ}\cap \Gamma,$ $r_i\le \min\{R^*,\epsilon,\dist(x,\p \overline A)\}$  such that $A_{\circ}\cap (\overline\Gamma\setminus \Gamma) \subset \cup_i \cl{B_{5r_i}(x_i)}.$ Without loss of generality we may assume that $B_{r_i}(x_i)\subset U.$ Since $Q_{r_i/\sqrt{2}}(x_i)\subset  B_{r_i}(x_i)\subset Q_{r_i}(x_i),$ from the definition of Hausdorff premeasure, \eqref{min_seq_density_low_Ju} and disjointness of $\{B_{r_i}(x_i)\}$ as well as the choice of $U$ we obtain
$$
\begin{aligned}
\cH_{10\epsilon}(A_{\circ}\cap (\overline\Gamma \setminus \Gamma))\le  & \sum\limits_{i\ge1} 2\pi (5r_i) \le 
\frac{10\pi\sqrt{2}}{\varsigma^*} \sum\limits_{i\le 1} \cH^1(Q_{r_i/\sqrt{2}}(x_i) \cap \Gamma) \\
=
& \frac{10\pi\sqrt{2}}{\varsigma^*}\, \cH^1\big(\cup_i Q_{r_i/\sqrt{2}}(x_i) \cap \Gamma\big) \le \frac{10\pi\sqrt{2}}{\varsigma^*}\, \cH^1\big(U \cap \Gamma\big) <
\frac{10\pi\sqrt{2}\epsilon}{\varsigma^*}.
\end{aligned}
$$
Now letting $\epsilon\to0$ we get \eqref{essential_closed_gamma}. 
\end{proof}

 In the following proposition we construct a ``regular'' minimizer of $\cF$ starting from a minimizer of $\tilde \cF$ in $\tilde\admissible.$

\begin{proposition}\label{prop:constructed_min_F}
Given $\lambda>0,$ let $(A,u)\in\tilde \admissible$ be any minimizer of $\tilde \cF^\lambda.$
Define 
$$
 A':=\Int{A^{(1)}}\setminus\overline \Gamma, 
$$
where $\Gamma:=\{x\in S_u^A:\,\theta_*(S_u^A,x)>0\},$ and, with a slight abuse of notation,  consider $u$ as defined in $A'\cup \substrate$ \emph{(}and so, also on  the $\mathcal{L}^2$-negligible set $A'\setminus \Int{A}$\emph{)}. %
Then $(A',u)\in\admissible$ is such that $\tilde\cF(A,u) = \cF(A',u)$ and satisfy the following assertions: 
\begin{itemize}
\item[(1)] $A'$ is open, $\theta_*(S_{u}^{A'},x)>0$ for all $x\in S_{u}^{A'},$ $|A\Delta A'|=0$ and $u\chi_{A\cup \substrate} = u\chi_{A'\cup\substrate}$ a.e.\ in $\Omega\cup \substrate.$

\item[(2)] The closure of $A'^{(1)}\cap\p A'$ coincide with $\cl{S_{u}^{A'}}$ and $\cH^1(\cl{S_{u}^{A'}}\setminus S_{u}^{A'})=0.$

\item[(3)] Let $\varsigma*$ and $R^*$ be given by Theorem \ref{teo:density_estimates_tilde}. Then 
\begin{equation}\label{eq:density_ssss}
\frac{\cH^1(Q_r(x)\cap \p A')}{r} \le \frac{16c_2 + 4\lambda_0}{c_1} 
\end{equation}
for every square $Q_r(x)\subset\Omega$ 
and 
\begin{equation}\label{eq:density_tttt}
\frac{\cH^1(Q_r(x)\cap\p A')}{r} \ge \varsigma*  
\end{equation}
for every $Q_r(x)\subset\Omega$ with for any $x\in \p A'$ and  $r\in(0,R^*).$
\end{itemize}

\end{proposition}

\begin{proof}
 Note that by definition $A'$ is open and $|A'\Delta A|=0.$ Moreover, $S_{u}^{A'}\subset \Gamma,$ and by \eqref{min_seq_density_low_Su} 
all points of $\Omega\cap \overline\Gamma$ satisfy uniform lower density estimates, hence, $\theta_*(S_{u}^{A'},x)>0$ for any $x\in S_{u}^{A'}.$

We claim that $A'\in\fA.$ Indeed, let $\tilde u$ be given as in Theorem \ref{teo:density_estimates_tilde}.  By definition 
\begin{equation}\label{boundary_newww}
\Omega\cap J_{\tilde u}^c = \Omega\cap \p A'\quad\text{and}\quad  \p A' \subset J_{\tilde u}^c\cup \Sigma, 
\end{equation}
where $J_{\tilde u}^c$ is the closure of the set $\{x\in J_{\tilde u}:\,\theta_*(J_{\tilde u},x)>0\}.$ Since $J_{\tilde u}$ is $\cH^1$-rectifiable, so is $J_{\tilde u}^c$ in view of \eqref{essential_closed_things}. 
Therefore, $\p A'$ is $\cH^1$-rectifiable, i.e., $A'\in\fA.$ 
Note that by construction $\cH^1(A'\cap J_{\tilde u})=0$ hence, by Proposition \ref{lem:gcbd_with_zero_jumps} $\tilde u\in H_\loc^1(A')$ and, since $u=\tilde u$ a.e.\ in $A'$ it follows that $u\in H_\loc^1(A').$ 

Since $|A\Delta A'|=0$ and $u=u$ a.e.\ in $A',$ it follows that 
$$
\cW(A,u) =\cW(A',u).
$$
Moreover, by the definition of $\Gamma$ and $S_u^A,$
$$
|\cS(A',u) - \tilde \cS(A,u)| = \int_{\Int{A^{(1)}}\cap (\overline \Gamma\setminus S_u^A)} \varphi(x,\nu_\Gamma)d\cH^1 \le c_2\cH^1(\Int{A^{(1)}}\cap (\overline\Gamma\setminus \Gamma)) =0, 
$$
where in the last equality we used \eqref{essential_closed_things}. 
Finally, \eqref{eq:density_ssss} and \eqref{eq:density_tttt} follows from \eqref{boundary_newww} and density estimates of \blue Theorem \ref{teo:density_estimates_tilde}. \black 
\end{proof}

Now we are ready to prove the existence of global minimizers of $\cF.$

\begin{proof}[Proof of Theorem \ref{teo:global_existence}]
First we prove the assertion for $\cG=\cF.$

Let $(A_m,u_m)\in\admissible_m$ be a minimizer of $\cF$ satisfying the volume constraint $|A_m|=\fm$ and  
let $(A_{m_h},u_{m_h}),$ and $A$ and $u$ 
be as in Proposition \ref{prop:lsc_surface}.
By \eqref{zur_tenglik}, \eqref{lsc_surface_energy} and  \eqref{lsc_elastic1} we have
$$
\inf\limits_{(B,v)\in\admissible,\,|B|=\fm} \cF(B,v) = \lim\limits_{h\to+\infty} \cF(A_{m_h},u_{m_h}) \ge \tilde \cF(\Int{A},u). 
$$
Since $ |\Int{A}|=\fm,$ by Propositions \ref{prop:min_extend} and \ref{prop:minimizers_tilde_cf} 
\begin{equation}\label{min_tengliklar}
\inf\limits_{(B,v)\in\admissible,\,|B|=\fm} \cF(B,v)= 
\inf\limits_{(B,v)\in\tilde \admissible,\,|B|=\fm} \tilde\cF^{\lambda_0}(B,v)=\tilde\cF^{\lambda_0}(\Int{A},u)=
\tilde\cF(\Int{A},u),
\end{equation}
hence, $(\Int{A},u)$ is a minimizer of $\tilde \cF^{\lambda_0}$ in $\tilde\admissible.$ Then by Proposition \ref{prop:constructed_min_F} there exists $(A',u)\in\admissible$ such that 
$$
\tilde \cF(\Int{A},u) =\cF(A',u), 
$$
and hence, in view of \eqref{min_tengliklar}, $(A',u)$ is a solution to \eqref{min_prob_globals}. 

The proof of the second assertion (i.e., the existence of $\lambda_1$ for which the set of minimizers in $\admissible$ of both $\cF$ and $\cF^\lambda$ coincide for all $\lambda\ge\lambda_1$) can be done using the first one and also following the arguments of \cite[Theorem 1.1]{EF:2011} and \cite[Proposition A.6]{HP:2019}. Without loss of generality we assume that $\lambda_1\ge\lambda_0,$ where $\lambda_0$ is given by Proposition \ref{prop:min_extend}.

Now we prove Theorem \ref{teo:global_existence} for $\cG=\tilde\cF.$ We have already shown above that the configuration $(\Int{A},u)$ given by Proposition \ref{prop:lsc_surface} solves the minimum problem \eqref{min_prob_globals} with $\cG=\tilde\cF.$  In view of \eqref{equiv_min_problem} every volume-constraint  minimizer of $\tilde \cF$  also minimizer of $\tilde \cF^\lambda$ for all $\lambda\ge\lambda_1.$ To prove the converse assertion, we fix any minimizer $(A,u)\in\tilde\admissible$ of $\tilde \cF^\lambda$ for $\lambda\ge\lambda_1.$ By Proposition \ref{prop:constructed_min_F} there exists $(A',u)\in \admissible$ such that $|A'|=|A|$ and $\cF(A',u)=\tilde\cF(A,u).$ By the first part of the proof and \eqref{equiv_min_problem} we know that 
$$
\inf\limits_{(B,v)\in\admissible}\cF^\lambda(B,v)  = \inf\limits_{(B,v)\in\admissible,|B|=\fm}\cF(B,v) = \inf\limits_{(B,v)\in\tilde \admissible}\tilde \cF^\lambda(B,v) =\cF^\lambda(A',u).
$$
Hence, $(A',u)$ is the minimizer of $\cF^\lambda.$ Since $\lambda\ge\lambda_1$ according to the first part of the proof, $ |A'|=\fm.$ Hence, $|A|=\fm$ and $(A,u)$ minimizer of \eqref{min_prob_globals}. 
\end{proof}

We are ready now to study the properties of the minimizers of $\cF$ in $\admissible$ provided by Theorem \ref{teo:global_existence}.

\begin{proof}[Proof of Theorem \ref{teo:regularity_of_minimizers}]
First we properties (1)-(4) the assertion for $\cG=\tilde \cF.$

Consider any solution $(A,u)\in\tilde\admissible$ of \eqref{min_prob_globals}. 
By Proposition \ref{prop:constructed_min_F} there exists a $(A',u)\in\admissible$ with $A'$ defined as in \eqref{def_a_prime},  such that the properties (1)-(4) hold except the conditions $\cH^1(\p A\Delta\p A')=0$ and $\cH^1(S_u^A\Delta S_{u}^{A'})=0$ of (1). To prove these two equations it is enough to observe that 
$$
0=|\cF(A,u) - \cF(A',u)| =2\int_{A^{(1)}\cap (\p A\Delta \p A')} \varphi(x,\nu_A)d\cH^1
$$
and
$$
0=|\tilde \cF(A,u) - \tilde \cF(A',u)| =2\int_{A^{(1)}\cap  (S_u^A\Delta S_{u}^{A'})} \varphi(x,\nu_A)d\cH^1.
$$

Now we assume that $\cG=\cF$ and let $(A,u)\in\admissible$ be a solution to \eqref{min_prob_globals}. Since $(A,u)\in\tilde\admissible,$ by Proposition \ref{prop:min_extend}
$$
\inf\limits_{(B,v)\in\tilde\admissible,|B|=\fm} \tilde\cF(B,v) = \cF(A,u) \ge \tilde\cF(A,u).
$$
Therefore,  $(A,u)$ is also a volume-constraint minimizer of $\tilde  \cF.$ Thus, applying first part of the proof we establish that $(A',u)\in\admissible$ satisfies (1)-(4). 

Finally, notice that if $E\subset A'$ is a connected component of $A'$ with $\cH^1(\p E\cap \Sigma\setminus J_u)=0,$ then   for $(A',v)$ with $v=u\chi_{(A\cup \substrate)\setminus E} + (u_0 +a)\chi_E,$ where $a$ is any rigid displacement,
we have 
$$
\cS(A',u) \ge \cS(A',v)
$$
and 
\begin{equation}\label{smaller_somssj}
\cW(A',u) \ge \cW(A',v), 
\end{equation}
where in \eqref{smaller_somssj} equality holds if and only of $u=u_0+a$ in $E.$ Therefore, by the minimality of $(A',u)$ it follows that $u=u_0+a$ in $E.$  It remains to prove 
\begin{equation}\label{smallest_element_sss}
|E|\ge 4\pi \,\Big(\frac{c_1}{\lambda_0}\Big)^2. 
\end{equation}
Consider the competitor $(A'\setminus E,u)\in\admissible.$ By minimality and Theorem \ref{teo:global_existence}, $\cF^{\lambda_1}(A',u) \le \cF^{\lambda_1}(A'\setminus E,u),$ so that by \eqref{smaller_somssj} and the additivity of the surface energy, 
$
\cS(E,u) \le \lambda_1 |E|.
$
Then by \eqref{finsler_norm} and the isoperimetric inequality in $\R^2$ 
$$
\lambda_1 |E| \ge c_1 \cH^1(\p E) \ge c_1\sqrt{4\pi} |E|^{1/2}.
$$
Hence, \eqref{smallest_element_sss} follows. 
\end{proof} 

\appendix

\section{}\label{appendix}
 
We include in this section auxiliary results used in the paper for the convenience of the Reader. We begin by a property satisfied by the free-crystal regions in $\mathcal{A}$ and $\widetilde{\mathcal{A}}$.

\begin{proposition} \label{prop:adm_sets_have_finite_per}
\red Let $A\subset\R^2$ be a bounded \ba  $\mathcal{L}^2$-measurable set with $\cH^1(\p A)<+\infty.$  Then $A$ is a set of finite perimeter in $\R^2$.   
\end{proposition}

\begin{proof}
\red Since $A\Delta \Int{\cl{A}}\subset  \cl{A}\setminus \Int{A} = \p A,$ we have 
$
|A\Delta \Int{\cl{A}}| \le |\p A| = 0,
$
and hence, it suffices to prove that the open set $E:=\Int{\cl{A}}$ has finite perimeter in $\R^2.$ Note that by construction, $\p E\subset \p A$ and $\cH^1(\p E)\le \cH^1(\p A)<+\infty.$

We divide the proof of $E\in BV(\R^2,\{0,1\})$ into three steps. 

{\it Step 1.} We claim that if $E$ is simply connected, then $E\in BV(\R^2;\{0,1\})$. Indeed, in this case $\p E$ is  a \ba connected compact set with $\cH^1(\p E) \le \cH^1(\p A)<+\infty$ and by \cite[Lemma 3.12]{Fa:1985} it contains a closed curve $\Gamma$ enclosing $\cl{E}$. Since $\cH^1(\Gamma)<+\infty,$ it is rectifiable in the sense of \cite[Section 3.2]{Fa:1985}: its length $\cH^1(\Gamma)$ is well-approximated by the length of closed polygonal curves $\pi_k$ whose vertices lie on $\Gamma,$ i.e.,  $\cH^1(\pi_k)\to \cH^1( \Gamma )$.  Let $E_k$ be the set enclosed by $\pi_k$ and observe that $\pi_k\overset{K}{\to}  \Gamma $. Since $E_k$ are Lipschitz sets, they are sets of finite perimeter and $P(E_k) =\cH^1(\pi_k) \le \cH^1(\Gamma)+1$ for large $k.$ Since $E$ is open, for every $x\in E$ there exists a ball $B_r(x)\subset E$ and by the Kuratowski convergence of $\pi_k$ to $\Gamma$, it follows that $B_r(x)\subset E_k$ for large $k$, and hence $\chi_{E_k}(x)=\chi_{E}(x)=1$. Similarly,  $\chi_{E_k}(x)=\chi_{E}(x)=0$ for every $x\in\R^2\setminus \cl{E}$ provided $k$ is large enough. Therefore, $\chi_{E_k} \to \chi_E$ a.e.\ in $\R^2$ and hence, $E_k\to E$ in $L^1(\R^2).$ Now  by  the $L^1$-lower semicontinuity of perimeter  (see \cite[Proposition 12.15]{Ma:2012}), $E$ is a set of finite perimeter.
\smallskip

\red 
{\it Step 2.} We claim that if $E$ is connected, then $E\in BV(\R^2;\{0,1\}).$ \ba Indeed, let $E'$ be the smallest simply connected open set containing $E$ (basically, $E'$ is contructed by filling in all ``holes'' in $E$) and let
$$
F:= E'\setminus \cl{E}
$$
be the union of all holes. Since $\p E'\subset \p E$ and $\cH^1(\p E)\le \cH^1(\p A) <+\infty,$ by Step 1 $E'\in BV(\R^2;\{0,1\}).$ Observing $E=E'\setminus \cl{F},$ to conclude this step it is enough to prove that $F$ has finite perimeter. Since every open set in $\R^2$ is a union of at most countably many connected components\footnote{This property easily follows by fact that we can always choose  in each connected component a different point with rational coordinates}, \red we have $F = \cup_j F_j,$ where $\{F_j\}$ are  open, connected and $F_i\cap F_j=\emptyset$ for $i\ne j.$ \ba Since $E$ is connected, each $F_j$ is simply connected, and hence, by Step \red 1 \ba $F_j\in BV(\R^2;\{0,1\}).$ \red Moreover, the set $\p F_i\cap \p F_j,$ $i\ne j,$ can have at most one point. Indeed, otherwise, by the definition of $F$ and the connectedness of $E$ we could find a curve $\gamma\subset \p F_i\cap \p F_j\cap \p E$ with $\cH^1(\gamma)>0,$  which contradicts the equality $E = \Int{\cl{E}}.$  Therefore, observing $\p F=\bigcup\p F_j\subset \p E,$ we obtain 
$$
\sum\limits_{j} P(F_j) \le \sum\limits_{j} \cH^1(\p F_j) = \cH^1\Big(\bigcup_j \p F_j\Big) = \cH^1(\p F)\le \cH^1(\p E) < +\infty.
$$
Thus, $F=\bigcup_jF_j$ has finite perimeter in $\R^2$. 

{\it Step 3.} Now we prove that $E\in BV(\R^2;\{0,1\})$ (without assuming any extra connectedness assumption). Let $\{E_j\}$ be the family of  connected components of $E.$ Since $\cH^1(\p E_j)\le \cH^1(\p E)<+\infty,$ by Step 2 $E_j\in BV(\R^2;\{0,1\}).$ Therefore, since  $\p E = \cup_j\p E_j$ we obtain that 
$$
\sum_{j} P(E_j) \le \sum_{j} \cH^1(\p E_j) \le \cH^1\Big(\bigcup_j \p E_j\Big) + \sum\limits_{i < j} \cH^1(\p E_i\cap \p E_j) \le 2 \cH^1\Big(\bigcup_j \p E_j\Big) = 2\cH^1(\p E),
$$
and hence, by the finiteness of $\cH^1(\p E),$ the set $E=\cup_j E_j$ has finite perimeter in $\R^2.$
\end{proof}
 
The following proposition, which is based on \cite[Proposition 2.16]{Ma:2012},  is used throughout the paper.

\begin{proposition}\label{prop:maggi_foliation}
Let $K\subset\R^2$ be such that $\cH^1(K)<+\infty$ and let $\{E_t\}_{t\in \varUpsilon}$ be a family of sets parametrized by $t\in\varUpsilon$ such that 
\begin{equation}\label{essential_disjointness}
\cH^1(K\cap E_t\cap E_s)=0 
\end{equation}
and  $\cH^1(K\cap E_t)>0.$ Then $\varUpsilon$ is at most countable.
\end{proposition}

\begin{proof}
The proof runs along the lines of the proof of \cite[Proposition 2.16]{Ma:2012}.
For $j\in\N$ let $\varUpsilon_j\subset \varUpsilon$ be the set of all $t\in \varUpsilon$ such that $\cH^1(K\cap E_t)>\frac{1}{j}.$ Then by \eqref{essential_disjointness}  $\varUpsilon_j$ cannot contain more than $j\cH^1(K)$ elements. Since $\varUpsilon=\cup_j \varUpsilon_j,$ the set $\varUpsilon$ is at most countable.
\end{proof} 

We finally state a regularity property of $GSBD$ functions with $\cH^{d-1}$-negligible jump.

\begin{proposition}\label{lem:gcbd_with_zero_jumps}
Let $U\subset\R^d$ be a connected bounded open set and $u\in GSBD^2(U)$ be such that $\cH^{d-1}(J_u)=0.$ Then $u\in H_\loc^1(U).$
\end{proposition}

\begin{proof}
Indeed, for $r>0$ let $Q:=x_0+(-r,r)^d\subset U$ be any cube centered at $x\in U$ and let $0<\theta''<\theta'<1.$ For shortness, write $Q':=x_0+(-\theta' r,\theta'r)^d$ and $Q'':=x_0+(-\theta''r,\theta''r)^d.$ By \cite[Proposition 3.1 (1)]{CCI:2019.jmpa} (see also \cite[Theorem 1.1]{CCF:2016}) there exists a $\mathcal{L}^2$-measurable set $\omega\subset Q'$ and a rigid displacement $a:\bu\R^d\ba\to\bu\R^d\ba$ such that $|\omega|\le c_* r\cH^{d-1}(J_u)=0$ and
$$
\int_{Q'} |u -a|^{\frac{2d}{d-1}}dx = \int_{Q'\setminus\omega} |u -a|^{\frac{2d}{d-1}}dx\le c_*r^2 \Big(\int_Q |\str{u}|^2\Big)^{\red \frac{d}{d-1}},
$$
where $c_*$ depends only on $\red d.$ Hence, $u\in L_\loc^{\frac{2d}{d-1}}(Q).$ Next, fix any mollifier $\rho_1\in C^\infty(B_r(0))$ with $\rho_\epsilon\in C_c^\infty(B_{\red(\theta' - \theta'')\epsilon}),$ where $\rho_\epsilon(x):=\rho_1(x/\epsilon),$ $\epsilon\in(0,r).$
By \cite[Proposition 3.1]{CCI:2019.jmpa} there exists $\overline p>0$ depending on $n$ and $\epsilon$ such that
$$
\int_{Q''} |e(u*\rho_\epsilon) - e(u)*\rho_\epsilon|^2 dx \le c \Big(\frac{\cH^{d-1}(J_u)}{r^{d-1}}\Big)^{\overline p} \int_Q |\str{u}|^2dx =0,
$$
where $c$ depends on $n,$ $\rho_1$ and $\epsilon.$
Hence,
\begin{equation}\label{aeconvergence}
e(u*\rho_\epsilon) = e(u)*\rho_\epsilon \quad \text{a.e. in $Q''.$}
\end{equation}
Recall that $u*\rho_\epsilon \in C^\infty(Q'').$ Since $\str{u}\in L^2(Q),$ $\str{u}*\rho_\epsilon \in C^\infty(Q'')\cap L^2(Q'')$ in particular, $\str{u*\rho_\epsilon} \in C^\infty(Q'')\cap L^2(Q'').$ By Poincar\'e-Korn inequality $u*\rho_\epsilon\in H^1(Q'').$ Since $e(u)*\rho_\epsilon \to e(u)$ in $L^2(Q'')$ as $\epsilon\to0,$ in view of \eqref{aeconvergence} there exists $\epsilon_0>0$ such that
$$
\|\str{u*\rho_\epsilon}\|_{L^2(Q'')} \le \|\str{u}\|_{L^2(Q'')} +1\quad \text{for all $\epsilon\in(0,\epsilon_0).$}
$$
Moreover, by Poincar\'e-Korn inequality for any $\epsilon\in(0,\epsilon_0)$ there exists a rigid displacement $a_\epsilon$ such that
$$
\| u*\rho_\epsilon - a_\epsilon \|_{H^1(Q'')} \le C \|\str{u*\rho _\epsilon}\|_{L^2(Q'')} \le C(\|\str{u}\|_{L^2(Q'')} + 1),
$$
where $C$ is the Poincar\'e-Korn constant for a cube.
Thus, the family $\{u*\rho_\epsilon\}_\epsilon$ is uniformly bounded in $H^1(Q'')$. Since $u*\rho_\epsilon \to u$ in $L^2(Q''),$ there exists a rigid displacement $a$ such that $a_\epsilon \to a$ in $L^2(Q'').$ Then $u*\rho_\epsilon - a_\epsilon $ weakly converges to $u-a$ in $H^1(Q''),$ i.e., $u-a\in H^1(Q'').$ Since $a$ is linear and $\theta''$ is arbitrary, $u\in H_\loc^1(Q).$ Now covering $U$ with finitely many cubes of edgelength $2r$ we get $u\in H_\loc^1(U).$
\end{proof}
\black

\section*{Acknowledgments} 
Sh. Kholmatov acknowledges support from the Austrian Science Fund (FWF) 
projects M~2571 and P 33716. P. Piovano acknowledges the support from the Austrian Science Fund (FWF) projects P 29681 and TAI 293, from the Vienna Science and Technology Fund (WWTF) together with the City of Vienna and Berndorf Privatstiftung through Project MA16-005, and from BMBWF through the OeAD-WTZ project HR 08/2020.  \bu Furthermore, P. Piovano acknowledges the support obtained by the Italian Ministry of University and Research  (MUR) through  the PRIN Project ``Partial differential equations and related geometric-functional inequalities'', is member of the Italian ``Gruppo Nazionale per l'Analisi Matematica, la Probabilit\`a e le loro Applicazioni'' (GNAMPA-INdAM) and has received funding from the GNAMPA 2022 project CUP: E55F22000270001.  Finally, \ba P. Piovano is grateful for the support received as \emph{Visiting Professor and Excellence Chair} at the Okinawa Institute of Science and Technology (OIST), Japan.

\end{document}